\documentclass{amsart}
\usepackage{amsmath,amsthm,amsfonts,amssymb,amscd,amsbsy}
\usepackage{caption,enumerate}
\usepackage{dsfont,lscape}
\usepackage[all]{xy}
\usepackage{hyperref}

\hypersetup{
    pdftoolbar=true,
    pdfmenubar=true,
    pdffitwindow=false,
    pdfstartview={FitH},
    pdftitle={The first eigenvalue of a homogeneous CROSS},
    pdfauthor={R. G. Bettiol, E. A. Lauret, P. Piccione},
    pdfsubject={},
    pdfkeywords={}
    pdfnewwindow=true,
    colorlinks=true, 
    linkcolor=blue,
    citecolor=blue,
    urlcolor=black,
}

\newcommand{\Spec}{\operatorname{Spec}}
\newcommand{\Met}{\operatorname{Met}}
\newcommand{\dd}{\mathrm{d}}
\newcommand{\id}{\operatorname{Id}}
\newcommand{\Vol}{\operatorname{Vol}}
\newcommand{\vol}{\operatorname{vol}}

\newcommand{\Ric}{\operatorname{Ric}}
\newcommand{\scal}{\operatorname{scal}}
\newcommand{\diam}{\operatorname{diam}}

\newcommand{\Ss}{\mathds S}
\newcommand{\Hr}{\mathds H}

\newcommand{\Z}{\mathds Z}
\newcommand{\R}{\mathds R}

\newcommand{\C}{\mathds C}
\newcommand{\Ca}{\mathds C\mathrm a}

\newcommand{\GL}{\mathsf{GL}}
\newcommand{\SO}{\mathsf{SO}}
\newcommand{\Ot}{\mathsf{O}}
\newcommand{\SU}{\mathsf{SU}}

\newcommand{\Ut}{\mathsf{U}}
\newcommand{\Sp}{\mathsf{Sp}}
\newcommand{\Spin}{\mathsf{Spin}}
\newcommand{\T}{\mathsf{T}}
\newcommand{\G}{\mathsf{G}}
\newcommand{\K}{\mathsf{K}}
\renewcommand{\H}{\mathsf H}

\newcommand{\gl}{\mathfrak{gl}}
\newcommand{\sll}{\mathfrak{sl}}

\newcommand{\spp}{\mathfrak{sp}}

\DeclareMathOperator{\Ad}{Ad}

\DeclareMathOperator{\spec}{Spec}

\DeclareMathOperator{\Iso}{Iso}
\DeclareMathOperator{\tr}{tr}
\DeclareMathOperator{\diag}{diag}

\newcommand{\Id}{\id}

\newcommand{\mi}{\mathrm{i}}
\newcommand{\mj}{\mathrm{j}}
\newcommand{\mk}{\mathrm{k}}
\DeclareMathOperator{\Cas}{Cas}

\newcommand{\innerdots}{\langle {\cdot},{\cdot}\rangle }
\newcommand{\abc}{(a,b,c)}
\newcommand{\abcs}{(a,b,c,s)}
\newcommand{\Span}{\operatorname{span}}

\newcommand{\g}{\mathrm g}
\newcommand{\gr}{\g_\mathrm{round}}
\newcommand{\gFS}{\g_{\rm FS}}
\newcommand{\calS}{\mathcal B}

\allowdisplaybreaks

\newtheorem{theorem}{Theorem}[]
\newtheorem{lemma}[theorem]{Lemma}
\newtheorem{proposition}[theorem]{Proposition}

\newtheorem{claim}{Claim}

\newtheorem{mainthm}{\sc Theorem}

\newtheorem{maincor}[mainthm]{\sc Corollary}

\theoremstyle{definition}

\theoremstyle{remark}
\newtheorem{remark}[theorem]{Remark}
\newtheorem{example}[theorem]{Example}

\title{The first eigenvalue of a homogeneous CROSS}


\author[R.~G.~Bettiol]{Renato G.~Bettiol}
\address{City University of New York (Lehman College) \newline
\indent Department of Mathematics  \newline
\indent 250 Bedford Park~Blvd W\newline
\indent Bronx, NY, 10468, USA }
\email{r.bettiol@lehman.cuny.edu}

\author[E.~A.~Lauret]{Emilio~A.~Lauret}
\address{
	Universidad Nacional del Sur (UNS) - CONICET\newline
	\indent Instituto de Matem\'atica de Bah\'ia Blanca (INMABB) \newline
	\indent Departamento de Matem\'atica\newline
	\indent Av.\ Alem 1255, Bah\'ia Blanca B8000CPB, Argentina}
\email{emilio.lauret@uns.edu.ar}

\author[P.~Piccione]{Paolo Piccione}
\address{Universidade de S\~ao Paulo \newline
\indent Departamento de Matem\'atica \newline
\indent Rua do Mat\~ao, 1010 \newline
\indent S\~ao Paulo, SP, 05508-090, Brazil}
\email{piccione@ime.usp.br}

\allowdisplaybreaks
\numberwithin{equation}{section}
\numberwithin{theorem}{section}

\date{\today}

\begin{document}

\begin{abstract}
We provide explicit formulae for the first eigenvalue of the Laplace--Beltrami operator on a compact rank one symmetric space (CROSS) endowed with any homogeneous metric. As consequences, we prove that homogeneous metrics on CROSSes are isospectral if and only if they are isometric, and also discuss their stability (or lack thereof) as solutions to the Yamabe problem.
\end{abstract}

\maketitle

\vspace{-0.6cm}

\section{Introduction}

The underlying manifold of many compact symmetric spaces admits families of homogeneous Riemannian metrics that include, but are strictly larger than, their symmetric space metric. For instance, all odd-dimensional spheres $\Ss^n$, $n\geq3$, carry a continuum of pairwise non-isometric homogeneous metrics, and only some among them -- the round metrics -- give $\Ss^n$ the structure of a symmetric space. Surprisingly, despite the extensive literature on the spectrum of the Laplace--Beltrami operator, the computation of its first eigenvalue $\lambda_1(M,\g)$ for \emph{every homogeneous metric} $\g$ on (the underlying manifold of a) compact rank one symmetric space (CROSS) $M$ was, to the best of our knowledge, still incomplete. In this paper, we rectify this by computing $\lambda_1(M,\g)$ in all the remaining cases. For simplicity, we henceforth refer to these metrics $\g$ as \emph{homogeneous metrics on a CROSS}. Out of many possible applications, we focus on two geometrically relevant consequences: the spectral uniqueness of homogeneous metrics on CROSSes, and their classification according to stability in the Yamabe problem.

It is well-known that the complete list of CROSSes consists of $\Ss^n$, $\R P^n$, $\C P^n$, $\Hr P^n$, and $\Ca P^2$, see e.g.~\cite[Ch.~3]{Besse-closed}. Homogeneous metrics on a CROSS were classified by Ziller~\cite{Ziller82}, see also \cite[Ex.~6.16, 6.21]{mybook}. Up to homotheties, in addition to the canonical (symmetric space) metrics, that is, the \emph{round} metric $\gr$  of constant sectional curvature $1$ on $\Ss^n$ and $\R P^n$, and the Fubini--Study metrics $\gFS$ on the projective spaces $\C P^n$, $\Hr P^n$, and $\Ca P^2$, they are as follows:
\begin{enumerate}[(i)]
\item A $1$-parameter family $\mathbf g(t)$ of $\SU(n+1)$-invariant metrics on $\Ss^{2n+1}$;
\item A $3$-parameter family $\mathbf h{(t_1,t_2,t_3)}$ of $\Sp(n+1)$-invariant metrics on $\Ss^{4n+3}$;
\item A $1$-parameter family $\mathbf k(t)$ of $\Spin(9)$-invariant metrics on $\Ss^{15}$;
\item A $1$-parameter family $\check{\mathbf h}(t)$ of $\Sp(n+1)$-invariant metrics on $\C P^{2n+1}$;
\end{enumerate}
and all metrics in (i), (ii), and (iii) above are invariant under the  antipodal (right) $\Z_2$-action, and hence descend to homogeneous metrics invariant under the same groups on $\R P^{2n+1}$, $\R P^{4n+3}$, and $\R P^{15}$, respectively, that we denote by the same symbols. Throughout this paper, as above, $t$ and $t_i$ denote positive real numbers.

Geometrically, the first three families above are obtained by rescaling the unit round metric $\gr$ in the vertical directions of the Hopf bundles
\begin{equation}\label{eq:hopfbundles}
\Ss^1\to \Ss^{2n+1}\to \C P^n,\qquad \Ss^3\to\Ss^{4n+3}\to\Hr P^n, \qquad \Ss^7\to \Ss^{15}\to\Ss^8\big(\tfrac12\big).
\end{equation}
As it turns out, this procedure keeps the corresponding $\G$-actions isometric.
More precisely, decomposing $\gr=\g_{\rm hor}+\g_{\rm ver}$ into horizontal and vertical components,
\begin{equation*}
\mathbf g(t)=\g_{\rm hor}+t^2\g_{\rm ver}, \quad \mathbf h(t_1,t_2,t_3)=\g_{\rm hor}+\textstyle\sum\limits_{i=1}^3 t_i^2\,\dd x_i^2, \quad \mathbf k(t)=\g_{\rm hor}+t^2\g_{\rm ver},
\end{equation*}
where $\dd x_i$, $1\leq i\leq 3$, are dual to a basis of $\gr$-orthonormal vertical (Killing) vector fields on $\Ss^{4n+3}$, so that $\g_{\rm ver}=\dd x_1^2+\dd x_2^2+\dd x_3^2$. In particular, the round metric is recovered by setting the parameters $t$ (or $t_i$) equal to $1$ in any of the above. 
Since permuting $(t_1,t_2,t_3)$ does not change the isometry class of $\mathbf h(t_1,t_2,t_3)$, we shall assume that $0<t_1\leq t_2\leq t_3$ without any loss of generality.

The first eigenvalue of the Laplacian was previously known on $\big(\Ss^{2n+1},\mathbf g(t)\big)$, $\big(\Ss^{15},\mathbf k(t)\big)$, and also on the subfamily $\big(\Ss^{4n+3},\mathbf h(t,t,t)\big)$, which is invariant under the larger isometry group $\Sp(n+1)\Sp(1)$. 
At the heart of these computations, which are carried out in \cite{tanno1,tanno2,bp-calcvar}, building on work of \cite{urakawa79,Berard-BergeryBourguignon82,besson-bordoni}, is the fact that these metrics are \emph{canonical variations} of the round metric with respect to Riemannian submersions with minimal fibers \eqref{eq:hopfbundles}.
That is no longer the case on $\big(\Ss^{4n+3},\mathbf h(t_1,t_2,t_3)\big)$ and $\big(\R P^{4n+3},\mathbf h(t_1,t_2,t_3)\big)$ when one chooses distinct values for the parameters $t_i$, and these metrics are also \emph{not normal homogeneous}, which renders the computation of their first eigenvalue substantially more challenging. 
This was recently achieved in \cite{Lauret-SpecSU(2)} in the lowest dimensional case $\big(\Ss^{3},\mathbf h(t_1,t_2,t_3)\big)$ and $\big(\R P^{3},\mathbf h(t_1,t_2,t_3)\big)$, i.e., that of left-invariant metrics on $\SU(2)\cong\Ss^3$ and $\SO(3)\cong\R P^3$, laying the groundwork for the cases $n\geq1$, which are settled in our first main result.

\begin{mainthm}\label{thm:A}
The first eigenvalue of the Laplacian on $\big(\Ss^{4n+3},\mathbf h(t_1,t_2,t_3)\big)$
and $\big(\R P^{4n+3},\mathbf h(t_1,t_2,t_3)\big)$, with $n\geq1$ and $0<t_1\leq t_2\leq t_3$, are respectively given by
\begin{equation*}
\begin{aligned}
\lambda_1\big(\Ss^{4n+3},\mathbf h(t_1,t_2,t_3) \big)&= 
\min \left\{ 4n +\frac{1}{t_1^2}+\frac{1}{t_2^2}+\frac{1}{t_3^2},\;  8n+ \frac{4}{t_2^2} + \frac{4}{t_3^2},\;  8(n+1) \right\},\\
\lambda_1\big(\R P^{4n+3},\mathbf h(t_1,t_2,t_3) \big)&=
\min \left\{ 8n+ \frac{4}{t_2^2} + \frac{4}{t_3^2},\;  8(n+1) \right\}.
\end{aligned}
\end{equation*}
\end{mainthm}

In the special case $t_1=t_2=t_3=t$, the (right) Hopf $\mathsf S^1$-action on $\big(\Ss^{4n+3},\mathbf h(t,t,t)\big)$ is isometric and commutes with the transitive (left) $\Sp(n+1)$-action.
Thus, the orbit space $\C P^{2n+1}=\Ss^{4n+3}/\mathsf S^1$ is also a homogeneous space with an action of $\Sp(n+1)$. The induced homogeneous metrics $\check{\mathbf h}(t)$ form the fourth (and last) family listed above. Geometrically,  $\check{\mathbf h}(t)=(\gFS)_{\rm hor}+t^2 (\gFS)_{\rm ver}$, where $\gFS=(\gFS)_{\rm hor}+(\gFS)_{\rm ver}$ is the decomposition into horizontal and vertical components with respect to the Hopf bundle $\C P^1\to \C P^{2n+1}\to \Hr P^n$.
These are the last homogeneous CROSSes whose first eigenvalue of the Laplacian had not been explicitly computed.

\begin{mainthm}\label{thm:B}
The first eigenvalue of the Laplacian on $\big(\C P^{2n+1},\check{\mathbf h}(t) \big)$ is given by
\begin{equation*}
\lambda_1\big(\C P^{2n+1},\check{\mathbf h}(t) \big)= \min \left\{   8n +\frac{8}{t^2},\; 8(n+1)\right\}.
\end{equation*}
\end{mainthm}

More detailed versions of Theorems~\ref{thm:A} and~\ref{thm:B} are found in Theorems~\ref{thm:lambda_1(a,b,c,s)} and~\ref{thm:lambda_1(b,s)}, where the multiplicity of these first eigenvalues is also provided. For the convenience of the reader, formulae for the first eigenvalue of the Laplacian on all homogeneous CROSSes are given in Table~\ref{tab:eigenvalues}.  Moreover, formulae for \emph{all} eigenvalues of the Laplacian on $\Ss^{4n+3}$ and $\R P^{4n+3}$ endowed with the metrics $\mathbf g(t)$ or $\mathbf h(t,t,t)$, and $\big(\C P^{2n+1},\check{\mathbf h}(t) \big)$ are given in Theorems~\ref{thm:SpecCP^2n+1} and \ref{thm:Spec(g(t))}; see also \cite{blp-new}. 

Although Theorem~\ref{thm:B} could have been obtained from the techniques in \cite{Berard-BergeryBourguignon82},  Theorem~\ref{thm:A} requires more general methods that might be of independent interest. In fact, these methods (described in Section~\ref{sec:homspectra}) can be used for spectral computations in any compact homogeneous space $\G/\K$ endowed with any homogeneous metric $\g$.
Recall that if $\g$ is normal homogeneous, then the Laplacian on $(\G/\K,\g)$ acts as the \emph{Casimir element}. Since it is in the \emph{center} of the universal enveloping algebra of $\mathfrak g$, the Casimir element acts via multiplication by a scalar in each irreducible $\G$-module that constitutes the Peter--Weyl decomposition \eqref{eq:PeterWeyl} of $L^2(\G/\K)$.
These scalars, which are the eigenvalues of the Laplacian on $(\G/\K,\g)$, can then be computed using Freudenthal's formula \eqref{eq:Casimirscalar} in terms of a root system. 
However, when the normality assumption on $\g$ is dropped, the Laplacian no longer coincides with the Casimir element, and does not necessarily act via multiplication by a scalar in every irreducible $\G$-module in \eqref{eq:PeterWeyl}. Instead, its action is represented by (typically non-diagonal) self-adjoint endomorphisms on each of these $\G$-modules. 
Our approach is to compute the Laplace spectrum as the union of the spectra of these endomorphisms. 
Although a closed formula analogous to Freudenthal's formula \eqref{eq:Casimirscalar} is probably unfeasible in this level of generality, sufficiently fine algebraic estimates allow us to identify in which $\G$-modules the \emph{smallest} eigenvalue is attained. In this way, at least the first few eigenvalues can be explicitly computed. 

As a first application, 
we show in our next main result that the Laplace spectrum distinguishes homogeneous metrics on a CROSS up to isometries.

\begin{mainthm} \label{thm:rigidity}
Two CROSSes endowed with homogeneous metrics are isospectral if and only if they are isometric.
\end{mainthm}

In dimension 3, a partial result was obtained independently in \cite[Thm.~1.3]{LSS2} and \cite[Thm.~1.5]{Lauret-SpecSU(2)}, in terms of left-invariant metrics on $\SU(2)$ and $\SO(3)$.

Although the hypotheses of Theorem~\ref{thm:rigidity} may seem rather stringent, one should keep in mind that establishing spectral uniqueness of a given Riemannian manifold in complete generality can be extremely challenging. For instance, it remains an open problem whether or not
there exist closed Riemannian manifolds that are isospectral but not isometric to a round sphere $(\Ss^n,\gr)$, $n\geq7$.
However, as in Theorem~\ref{thm:rigidity}, such questions can sometimes be tackled in the presence of symmetries. 
Similar spectral uniqueness results among certain families of homogeneous metrics were recently obtained in~\cite{GordonSutton10, GordonSchuethSutton10, Sutton16pp, Yu15, Yu18pp, Lauret-SpecSU(2), Lauret18, LSS1, LSS2}. 
In contrast, there are also several constructions of (non-isometric) isospectral homogeneous Riemannian manifolds, including curves of left-invariant metrics on several compact Lie groups \cite{Schueth01,Proctor05}, and normal homogeneous metrics on distinct homogeneous spaces \cite{Sutton02, AnYuYu13}. 

As a second application, we finalize the classification of homogeneous metrics on a CROSS that are stable solutions to the Yamabe problem. Since they have constant scalar curvature, homogeneous metrics are trivial solutions to the Yamabe problem, i.e., critical points of the normalized total scalar curvature functional \eqref{eq:HilbertEinstein} in their conformal class. However, they need not be \emph{stable} critical points (i.e., local minimizers), depending on the relative values of their scalar curvature and first Laplace eigenvalue.
These are instances where optimality in a geometric variational problem is not necessarily achieved with the most symmetries, since a global minimizer exists in every conformal class, and a conformal class contains at most one homogeneous metric (up to homotheties).
Stable homogeneous spheres among canonical variations of the round metric were classified in \cite{bp-calcvar}, and among $\big(\Ss^3,\mathbf h(t_1,t_2,t_3)\big)$ and $\big(\R P^3,\mathbf h(t_1,t_2,t_3)\big)$ in \cite{Lauret-SpecSU(2)}. Thus, the only families left to consider are $\big(\C P^{2n+1},\check{\mathbf h}(t)\big)$, for which the stability classification follows easily from Theorem~\ref{thm:B}, see Remark~\ref{rem:cp2n1stability}, as well as $\big(\Ss^{4n+3},\mathbf h(t_1,t_2,t_3)\big)$ and $\big(\R P^{4n+3},\mathbf h(t_1,t_2,t_3)\big)$, which are settled in our next main result.

\begin{mainthm}\label{mainthm:stability3param}
The metric $\mathbf h(t_1,t_2,t_3)$, $(t_1,t_2,t_3)\neq(1,1,1)$, is a stable nondegenerate solution to the Yamabe\- problem on $\Ss^{4n+3}$, and $\R P^{4n+3}$, $n\geq1$, 
 if and only~if
\begin{equation*}
t_1^4+t_2^4+t_3^4
+\big(2n(t_1^2+t_2^2+t_3^2)+8(n^2+n+1)\big)(t_1t_2t_3)^2>2(t_1^2t_2^2+t_1^2t_3^2+t_2^2t_3^2).
\end{equation*}
The parameters $(t_1,t_2,t_3)$ corresponding to these metrics form an unbounded and connected open subset $\mathcal S_n\subset \R^3_{>0}=\big\{(t_1,t_2,t_3)\in\R^3 : t_i>0\big\}$, whose boundary $\partial \mathcal S_n$ in $\R^3_{>0}$ is a smooth, connected, and bounded surface.
\end{mainthm}

For completeness, recall that $\mathbf h(1,1,1)$ is the metric of constant sectional curvature $1$, and it is stable, but \emph{degenerate} on $\Ss^{4n+3}$ and \emph{nondegenerate} on $\R P^{4n+3}$.
For the convenience of the reader, the complete list of homogeneous metrics on CROSSes that are stable solutions to the Yamabe problem is provided in Table~\ref{tab:stable}, in Appendix~\ref{appendix}, combining Theorem~\ref{mainthm:stability3param} and Remark~\ref{rem:cp2n1stability} with \cite{bp-calcvar,Lauret-SpecSU(2)}.

The polynomial inequality in Theorem~\ref{mainthm:stability3param} that defines $\mathcal S_n$ has some interesting algebraic features. Namely, the locus of $(t_1,t_2,t_3)\in\R^3$ where this inequality becomes an equality is an irreducible real algebraic variety $\mathcal V_n\subset\R^3$ of dimension~$2$, such that $\partial\mathcal S_n=\mathcal V_n\cap\R^3_{>0}$. 
However, $\mathcal V_n$ contains (and is singular along) each diagonal line $t_i=t_j$ in the coordinate plane $t_k=0$, where $(i,j,k)$ is any permutation of $(1,2,3)$, cf.~\eqref{eq:xy0}.
Thus, $\mathcal V_n\cap\R^3_{\geq0}$ is \emph{noncompact}, which substantially complicates the proof that the (topological) closure of $\partial \mathcal S_n$ in $\R^3_{\geq0}$ is compact. This is achieved through careful estimates in terms of elementary symmetric polynomials in the variables $(x,y,z)=(t_1^2,t_2^2,t_3^2)$. 
As a consequence, the subset $\R^3_{>0} \smallsetminus \mathcal S_n$ of parameters corresponding to unstable homogeneous solutions is bounded (but not compact).

Combining the above classification of stable solutions to the Yamabe problem and classical results in Bifurcation Theory, it is possible to detect the existence of branches of  solutions issuing from paths of homogeneous metrics when they lose stability, i.e., 
when $(t_1,t_2,t_3)$ leaves the set $\mathcal S_n$.
By uniqueness of homogeneous metrics in their conformal class, these bifurcating solutions must be \emph{inhomogeneous}, fitting a wider context of symmetry-breaking bifurcations~\cite{bp-calcvar,bp-pacific,frankenstein}.

\begin{maincor}\label{maincor:bifurcation}
Branches of inhomogeneous solutions to the Yamabe problem on $\Ss^{4n+3}$ and $\R P^{4n+3}$ bifurcate from any continuous curve $\mathbf h\big(t_1(s),t_2(s),t_3(s)\big)$ of homogeneous metrics such that $\alpha(s)=\big(t_1(s),t_2(s),t_3(s)\big)$ crosses the surface~$\partial \mathcal S_n$.
\end{maincor}

Further bifurcations occur if the Morse index of a path of solutions keeps growing, which happens if higher eigenvalues of the Laplacian become small compared to the scalar curvature. 
For instance, it is known that $i_{\textrm{Morse}}\big(\mathbf h(t,t,t)\big)\nearrow +\infty$ as $t\searrow0$, hence there are infinitely many bifurcation instants as $\Ss^{4n+3}$ collapses to~$\Hr P^n$ along this path of metrics~\cite{bp-calcvar}. In Section~\ref{sec:bifurcations}, we characterize some ways in which the Morse index blows up, without the need to explicitly compute Laplace eigenvalues. In particular, we prove the converse statement to a recent bifurcation criterion for the Yamabe problem on canonical variations of Otoba and Petean~\cite[Thm.~1.1]{OtobaPetean1}, see Proposition~\ref{prop:collapsebif2}.
Finally, we also use Theorem~\ref{mainthm:stability3param} to analyze the stability of $\mathbf h(t_1,t_2,t_3)$ as it degenerates, i.e., as some $t_i\searrow0$, see Proposition~\ref{prop:stability0}.

This paper is organized as follows. The main Lie-theoretic tools used in our spectral computations are presented in Section~\ref{sec:homspectra}. In Section~\ref{sec:firsteigen}, we
fix convenient parametrizations for the families of 
homogeneous metrics on CROSSes 
and prove Theorems~\ref{thm:A} and \ref{thm:B}. 
Section~\ref{sec:specunique} contains the proof of Theorem~\ref{thm:rigidity}.
The applications related to stability and bifurcation in the Yamabe problem are given in Sections~\ref{sec:yamabe} and \ref{sec:bifurcations} respectively, including the proofs of Theorem~\ref{mainthm:stability3param} and Corollary~\ref{maincor:bifurcation}. Tables with the first eigenvalue, volume, scalar curvature, and Yamabe stability classification of all homogeneous metrics on CROSSes are given in the Appendix~\ref{appendix}.

\subsection*{Acknowledgements}
It is our great pleasure to thank the anonymous referees for their commendable attention to details, and many suggestions that greatly improved the final version of this paper.
We also thank to Carlos Martins Da Fonseca for several discussion on the spectra of tridiagonal matrices.

The first-named author would like to thank the Max Planck Institute for Mathematics in Bonn, Germany, and the University of S\~ao Paulo, Brazil, for the excellent working conditions during research visits in the Summer of 2019 and January 2020, respectively, during which parts of this paper were written.

The first-named author was supported by grants from the National Science Foundation (DMS-1904342), PSC-CUNY (Award \# 62074-00 50), the Max Planck Institute for Mathematics in Bonn, and Fapesp (2019/19891-9). The second-named author was supported by grants from FonCyT (BID-PICT 2018-02073) and the Alexander von Humboldt Foundation (return fellowship).
The third-named author was supported by grants from Fapesp (2016/23746-6 and 2019/09045-3).

\section{Spectrum of the Laplacian on a Homogeneous Space}\label{sec:homspectra}

In this section, we briefly recall some elementary facts about the spectrum of the Laplacian on a compact homogeneous space. 
Although this material is classical, usually only the case of \emph{normal homogeneous} metrics is discussed in the literature (see e.g.~\cite[pp.~123-125]{wallach-book}), with the notable exception~\cite{MutoUrakawa80}. 
We shall treat the general case of $\G$-invariant metrics, which is needed to prove Theorems~\ref{thm:A} and~\ref{thm:B}.

Let $\G$ be a compact Lie group and $\K\subset \G$ a closed subgroup, with Lie algebras $\mathfrak g$ and $\mathfrak k$, and fix an $\Ad(\K)$-invariant complement $\mathfrak p$ of $\mathfrak k$ in $\mathfrak g$. 
It is well-known that the space of $\G$-invariant metrics $\g$ on the homogeneous space $\G/\K$ is identified with the space of $\Ad(\K)$-invariant inner products $\langle\cdot,\cdot\rangle$ on $\mathfrak p$, see e.g.~\cite[p.~182]{Besse}.

Let $\pi$ be an irreducible representation of $\G$, that is, $\pi\colon \G\to\GL(V_\pi)$ is a continuous homomorphism of groups, and the (complex) vector space $V_\pi$ does not have any proper $\G$-invariant subspaces. Abusing notation, we also denote by $\pi$ the induced representations of the Lie algebra $\mathfrak g$, of its complexification $\mathfrak g_\C:=\mathfrak g\otimes_\R\C$, and of its universal enveloping algebra $\mathcal U(\mathfrak g_\C)$.
Denote by $V_\pi^\K$ the subspace of $V_\pi$ consisting of elements fixed by $\K$; and by $\langle\cdot,\cdot\rangle_\pi$ an inner product on $V_\pi$ for which $\pi(g)$ is unitary for all $g\in \G$, which exists since $\G$ is compact.
The linear map
\begin{equation*}
\begin{aligned}
V_\pi\otimes (V_\pi^*)^\K &\longrightarrow C^\infty(\G/\K) \\
v\otimes \varphi &\longmapsto {f_{v\otimes\varphi}, \rule{30pt}{0pt} \text{with }}f_{v\otimes \varphi}(x\K) := \varphi\big(\pi(x)^{-1}v\big),
\end{aligned}
\end{equation*}
is well-defined and $\G$-equivariant, where $\G$ acts on the first factor of $V_\pi\otimes (V_\pi^*)^\K$, i.e., $g\cdot v\otimes \varphi=\pi(g)v\otimes \varphi$, and on $C^\infty(\G/\K)$ as $(g\cdot f)(x\K)=f(g^{-1}x\K)$.

Given a $\G$-invariant metric $\g$, denote by $\Delta_\g$ the Laplace--Beltrami operator of the Riemannian manifold $(\G/\K,\g)$. 
It is well-known that, for all $f\in C^\infty(\G/\K)$, 
\begin{equation*}
(\Delta_{\g} f)(x\K) = -\sum_{i=1}^{n} \left. \frac{\dd^2}{\dd t^2} f\big(x\exp(t X_i) \cdot e\K\big)\right|_{t=0},
\end{equation*}
where $\{X_1,\dots,X_n\}$ is an orthonormal basis of $\mathfrak p$, with respect to the inner product $\innerdots$ that induces the metric $\g$ on $\G/\K$, see e.g.~\cite[Thm.~1]{MutoUrakawa80}.
Consider the element $C_\g=\sum_{i=1}^{n} X_i^2 \in  \mathcal U(\mathfrak g)$, and observe that
\begin{align*}
(\Delta_\g  f_{v\otimes \varphi} ) (x\K)
&= -\sum_{i=1}^{n} \left. \frac{\dd^2}{\dd t^2} \,\varphi \left(\pi(\exp(tX_i)) \pi(x^{-1})v \right)\right|_{t=0}  \\
&= \sum_{i=1}^{n}  \varphi \left(\pi(-X_i^2) \, \pi(x^{-1})v \right)\\
&=  \sum_{i=1}^{n}  \big( \pi^*(-X_i^2)\cdot \varphi \big) \left(\pi(x^{-1})v \right)  \\
&= \big( \pi^*(-C_\g)\cdot \varphi \big) \left(\pi(x^{-1})v \right)\\
&=f_{v\otimes (\pi^*(-C_\g) \varphi)}(x\K).
\end{align*}
Note that $C_\g$ depends only on the inner product $\innerdots$ on $\mathfrak p$ that induces the metric $\g$, and not on the choice of orthonormal basis $\{X_1,\dots,X_n\}$.

It is a simple matter to check that $\pi^*(-C_\g):V_\pi^*\to V_\pi^*$ is self-adjoint with respect to $\langle\cdot,\cdot\rangle_{\pi^*}$ and preserves $(V_\pi^*)^\K\simeq V_{\pi^*}^\K$. If $\varphi\in V_{\pi^*}^\K$ is an eigenvector of $\pi^*(-C_\g)|_{V_{\pi^*}^\K} $ with eigenvalue $\lambda$, then 
\begin{equation*}
\Delta_\g f_{v\otimes \varphi} 
= f_{v\otimes (\pi^*(-C_\g) \varphi)}
= f_{ v\otimes (\lambda\varphi)}
= \lambda\, f_{v\otimes \varphi},
\end{equation*}
that is, $f_{v\otimes\varphi}$ is an eigenvector of $\Delta_\g$ with eigenvalue $\lambda$, for every $v\in V_\pi$.
By the Peter--Weyl Theorem, there exists a basis of $L^2(\G/\K,\g)$ consisting of eigenfunctions as above. 
More precisely, the left regular representation of $\G$ on $L^2(\G/\K,\g)$ decomposes as (the closure of) the direct sum of $\G$-modules
\begin{equation}\label{eq:PeterWeyl}
L^2(\G/\K,\g) \simeq \widehat{\bigoplus_{\pi \in \widehat \G_\K}} V_\pi\otimes V_{\pi^*}^\K,
\end{equation}
where $\widehat \G$ is the unitary dual of $\G$, i.e., the set of (equivalence classes of) irreducible unitary representations of $\G$, and $\widehat \G_\K:= \{\pi\in\widehat \G: \dim V_\pi^\K=\dim V_{\pi^*}^\K> 0\}$ is the set of \emph{spherical representations} of the pair $(\G,\K)$. Therefore, we have the following:

\begin{proposition}\label{prop:spec}
The spectrum of the Laplacian $\Delta_\g$ of a compact homogeneous space $\G/\K$, endowed with an arbitrary $\G$-invariant metric $\g$, is given by
\begin{equation}\label{eq:spec}
\spec(\G/\K, \g) :=\spec(\Delta_\g) = \bigcup_{\pi \in\widehat \G_\K} 
\Big\{ \underbrace{\lambda_j^{\pi}(\g),\dots, \lambda_j^{\pi}(\g)}_{d_\pi\text{\rm-times}}: 1\leq j\leq d_\pi^\K \Big\},
\end{equation}
where, for each $\pi\in\widehat \G_\K$, 
we write $d_\pi=\dim V_\pi$, $d_\pi^\K= \dim V_\pi^\K$, and $\lambda_1^{\pi}(\g), \dots,\lambda_{d_\pi^\K}^{\pi}(\g)$ are the eigenvalues of the self-adjoint linear endomorphism $\pi^*(-C_\g)|_{V_{\pi^*}^\K}$ of $V_{\pi^*}^\K$.
\end{proposition}

Note that if $\G/\K$ is connected, the trivial representation is the only irreducible representation of $\G$ contributing the eigenvalue $0\in\spec(\G/\K, \g)$. Consequently, if $\pi\in\widehat \G_\K$ is nontrivial, then $\pi^* (-C_\g)|_{V_{\pi^*}^\K}$ is positive-definite, i.e., $\lambda_j^\pi(\g)>0$.

\subsection{Normal homogeneous case}
Let us now specialize to the situation in which $\G$ is semisimple and connected, and $\innerdots_0$ is a bi-invariant (i.e., $\Ad(\G)$-invariant) inner product on~$\mathfrak g$; for instance, a negative multiple of its Killing form. The corresponding metric $\g_0$ on $\G/\K$ is then called \emph{normal homogeneous}.

Set $m=\dim \G$ and let $\{X_1,\dots,X_m\}$ be an orthonormal basis of $\mathfrak g$ with respect to $\innerdots_0$ such that $X_i\in\mathfrak p$ for all $1\leq i\leq n$, and $X_i\in\mathfrak k$ for all $n+1\leq i\leq m$. Given $\pi\in \widehat \G_\K$, since $\pi(X)\cdot v=0$ for all $X\in \mathfrak k$ and $v\in V_\pi^\K$, it follows that $\pi(C_{\g_0})|_{V_{\pi}^\K}=\pi(\Cas_0)|_{V_{\pi}^\K}$, where $\Cas_{0}=\sum_{i=1}^{m} X_i^2$ is the \emph{Casimir element} of $\mathfrak g$ with respect to $\innerdots_0$.
If the Killing form of $\mathfrak g$ is equal to $-\innerdots_0$, then $\Cas_{0}$ is the standard Casimir element in $\mathcal U(\mathfrak g_\C)$ associated to the complex semisimple Lie algebra $\mathfrak g_\C$.
Since $\Cas_{0}$ lies in the center of $\mathcal U(\mathfrak g)$, 
by Schur's Lemma, $\pi(-\Cas_{0})$ acts on $V_\pi$ as multiplication by a scalar $\lambda^\pi$.
Therefore, in this special case, \eqref{eq:spec} simplifies to
\begin{equation}\label{eq:spec(G/K,g_I)}
\spec(\G/\K, \g_0)=\spec(\Delta_{\g_0}) = \bigcup_{\pi\in\widehat \G_\K} 
\Big\{\!\!\underbrace{\lambda^{\pi}, \dots,  \lambda^{\pi}}_{(d_\pi \times d_\pi^\K)\text{-times}} 
\!\!\Big\}.
\end{equation}
The above scalars $\lambda^\pi$ can be computed using Freudenthal's formula, see \cite[Lemma~5.6.4]{wallach-book} or \cite[Prop.~10.6]{hall-book}.
Namely, fixing a maximal torus $\T$ in $\G$, and a positive system in the induced root system $\Phi(\mathfrak g_\C,\mathfrak t_\C)$, 
\begin{equation}\label{eq:Casimirscalar}
\lambda^\pi=\langle \Lambda_\pi,\Lambda_\pi+ 2\rho_{\mathfrak g}\rangle_{0},
\end{equation}
where $\Lambda_\pi$ is the highest weight of the representation $\pi$, $\rho_{\mathfrak g}$ is half of the sum of positive roots in $\Phi(\mathfrak g_\C,\mathfrak t_\C)$, and $\langle\cdot,\cdot\rangle_{0}$ is the Hermitian extension to $\mathfrak t_\C^*$ of $\innerdots_0|_{\mathfrak t}$. For a general homogeneous metric $\g$ which is not normal, no analogous formula to \eqref{eq:Casimirscalar} that explicitly computes the scalars $\lambda_j^\pi(\g)$ in Proposition~\ref{prop:spec} seems to exist.

\section{\texorpdfstring{Eigenvalues of the Laplacian on $\Ss^{4n+3}$, $\R P^{4n+3}$, and $\C P^{2n+1}$}{Eigenvalues of the Laplacian}}\label{sec:firsteigen}

In this section, we provide explicit formulae for the smallest positive eigenvalue of the Laplace--Beltrami operator on $\big(\Ss^{4n+3},\mathbf h{(t_1,t_2,t_3)}\big)$, $\big(\R P^{4n+3},\mathbf h{(t_1,t_2,t_3)}\big)$, and on $\big(\C P^{2n+1},\check{\mathbf h}(t)\big)$, proving Theorems~\ref{thm:A} and \ref{thm:B} in the Introduction.
The full spectrum of the latter and of the subfamily $\mathbf g(t)$ on $\Ss^{4n+3}$ and $\R P^{4n+3}$ are also computed, see Theorem~\ref{thm:SpecCP^2n+1} and~\ref{thm:Spec(g(t))}, and also \cite{blp-new}.

\subsection{Homogeneous structures}
Consider the quaternionic unitary group
\begin{equation*}
\G=\Sp(n+1)= \left\{ g\in \GL(n+1,\Hr): g^*g=\id \right\},
\end{equation*}
whose Lie algebra is $\mathfrak g=\spp(n+1)= \left\{ X\in \gl(n+1,\Hr): X^*+X=0 \right\}$.
The defining representation of $\G$ on $\Hr^{n+1}$ restricts to an isometric transitive $\G$-action on the unit sphere $\Ss^{4n+3}\subset\Hr^{n+1}$, whose isotropy at $(0,\dots,0,1)\in \Hr^{n+1}$ is the Lie subgroup
\begin{equation*}
\K=\{\diag(A,1)\in\G: A\in \Sp(n)\}\simeq\Sp(n),
\end{equation*}
so that $\Ss^{4n+3}=\G/\K$.
Clearly, the corresponding Lie subalgebra is $\mathfrak k=\{\diag(X,0)\in\mathfrak g:X\in\spp(n)\}\simeq\spp(n)$.
Consider the reductive decomposition
$\mathfrak g=\mathfrak k\oplus \mathfrak p$, where $\mathfrak p=\mathfrak p_0\oplus \mathfrak p_1$ splits as the vertical space $\mathfrak p_0\simeq\operatorname{Im}\Hr$ and horizontal space $\mathfrak p_1\simeq\Hr^n$ for the Hopf fibration $\Ss^3\to \Ss^{4n+3}\to\Hr P^n$. 
Recall the isotropy representation of $\K$ is trivial on $\mathfrak p_0$, and irreducible on $\mathfrak p_1$.
Note that $\mathfrak p_0\simeq\mathfrak{sp}(1)$ is a Lie subalgebra of $\mathfrak g$, spanned by the unit imaginary quaternions
\begin{equation}\label{eq:ijk}
X_1=\diag(0,\dots,0,\mi), \quad X_2=\diag(0,\dots,0,\mj), \quad X_3=\diag(0,\dots,0,\mk),
\end{equation}
and the corresponding Lie subgroup is
\begin{equation}\label{eq:subgroupH}
\H=\{\diag(\id,q)\in\G:  |q|^2=q\bar q=1 \}\simeq\Sp(1)\simeq\SU(2).
\end{equation}

The above (left) $\G$-action on $\Ss^{4n+3}\subset \Hr^{n+1}$ commutes with the (right) actions of $\Z_2$ via the antipodal map, and of $\mathsf S^1$-action by complex unit multiplication. Thus, it descends to transitive $\G$-actions on the quotients $\R P^{4n+3}=\Ss^{4n+3}/\Z_2$ and $\C P^{2n+1}=\Ss^{4n+3}/\mathsf S^1$, respectively. These $\G$-actions have isotropy (conjugate to)
\begin{equation*}
\begin{aligned}
\K\cdot\Z_2&=\{\diag(A,\pm 1)\in\G : A \in\Sp(n)\}\simeq\Sp(n)\Z_2,\\
\check{\K}&=\{\diag(A,e^{\mi \theta})\in\G:A\in\Sp(n), e^{\mi \theta}\in\mathsf S^1\}\simeq\Sp(n)\Ut(1),
\end{aligned}
\end{equation*}
respectively, so that $\R P^{4n+3}=\G/(\K\cdot\Z_2)$ and $\C P^{2n+1}=\G/\check{\K}$.
Note that the $\mathsf S^1$-action extends the $\Z_2$-action, so $\K\cdot\Z_2\subset\check{\K}$; and, as $\Ut(1)/\Z_2\cong \mathsf S^1$, we have $\C P^{2n+1}=\R P^{4n+3}/\mathsf S^1$.

The Lie algebra of $\K\cdot\Z_2$ is the same as that of its identity connected component $\K$, that is, $\mathfrak k$. The isotropy representation of $\K\cdot\Z_2$ on $\mathfrak g=\mathfrak k\oplus\mathfrak p$ extends that of $\K$, with the element $\diag(\Id,-1)$ acting trivially on $\mathfrak p_1\oplus\Span_\R\{X_1\}$ and nontrivially, i.e., as multiplication by $-1$, on $\check{\mathfrak p}_0:=\Span_\R \{X_2,X_3\}$.
Meanwhile, the Lie algebra of $\check\K$ is $\check{\mathfrak k} = \mathfrak k\oplus\Span_\R\{X_1\}$, and the corresponding reductive decomposition is $\mathfrak g=\check{\mathfrak k}\oplus \check{\mathfrak p}$, where $\check{\mathfrak p}=\check{\mathfrak p}_0\oplus \mathfrak p_1$.
Both $\check{\mathfrak p}_0$ and $\mathfrak p_1$ are irreducible for the isotropy representation of $\check\K$, with the $\Sp(n)$ factor acting trivially on $\check{\mathfrak p}_0$ and via the defining representation on $\mathfrak p_1$, and the $\Ut(1)$ factor acting by rotation on $\check{\mathfrak p}_0$ and trivially on $\mathfrak p_1$.

Geometrically, the inclusions $\K\subset\K\cdot\Z_2\subset\check{\K}$ correspond to successive quotients of the Hopf fibration (top row) by the (right) actions of $\Z_2$ and $\mathsf S^1$, as follows:
\begin{equation}\label{eq:Hopffibrations}
\begin{gathered}
 \xymatrix{
 \Ss^3\; \ar[r] \ar[d] & \Ss^{4n+3} \ar[d] \ar[r]  & \Hr P^n\ar@{=}[d] \\
 \R P^3\; \ar[r] \ar[d] & \R P^{4n+3} \ar[d] \ar[r]  & \Hr P^n\ar@{=}[d] \\
 \C P^1\ar[r] & \C P^{2n+1} \ar[r] &\Hr P^n. }
 \end{gathered}
\end{equation}
The arrows from top to middle row are double covers, while the arrows from middle to bottom row are projections of $\mathsf S^1$-bundles. Note that $\check{\mathfrak p}_0$ and $\mathfrak p_1$ are the vertical and horizontal spaces for the bundle in the bottom row.

\subsection{Homogeneous metrics}
We now parametrize (up to isometries) the spaces of $\G$-invariant metrics on $\Ss^{4n+3}$, $\R P^{4n+3}$, and $\C P^{2n+1}$, with respect to the above homogeneous structures. For more details, see~\cite[Ex.~6.16, 6.21]{mybook} and \cite{Ziller82}.

We begin with $\G$-invariant metrics on $\Ss^3$ and $\R P^3$, that is, left-invariant metrics on $\Sp(1)\simeq\SU(2)\cong\Ss^3$ and $\SO(3)\cong \R P^3$. Every such metric is isometric to one induced by a \emph{diagonal} inner product with respect to the basis $\{\mi,\mj,\mk\}$ of the Lie algebra $\spp(1)$, i.e., of the form
\begin{equation*}
\phantom{, \quad a,b,c\in\R_{>0},} \innerdots_{\abc}:=
 \frac{1}{a^2}\, \widehat\mi\otimes\widehat\mi 
+\frac{1}{b^2}\, \widehat\mj\otimes\widehat\mj 
+\frac{1}{c^2}\, \widehat\mk\otimes\widehat\mk, 
\quad a,b,c\in\R_{>0},
\end{equation*}
where $\{\widehat\mi,\, \widehat\mj,\, \widehat\mk\}$ is the basis of $\mathfrak{sp}(1)^*$ dual to $\{\mi,\mj,\mk\}$. Note that $\{a\mi,b\mj,c\mk\}$ is $\innerdots_{\abc}$-orthonormal.
Denote by $\g_{\abc}$ the corresponding $\G$-invariant metric on $\Ss^3$, and observe that $(\Ss^3,\g_{(a,a,a)})$ is a round sphere of constant sectional curvature~$a^2$.
Clearly, permuting $(a,b,c)\in\R^3_{>0}$ gives rise to metrics $\g_{\abc}$ that are isometric, and it is not difficult to see that there are no other isometries among them (this follows, e.g., by inspecting their Ricci endomorphisms).
Moreover, all $\g_{\abc}$ descend to $\G$-invariant metrics on $\R P^3$, that we shall denote by the same symbol. Similarly, the only isometries among these metrics on $\R P^3$ arise from permuting $(a,b,c)$.
Altogether, we have the following spaces of isometry classes of $\G$-invariant metrics:
\begin{equation*}
\Met^{\Sp(1)}(\Ss^3)\cong\Met^{\Sp(1)}(\R P^3)\cong \big\{\g_{(a,b,c)}: a\geq b\geq c>0\big\}.
\end{equation*}

For $n\geq1$, fix the $\Ad(\G)$-invariant inner product $\langle X,Y\rangle_0 =-\frac{1}{2}\operatorname{Re}\tr(XY)$ on the Lie algebra $\mathfrak g=\spp(n+1)$.
Identify $\mathfrak p_0\cong\spp(1)$ via the isomorphism that associates each diagonal matrix in \eqref{eq:ijk} to their unique nonzero entry, 
and define an $\Ad(\K)$-invariant inner product on $\mathfrak p=\mathfrak p_0\oplus\mathfrak p_1$ as follows:
\begin{equation*}
\phantom{a,b,c,s\in\R_{>0}\quad }\innerdots_{\abcs} :=
\frac12 \innerdots_{\abc} |_{\mathfrak p_0} + \frac{1}{s^{2}}  \innerdots_0|_{\mathfrak p_1}, \quad a,b,c,s\in\R_{>0}.
\end{equation*}
Denote by $\g_{\abcs}$ the corresponding $\G$-invariant metric on $\Ss^{4n+3}=\G/\K$, and observe that $\innerdots_0|_{\mathfrak p}=\innerdots_{(1,1,1,1)}$, hence $\big(\Ss^{4n+3},\g_{(1,1,1,1)}\big)$ is normal homogeneous.
Once again, it is not difficult to see that the only isometries among $\g_{(a,b,c,s)}$ arise from permuting $(a,b,c)\in\R^3_{>0}$, and all such $\G$-invariant metrics on $\Ss^{4n+3}$ descend to $\G$-invariant metrics on $\R P^{4n+3}$, that we shall denote by the same symbol. (Endowing both spaces with $\g_{(a,b,c,s)}$, the vertical arrow $\Ss^{4n+3}\to \R P^{4n+3}$ in \eqref{eq:Hopffibrations} is a \emph{Riemannian} double cover.) Altogether, we have the following spaces of isometry classes of $\G$-invariant metrics:
\begin{equation*}
\Met^{\Sp(n+1)}(\Ss^{4n+3})\cong  \Met^{\Sp(n+1)}(\R P^{4n+3}) \cong \big\{ \g_{(a,b,c,s)}: a\geq b\geq c>0,\, s>0\big\}.
\end{equation*}

Furthermore, the restriction of $\innerdots_{(a,b,c,s)}$ to $\check{\mathfrak p}$ is $\Ad(\check\K)$-invariant if and only if $b=c$, in which case it defines a $\G$-invariant metric $\check\g_{(b,s)}$ on $\C P^{2n+1}=\G/ \check\K$. In this situation, the quotient maps from $\Ss^{4n+3}$ and $\R P^{4n+3}$ endowed with $\g_{(a,b,b,s)}$ onto $\big(\C P^{2n+1},\check\g_{(b,s)}\big)$ corresponding to (right) $\mathsf S^1$-actions, i.e., the vertical arrows in \eqref{eq:Hopffibrations}, are Riemannian submersions.
Similarly to the previous cases, it is not hard to check that the metrics $\check\g_{(b,s)}$ are pairwise non-isometric, so the space of isometry classes of $\G$-invariant metrics on $\C P^{2n+1}$ is
\begin{equation*}
\Met^{\Sp(n+1)}(\C P^{2n+1})\cong \big\{\check\g_{(b,s)}: b>0,\, s>0\big\}.
\end{equation*}

\begin{remark}\label{rem:isometries}
The above parameterizations $\g_{\abc}$, $\g_{\abcs}$, and $\check\g_{(b,s)}$ of $\G$-invariant metrics on $\Ss^3$, $\Ss^{4n+3}$, $\R P^{3}$, $\R P^{4n+3}$, and $\C P^{2n+1}$ are convenient for the spectral calculations.
In fact, the first eigenvalues of their respective Laplacians are homogeneous quadratic polynomials in the parameters $a,b,c,s$. 
However, from a geometric viewpoint, these metrics are more naturally parametrized in terms of the lengths $t_i$ of vertical Killing vector fields in the Hopf bundles \eqref{eq:hopfbundles}, compared to those in the round or Fubini--Study metric, with horizontal directions unchanged. These parametrizations, used in the Introduction and in subsequent sections, are related to the above via the isometries (recall that $n\geq1$ throughout)
\begin{equation}\label{eq:different-parameters1}
\begin{aligned}
&\text{On } \Ss^3 \text{ and } \R P^3: &\mathbf h(t_1,t_2,t_3) &\cong \g_{\left(t_1^{-1},t_2^{-1},t_3^{-1}\right)},\\
&\text{On } \Ss^{4n+3} \text{ and } \R P^{4n+3}: &\mathbf h(t_1,t_2,t_3) &\cong \g_{\left((\sqrt2 t_1)^{-1} ,(\sqrt2 t_2)^{-1},(\sqrt2 t_3)^{-1},1\right)},  \\
&\text{On } \C P^{2n+1}: &\check{\mathbf h}(t) &\cong \check{\g}_{\left((\sqrt2 t)^{-1},1\right)},
\end{aligned}
\end{equation}
or, equivalently, 
\begin{equation}\label{eq:different-parameters2}
\begin{aligned}
&\text{On } \Ss^3 \text{ and } \R P^3:  &\g_{(a,b,c)} &\cong \mathbf h(\tfrac1a,\tfrac1b,\tfrac1c),\\
&\text{On } \Ss^{4n+3} \text{ and } \R P^{4n+3}: &\g_{(a,b,c,s)} &\cong \tfrac{1}{s^2}\,\mathbf h\!\left(\tfrac{s}{\sqrt2a}, \tfrac{s}{\sqrt2b}, \tfrac{s}{\sqrt 2 c}\right), \\
&\text{On } \C P^{2n+1}: &\check\g_{(b,s)} &\cong \tfrac{1}{s^2}\check{\mathbf h}(\tfrac{s}{\sqrt{2}b}).
\end{aligned}
\end{equation}

In particular, note that the normal homogeneous metrics on $\Ss^{4n+3}$ and $\R P^{4n+3}$, $n\geq1$, induced by $\langle\cdot,\cdot\rangle_0$ are 
$\mathbf h\big(\tfrac{1}{\sqrt2},\tfrac{1}{\sqrt2},\tfrac{1}{\sqrt2}\big)=\g_{(1,1,1,1)} =\g_{\rm hor}+\tfrac12\g_{\rm ver}$, where $\gr=\g_{\rm hor}+\g_{\rm ver}$ is the decomposition of the metric of constant sectional curvature $1$ with respect to the bundle in the top (respectively, middle) row in \eqref{eq:Hopffibrations}. 
Similarly, the normal homogeneous metric on $\C P^{2n+1}$, $n\geq1$, induced by $\langle\cdot,\cdot\rangle_0$ is  
$\check{\mathbf h}\big(\tfrac{1}{\sqrt2}\big)=\g_{(1,1)} =\g_{\rm hor}+\tfrac12\g_{\rm ver}$, where $\gFS=\g_{\rm hor}+\g_{\rm ver}$ is the decomposition of the Fubini--Study metric with respect to the bottom row in \eqref{eq:Hopffibrations}.
\end{remark}

\subsection{Implicit spectra}
We now describe the spectra 
\begin{equation*}
\Spec(\Ss^{4n+3}, \g_{\abcs}), \; \Spec(\R P^{4n+3}, \g_{\abcs}), \, \text{ and }\, \Spec(\C P^{2n+1}, \check\g_{(b,s)}), \quad n\geq 1,
\end{equation*}
implicitly in terms of $\Spec(\Ss^3,\g_{\abc})$. 

For any integer $k\geq0$, let $(\tau_k,V_{\tau_k})$ denote the (unique, up to equivalence) irreducible representation of $\H\simeq \SU(2)$ of dimension $k+1$. 
For $a,b,c>0$, let $\nu_1^{(k)}(a,b,c),\dots,\nu_{k+1}^{(k)}(a,b,c)$ denote the eigenvalues of the positive-definite self-adjoint operator 
\begin{equation}\label{eq:tau_k}
\tau_k\big(\!-a^2X_1^2 -b^2X_2^2 -c^2X_3^2\big)\colon V_{\tau_k}\to V_{\tau_k},
\end{equation}
where $X_i$ are as in \eqref{eq:ijk}. From Proposition~\ref{prop:spec}, we conclude that
\begin{equation*}
\spec(\Ss^3,\g_{(a,b,c)}) = 
\bigcup_{k\geq0} \Big\{\underbrace{\nu_j^{(k)}(a,b,c),\dots, \nu_j^{(k)}(a,b,c)}_{(k+1)\text{-times}}: 1\leq j\leq k+1 \Big\}.
\end{equation*}
This spectrum is studied in detail in \cite{Lauret-SpecSU(2)}, where it is shown that
\begin{equation}\label{eq:nu^(2)}
\begin{aligned}
\nu_1^{(0)}(a,b,c) &=0,& \rule{10mm}{0mm}
\nu_1^{(2)}(a,b,c) &=4(b^2+c^2), \\
\nu_1^{(1)}(a,b,c) &=a^2+b^2+c^2, &
\nu_2^{(2)}(a,b,c) &=4(a^2+c^2), \\
\nu_2^{(1)}(a,b,c) &=a^2+b^2+c^2,&
\nu_3^{(2)}(a,b,c) &=4(a^2+b^2),
\end{aligned}
\end{equation}
and $\lambda_1(\Ss^3,\g_{\abc})$ is the smallest among the above, leaving out $\nu_1^{(0)}(a,b,c) =0$. More precisely, if $a\geq b\geq c>0$, then $\nu_1^{(2)}(a,b,c)\leq \nu_2^{(2)}(a,b,c) \leq \nu_3^{(2)}(a,b,c)$, and
\begin{equation*}
\lambda_1(\Ss^3,\g_{\abc})=\min\!\big\{a^2+b^2+c^2, 4(b^2+c^2)\big\}.
\end{equation*}
The main tool to prove this result is \cite[Lem.~3.4]{Lauret-SpecSU(2)}, namely, given
integers $1\leq j\leq k+1$, we have:
\begin{equation}\label{eq:nu^(k)>=}
	\nu_j^{(k)}(a,b,c) \geq 
	\begin{cases}
		2kb^2+k^2c^2&\quad\text{ if $k\geq0$},\\
		a^2+(2k-1)b^2+k^2c^2&\quad\text{ if $k\geq0$ is odd}.
	\end{cases}
\end{equation} 
Furthermore, for any integers $k\geq0$ and $1\leq j\leq k+1$, it is easy to see that 
\begin{equation}\label{eq:nu_j^k(a,a,a)}
\nu_j^{(k)}(a,a,a)= k(k+2)a^2.
\end{equation}

In order to apply Proposition~\ref{prop:spec} to describe the spectra $\Spec(\Ss^{4n+3}, \g_{\abcs})$, $\Spec(\R P^{4n+3}, \g_{\abcs})$, and $\Spec(\C P^{2n+1}, \check\g_{(b,s)})$ for $n\geq1$, we need to introduce some Lie-theoretic objects.
Fix the maximal torus of $\G$ given by
\begin{equation*}
\T := \big\{\!\diag(e^{\mi\theta_1},\dots, e^{\mi\theta_{n+1}}) : \theta_1,\dots,\theta_{n+1}\in\R \big\},
\end{equation*}
whose Lie algebra $\mathfrak t$ (respectively, its complexification $\mathfrak t_\C:=\mathfrak t\otimes_\R\C$) consists of elements of the form
$Y=\diag({\mi\theta_1},\dots, {\mi\theta_{n+1}}),$
with $\theta_j\in \R$ (respectively, $\theta_j\in \C$), for all $1\leq j\leq n+1$.
Let $\varepsilon_j\colon \mathfrak t_\C\to \C$ be given by $\varepsilon_j(Y)=\mi \theta_j$, where $Y$ is as above, so that $\{\varepsilon_1,\dots, \varepsilon_{n+1}\}$ is a basis of $\mathfrak t_\C^*$.

Denote the Hermitian extension of $\innerdots_0$ to $\mathfrak g_\C$ and $\mathfrak t_\C^*$ by the same symbol $\innerdots_0$. One easily checks that $\langle \varepsilon_i, \varepsilon_j\rangle_0 ={ 2} \delta_{ij}$ for all $1\leq i,j\leq n+1$. Indeed, setting $Y_j=\diag(0,\dots,0,\mi,0,\dots,0)$, where the nonzero coordinate is in the $j$th entry, one has that $\big\{\sqrt{2} Y_1,\dots,\sqrt{2} Y_{n+1}\big\}$ is an orthonormal basis of $\mathfrak t_\C$ with respect to $\langle\cdot,\cdot\rangle_0$, so its corresponding dual basis $\big\{\frac{1}{\sqrt{2}} \varepsilon_1,\dots, \frac{1}{\sqrt{2}} \varepsilon_{n+1}\big\}$ is an orthonormal basis of $\mathfrak t_\C^*$. 

The root system of $\mathfrak g_\C$ with respect to the Cartan subalgebra $\mathfrak t_\C$ is given by $\Phi(\mathfrak g_\C,\mathfrak t_\C) = \{\pm \varepsilon_i\pm\varepsilon_j:i\neq j\}\cup \{\pm2\varepsilon_i \}$. 
Consider the standard positive system, which has positive roots $\Phi^+(\mathfrak g_\C,\mathfrak t_\C) = \{\varepsilon_i\pm\varepsilon_j:i<j\}\cup \{2\varepsilon_i \}$. In particular, half of the sum of positive roots is $\rho_{\mathfrak g}= \sum_{j=1}^{n+1} (n+2-j)\varepsilon_j$. 

Since $\G$ is simply-connected, the set of dominant $\G$-integral weights coincides with the set of dominant algebraically integral weights of $\mathfrak g_\C$, which is given by elements of the form $\sum_{i=1}^{n+1} a_i\varepsilon_i$ with $a_i\in\Z$ satisfying $a_1\geq\dots\geq a_{n+1}\geq0$. 
If $\Lambda$ is a dominant $\G$-integral weight, we denote by $\pi_\Lambda$ the irreducible $\G$-representation having highest weight $\Lambda$, which exists and is unique (up to equivalence) by the Highest Weight Theorem, see e.g. Hall~\cite[Thm 9.4, 9.5]{hall-book}.

\begin{lemma}\label{lem:implicitSpec}
Let $n\geq1$ be an integer. For positive real numbers $a,b,c,s$ and integers $p\geq q\geq0$, we have that
\begin{align*}
\spec(\Ss^{4n+3},\g_{(a,b,c,s)}) 
&= \!\!\!\bigcup_{\substack{p\geq q\geq0 \\ 1\leq j\leq p-q+1} } \!\! \Big\{ \!\underbrace{\lambda_j^{(p,q)}(a,b,c,s),..., \lambda_j^{(p,q)}(a,b,c,s) }_{d_{p,q}} \Big\},
\\
\spec(\R P^{4n+3},\g_{(a,b,c,s)}) 
    &= \!\!\!\bigcup_{\substack{ p\geq q\geq0 \\ p-q \text{ \rm even}\\ 1\leq j\leq p-q+1  }} \!\! \Big\{ \!\underbrace{\lambda_j^{(p,q)}(a,b,c,s),..., \lambda_j^{(p,q)}(a,b,c,s) }_{d_{p,q}}\Big\},
    \\
\Spec(\C P^{2n+1} , \check{\g}_{(b,s)}) 
&= \!\!\bigcup_{\substack{ p\geq q\geq0 \\ p-q \text{ \rm even} }} \! \! \Big\{ \underbrace{\check\lambda^{(p,q)}(b,s), \dots, \check\lambda^{(p,q)}(b,s) }_{d_{p,q}}  \Big\},
\end{align*}
where
\begin{align}
\label{eq:lambda}
\lambda_j^{(p,q)}(a,b,c,s) 	&= \big(4pn+{ 4}q(p+n+1) \big)s^2 +  2 \nu_j^{(p-q)}(a,b,c),\\
\label{eq:lambda2}
\check{\lambda}^{(p,q)}(b,s) &= \big( 4pn+4q(p+n+1) \big)s^2 + 2(p-q)(p-q+2)b^2,\\
\label{eq:d_pq}
d_{p,q} &= \frac{(p+q+2n+1)(p-q+1)}{(2n+1)(p+1)}\binom{p+2n}{p} \binom{q+2n-1}{q}.
\end{align}
\end{lemma}

\begin{proof}
	We begin by identifying the corresponding spherical representations. 
	It is well-known that (see for instance \cite[Problem IX.11]{Knapp-book-beyond})
	\begin{equation*}
	\widehat \G_\K = \{\pi_{p,q}:=\pi_{p\varepsilon_1+ q\varepsilon_2}: p\geq q\geq 0\}. 
	\end{equation*}
We henceforth abbreviate $V_{p,q}=V_{\pi_{p,q}}$. Since $\K$ and $\H$ commute, the subspace $V_{p,q}^\K$ is $\H$-invariant.
From Lepowsky's classical branching law from $\G$ to $\K\times\H$, or as a direct consequence of \cite[Thm.~3.3]{WallachYacobi09}, we have that
	\begin{equation}\label{eq:V_pq|K}
	V_{p,q}^\K\simeq V_{\tau_{p-q}}
	\;\text{as $\H$-modules}.
	\end{equation}
	In particular, $d_{\pi_{p,q}}^\K=\dim V_{p,q}^\K =\dim V_{\tau_{p-q}}=p-q+1$. 
	
Since  $\K\subset \K\cdot\Z_2\subset \check\K$, we have $\widehat \G_{\check\K}\subset \widehat \G_{\K\cdot\Z_2} \subset \widehat \G_\K$. First, we determine $\widehat \G_{\check\K}$.
	An element $\pi_{p,q}\in \widehat \G_{\K}$ is in $\widehat \G_{\check\K}$ if there is a nonzero vector in $V_{p,q}^\K$ fixed by the $\Ut(1)$ factor in $\check\K$ or, equivalently, annihilated by $X_1$ in \eqref{eq:ijk}. 
	As an $\H$-module, $V_{p,q}^\K$ is irreducible with highest weight $p-q$ by \eqref{eq:V_pq|K}. 
	By the standard representation theory of $\sll(2,\C)$-modules, we have the (weight) decomposition
\begin{equation*}
V_{p,q}^\K =\bigoplus_{l=0}^{p-q} V_{p,q}^\K(p-q-2l),
\end{equation*}
	where $\dim V_{p,q}^\K(p-q-2l)=1$ for all $0\leq l\leq p-q$, and $\pi_{p,q}(X_1) v= (p-q-2l) \mi\, v$ for all $v\in V_{p,q}^\K(p-q-2l)$. 
	Hence, $V_{p,q}^{\check\K}= V_{p,q}^\K(0)$, which is nontrivial if and only if $p-q$ is even. 
	Thus, we conclude that 
\begin{equation*}
\widehat \G_{\check\K}= \{\pi_{p,q}: p\geq q\geq0,\; p\equiv q\mod 2\}
\end{equation*}
	and $\dim V_\pi^{\check\K}=1$ for all $\pi\in\widehat \G_{\check\K}$, i.e., the branching from $\G$ to $\check\K$ is multiplicity-free. 
We now determine $\widehat \G_{\K\cdot\Z_2}$.
Multiplication by $\diag(\Id,-1)$ maps the identity connected component (identified with $\K$) to the other connected component of $\K\cdot\Z_2$, and
$\diag(\Id,-1)$ lies in the maximal torus $\T$.
In fact, $\diag(\Id,-1)=\exp(0,\dots,0,\pi\mi)$. 
Its action on a weight space $V_{p,q}^\K(p-q-2l)$ is thus given by multiplication by $e^{(p-q-2l)\pi\mi}=(-1)^{p-q}$, i.e., the action on $V_{p,q}^\K$  is by $(-1)^{p-q}\Id_{V_{p,q}^\K}$. Consequently,
\begin{equation*}
\widehat \G_{\K\cdot\Z_2} = \{\pi_{p,q}: p\geq q\geq 0,\; p-q\text{ even}\}.
\end{equation*}

It is a simple matter to check that $\dim V_{{p,q}}=d_{p,q}$ as in \eqref{eq:d_pq} by using the Weyl Dimension Formula, see e.g.\ \cite[Thm.~5.84]{Knapp-book-beyond}.

From Proposition~\ref{prop:spec}, it just remains to show that, for every $p\geq q\geq0$, the eigenvalues of $\pi_{p,q}( -C_{\abcs })|_{V_{{p,q}}^\K}$ are 
$\lambda_j^{(p,q)}(a,b,c,s)$, $1\leq j\leq p-q+1$, as in \eqref{eq:lambda}, and the (only) eigenvalue of $\pi_{p,q}( -\check C_{(b,s)})|_{V_{{p,q}}^{\check\K}}$ is $\check\lambda^{(p,q)}(b,s)$ if $p\equiv q\mod  2$,
as in \eqref{eq:lambda2}.
Here, we abbreviate $C_{\abcs}= C_{\g_{\abcs }}$ and $\check C_{(b,s)}= C_{\check\g_{(b,s)}}$. 
Note that this includes the case of $\Spec(\R P^{4n+3},\g_{\abcs})$, since the Laplace operator of $(\R P^{4n+3},\g_{\abcs})$ has the same spectrum as the restriction to 
\begin{equation}\label{eq2:projective-restriction}
\widehat{\bigoplus_{\substack{ p\geq q\geq0 \\ p-q \text{ even} }}} V_{p,q}\otimes V_{\pi_{p,q}^*}^\K
\simeq L^2\!\big(\G/(\K\cdot\Z_2), \g_{\abcs}\big)
\end{equation}
of the Laplace operator of $(\Ss^{4n+3},\g_{\abcs})$.
	
Let $\{X_4,\dots,X_m\}$ be an orthonormal basis of $\mathfrak p_1$ with respect to $\langle\cdot,\cdot\rangle_0$. Then $\big\{\sqrt2 aX_1,  \sqrt2 bX_2, \sqrt2 cX_3, sX_4,\dots,sX_m\big\}$ and $\big\{ \sqrt2 bX_2, \sqrt2 bX_3, sX_4,\dots,sX_m\big\}$ are orthonormal bases of $(\mathfrak p,\innerdots_{\abcs})$ and $(\check{\mathfrak p}, \innerdots_{(a,b,b,s)}|_{\check{\mathfrak p}})$ respectively. 
Hence
\begin{align*}
	C_{\abcs}
	&=2a^2X_1^2+ 2b^2X_2^2+ 2c^2X_3^2 +s^2(X_4^2+\dots+X_m^2) 
	\\ 
	&=s^2\Cas_0+ 2(a^2X_1^2+ b^2X_2^2+ c^2X_3^2) -2s^2(X_1^2+ X_2^2+ X_3^2) -s^2\Cas_{\mathfrak k}, 
	\\ 
	\check C_{(b,s)}
	&=2b^2X_2^2+ 2b^2X_3^2 +s^2(X_4^2+\dots+X_m^2) 
	\\ 
	&=s^2\Cas_0 + 2(b^2-s^2) (X_2^2+ X_3^2) -s^2\Cas_{\mathfrak k}, 
	\end{align*}
where $\Cas_{\mathfrak k}$ denotes the Casimir element of ${\mathfrak k}$ with respect to $\innerdots_0|_{\mathfrak k}$, that is, $\Cas_{\mathfrak k}=\sum_{i=1}^{\dim\mathfrak k} Y_i^2$, where $\{Y_1,\dots,Y_{\dim\mathfrak k}\}$ is a $\innerdots_0$-orthonormal basis of $\mathfrak k$.
Clearly, $\pi_{p,q}(\Cas_{\mathfrak k})$ acts trivially on $V_{p,q}^\K$.
	From \eqref{eq:Casimirscalar}, we have that $\pi_{p,q}(-\Cas_0)$ acts on $V_{{p,q}}$ by multiplication by the scalar
	\begin{equation*}
	\lambda^{\pi_{p,q}}
	= \langle p\varepsilon_1+q\varepsilon_2 +2\rho_{\mathfrak g}, p\varepsilon_1+q\varepsilon_2 \rangle_0 
	= 2p(p+2n+2)+2q(q+2n), 
	\end{equation*}
	and $\pi_{p,q}\big(\!-(X_1^2+ X_2^2+ X_3^2)\big)|_{V_{{p,q}}^\K} = \tau_{p-q}\big(\!-(X_1^2+ X_2^2+ X_3^2)\big)$ by multiplication by $(p-q)(p-q+2)$. 
Since the eigenvalues of $\pi_{p,q}\big(\!-(a^2X_1^2+ b^2X_2^2+ c^2X_3^2)\big)|_{V_{{p,q}}^\K} = \tau_{p-q}\big(\!-(a^2X_1^2+ b^2X_2^2+ c^2X_3^2)\big)$ are precisely $\nu_j^{(p-q)}(a,b,c)$ for $1\leq j\leq p-q+1$, the claim regarding \eqref{eq:lambda} follows.
	Furthermore, $\pi_{p,q}\big(\!-(X_2^2+ X_3^2)\big)|_{V_{{p,q}}^{\check\K}} = \pi_{p,q}\big(\!-(X_1^2+ X_2^2+ X_3^2)\big) |_{V_{{p,q}}^{\check\K}}$ because $X_1$ acts trivially on $V_{{p,q}}^{\check\K}$, thus  $\pi_{p,q}\big(\!-(X_2^2+ X_3^2)\big)|_{V_{{p,q}}^{\check\K}} = (p-q)(p-q+2) \Id_{V_{{p,q}}^{\check\K}}$. 
	We conclude that the eigenvalue of $\pi(-\check C_{(b,s)})|_{V_{{p,q}}^{\check \K}}$ is
	\begin{align*}
	\check\lambda^{(p,q)}(b,s) 
	&=  (2p(p+2n+2)+2q(q+2n))\, s^2 + 2 (p-q)(p-q+2)\, (b^2-s^2) \\
	&= \big( 4pn+4q(p+n+1) \big)s^2 + 2(p-q)(p-q+2)b^2,
	\end{align*}
	as claimed in \eqref{eq:lambda2}, concluding the proof. 
\end{proof}

\begin{remark}\label{rem:basic}
The eigenvalue $\lambda_j^{(p,q)}(a,b,c,s)$, respectively $\check{\lambda}^{(p,q)}(b,s)$, is \emph{basic}, in terms of the Riemannian submersions $\big(\Ss^{4n+3},\g_{\abcs}\big)\to \big(\Hr P^{n},\frac{1}{s^2}\gFS\big)$,
respectively $\big(\C P^{2n+1},\check{\g}_{(b,s)}\big)\to \big(\Hr P^{n},\frac{1}{s^2}\gFS\big)$, if and only if $p=q$.
Recall that if $\pi\colon (M,\g) \to (\check M,\check \g)$ is a Riemannian submersion with minimal fibers, there is a natural inclusion $\Spec(\check M,\check \g)\subset \Spec(M,\g)$ of so-called \emph{basic eigenvalues}, since lifts of Laplace eigenfunctions on $(\check M,\check \g)$ are Laplace eigenfunctions on $(M,\g)$ with the same eigenvalue, see e.g. \cite{Berard-BergeryBourguignon82,besson-bordoni}. 
Note that, from \eqref{eq:nu^(2)}, \eqref{eq:lambda}, and \eqref{eq:lambda2}, $\lambda^{(p,p)}_j(a,b,c,s)=\check\lambda^{(p,p)}(b,s)=4p(p+2n+1)s^2$, $p\geq0$, are precisely the eigenvalues of the Laplacian on $\big(\Hr P^{n},\frac{1}{s^2}\gFS\big)$.
In representation-theoretic terms, basic eigenvalues on $\Ss^{4n+3}=\G/\K$ arise from $\G$-modules $V_{p,q}^\K$ that are fixed by $\H$, see~\eqref{eq:V_pq|K}.
\end{remark}

\subsection{First eigenvalues}
We now use algebraic estimates to extract formulae for the first eigenvalue of the Laplacian on $(\Ss^{4n+3},\g_{\abcs})$, $(\R P^{4n+3},\g_{\abcs})$, and $(\C P^{2n+1},\check{\g}_{(b,s)})$ from the description of their spectra given in Lemma~\ref{lem:implicitSpec}. Through the isometries~\eqref{eq:different-parameters1}, Theorems~\ref{thm:lambda_1(a,b,c,s)} and \ref{thm:lambda_1(b,s)} below imply Theorems~\ref{thm:A} and \ref{thm:B} in the Introduction.

\begin{lemma}\label{lem:lambda^(1,1)}
Let $n\geq1$. For $a\geq b\geq c>0$, $s>0$, and $p\geq q\geq0$ satisfying
\begin{equation*}
(p,q)\notin 
\begin{cases}
(0,0),(1,0), (1,1), (2,0),& \text{ if } n\geq 2, \\
(0,0),(1,0), (1,1), (2,0),(3,0), &\text{ if }n= 1, \\
\end{cases}
\end{equation*}
we have that $\lambda_1^{(1,1)} (a,b,c,s) < \lambda_j^{(p,q)}(a,b,c,s)$ for all $1\leq j\leq p-q+1$. 
\end{lemma}

\begin{proof}
We repeatedly use formula \eqref{eq:lambda} for $\lambda_1^{(p,q)}(a,b,c,s)$; in particular, recall that $\lambda_1^{(1,1)}(a,b,c,s)=8(n+1)s^2$.
For all $p\geq 1$, $k\geq0$ and $1\leq j\leq k+1$, we have
	\begin{align*}
		\lambda_j^{(p+k,p)}(a,b,c,s) 
		&\geq (4(p+k)n+ 4p(p+k+1+n) )s^2 \\
		&\geq (4n+ 4(2+n) )s^2 = 8(n+1)s^2,
	\end{align*}
	with strict inequalities in both estimates if $k\geq1$. 
	Furthermore, for $k=0$, the second inequality is strict for $p\geq2$. 
	Similarly, for any $1\leq j\leq k+1$,
	\begin{align*}
		\lambda_j^{(k,0)}(a,b,c,s)
		&>4kns^2\geq8(n+1)s^2 
	\end{align*}
	for all $k\geq4$, and also for $k=3$ and $n\geq2$. 
	This concludes the proof.
\end{proof}

\begin{theorem}\label{thm:lambda_1(a,b,c,s)}
	Let $n\geq1$, $a\geq b\geq c>0$, and $s>0$. 
	We abbreviate 
	\begin{equation}\label{eq:lambda_1^(1,0)(2,0)(1,1)}
	\begin{aligned}
	\lambda_1^{(1,0)} &= \lambda_1^{(1,0)}\abcs =4ns^2+ 2(a^2+b^2+c^2), \\ 
	\lambda_1^{(2,0)} &= \lambda_1^{(2,0)}\abcs =8(ns^2+b^2+c^2), \\
	\lambda_1^{(1,1)} &= \lambda_1^{(1,1)}\abcs =8(n+1)s^2.
	\end{aligned}
	\end{equation}
	The smallest positive eigenvalue of the Laplace--Beltrami operator on the homogeneous space $\big(\Ss^{4n+3}, \g_{(a,b,c,s)}\big)$ is given by
	\begin{equation}\label{eq:lambda_1(a,b,c,s)}
	\lambda_1\big(\Ss^{4n+3}, \g_{(a,b,c,s)}\big)=
	\min\! \big\{ \lambda_1^{(1,0)},\; \lambda_1^{(2,0)},\; \lambda_1^{(1,1)}\big\},
	\end{equation}
	and its multiplicity is
	\begin{equation} \label{eq:mult-lambda_1(a,b,c,s)}
\begin{cases}
4(n+1)
	&\quad \text{if }\; \lambda_1^{(1,0)} < \min\! \big\{\lambda_1^{(2,0)},\; \lambda_1^{(1,1)}\big\},\\
n(2n+3)
	&\quad \text{if }\; \lambda_1^{(1,1)}< \min\! \big\{ \lambda_1^{(1,0)},\;  \lambda_1^{(2,0)} \big\},\\
(n+1)(2n+3)
	&\quad \text{if }\; \lambda_1^{(2,0)}< \min \! \big\{\lambda_1^{(1,0)},\;  \lambda_1^{(1,1)}\big\},\\
2n^2+7n+4
	&\quad \text{if }\; \lambda_1^{(1,0)}= \lambda_1^{(1,1)}< \lambda_1^{(2,0)} ,\\
2n^2+9n+7
	&\quad \text{if }\; \lambda_1^{(1,0)}= \lambda_1^{(2,0)}< \lambda_1^{(1,1)} ,\\
4n^2+8n+3
	&\quad \text{if }\; \lambda_1^{(1,1)}=\lambda_1^{(2,0)}< \lambda_1^{(1,0)} ,\\
4n^2+12n+7
	&\quad \text{if }\; \lambda_1^{(1,0)}= \lambda_1^{(2,0)}= \lambda_1^{(1,1)} .
\end{cases}
\end{equation}
Furthermore, the smallest positive eigenvalue of the Laplace--Beltrami operator on the homogeneous space $\big(\R P^{4n+3}, \g_{(a,b,c,s)}\big)$ is given by
\begin{equation}\label{eq:lambda_1(a,b,c,s)projective}
\lambda_1\big(\R P^{4n+3}, \g_{(a,b,c,s)}\big)=
\min\! \big\{ \lambda_1^{(2,0)},\; \lambda_1^{(1,1)}\big\},
\end{equation}
and its multiplicity is
\begin{equation} \label{eq:mult-lambda_1(a,b,c,s)projective}
\begin{cases}
n(2n+3)		&\quad \text{if }\; \lambda_1^{(1,1)}<  \lambda_1^{(2,0)} ,\\
(n+1)(2n+3)	&\quad \text{if }\; \lambda_1^{(2,0)}< \lambda_1^{(1,1)}\text{ and }a>b,\\
2(n+1)(2n+3)&\quad \text{if }\; \lambda_1^{(2,0)}< \lambda_1^{(1,1)}\text{ and }a=b>c,\\
3(n+1)(2n+3)&\quad \text{if }\; \lambda_1^{(2,0)}< \lambda_1^{(1,1)}\text{ and }a=b=c,\\
(2n+1)(2n+3)&\quad \text{if }\; \lambda_1^{(1,1)}=\lambda_1^{(2,0)}\text{ and }a>b,\\
(3n+2)(2n+3)&\quad \text{if }\; \lambda_1^{(1,1)}=\lambda_1^{(2,0)}\text{ and }a=b>c,\\
(4n+3)(2n+3)&\quad \text{if }\; \lambda_1^{(1,1)}=\lambda_1^{(2,0)}\text{ and }a=b=c.
\end{cases}
\end{equation}
\end{theorem}

\begin{proof}
We begin with the case of $(\Ss^{4n+3}, \g_{(a,b,c,s)})$.
Let $\lambda_{\min}(a,b,c,s)$ denote the right-hand side of \eqref{eq:lambda_1(a,b,c,s)}. 
Since the three quantities in \eqref{eq:lambda_1^(1,0)(2,0)(1,1)} are eigenvalues of $\Delta_{\g_{(a,b,c,s)}}$ by Lemma~\ref{lem:implicitSpec} and \eqref{eq:nu^(2)}, it follows that $\lambda_{\min}(a,b,c,s) \geq \lambda_1(\Ss^{4n+3}, \g_{(a,b,c,s)})$.  
 
To establish \eqref{eq:lambda_1(a,b,c,s)}, it remains to show that 
\begin{equation}\label{eq:remains}
\lambda_j^{(p,q)}(a,b,c,s)\geq \lambda_{\min}(a,b,c,s)
\quad
\text{for }
\begin{cases}
p\geq q\geq 0 \text{ with }(p,q)\neq (0,0),\\
1\leq j\leq p-q+1.
\end{cases}
\end{equation}
The case $(p,q)=(0,0)$ is excluded because it corresponds to the trivial representation, which only contributes the eigenvalue $0\in\Spec(\Ss^{4n+3}, \g_{(a,b,c,s)})$.
Lemma~\ref{lem:lambda^(1,1)} shows the above claim \eqref{eq:remains} for $n\geq2$, and also for $n=1$ provided $\lambda_j^{(3,0)}(a,b,c,s) \geq \lambda_{\min}(a,b,c,s)$. The last fact holds since, for $n=1$, \eqref{eq:nu^(k)>=} gives
\begin{align*}
\lambda_j^{(3,0)}(a,b,c,s) 
&= 12s^2 +  2\nu_j^{(3)}(a,b,c) \geq 12s^2+2(a^2+5b^2+9c^2)\\
&> 4s^2+2(a^2+b^2+c^2)=\lambda_1^{(1,0)}(a,b,c,s) \geq \lambda_{\min}(a,b,c,s).
\end{align*}
	
Regarding the multiplicity of this eigenvalue, from Lemma~\ref{lem:implicitSpec} we have that
\begin{enumerate}[$\bullet$]
\item $\pi_{1,0}$ contributes the eigenvalue $\lambda_1^{(1,0)}$ to $\spec(\Ss^{4n+3}, \g_{(a,b,c,s)})$ with multiplicity $2 d_{1,0}=4(n+1)$, since $\lambda_1^{(1,0)}(a,b,c,s)= \lambda_2^{(1,0)}(a,b,c,s)$.
\item $\pi_{2,0}$ contributes with the eigenvalue $\lambda_1^{(2,0)}$ to $\spec(\Ss^{4n+3}, \g_{(a,b,c,s)})$ with multiplicity $d_{2,0}= (n+1)(2n+3)$ if $a>b$, since $\lambda_1^{(2,0)}(a,b,c,s)< \lambda_2^{(2,0)}(a,b,c,s)$.
(Note that $\lambda_1^{(2,0)}(a,b,c,s)\leq  \lambda_1^{(1,0)}(a,b,c,s)$ forces $a>b$.)
\item $\pi_{1,1}$ contributes with the eigenvalue $\lambda_1^{(1,1)}$ to $\spec(\Ss^{4n+3}, \g_{(a,b,c,s)})$ with multiplicity $d_{1,1}= n(2n+3)$. 
\end{enumerate}
Thus, we obtain the values in the first three rows in \eqref{eq:mult-lambda_1(a,b,c,s)}. 
The remaining rows follow by summing the multiplicities of eigenvalues when they coincide. 

	Next, we consider the case of $(\R P^{4n+3},\g_{\abcs})$. Since its spectrum is the same as that of the restriction to \eqref{eq2:projective-restriction} of the Laplace operator of $(\Ss^{4n+3},\g_{\abcs})$, clearly \eqref{eq:lambda_1(a,b,c,s)projective} follows from \eqref{eq:lambda_1(a,b,c,s)}. 
	Concerning multiplicities, by Lemma~\ref{lem:implicitSpec}, 
\begin{enumerate}[$\bullet$]
	\item $\pi_{2,0}$ contributes the eigenvalue $\lambda_1^{(2,0)}$ to $\spec(\R P^{4n+3}, \g_{(a,b,c,s)})$ with multiplicity 
\begin{equation*}
\begin{cases}
d_{2,0} &\;\text{if } \lambda_1^{(2,0)}<\lambda_2^{(2,0)}, \text{ i.e., if }a>b,\\
2d_{2,0} &\;\text{if } \lambda_1^{(2,0)}=\lambda_2^{(2,0)}<\lambda_3^{(2,0)}, \text{i.e., if } a=b>c,\\
3d_{2,0} &\;\text{if } \lambda_1^{(2,0)}=\lambda_2^{(2,0)}=\lambda_3^{(2,0)}, \text{ i.e., if } a=b=c.
\end{cases}
\end{equation*}
(Note that the equivalent condition at the right on each of the rows holds since $\lambda_2^{(2,0)}=8(ns^2+a^2+c^2)$ and $\lambda_3^{(2,0)}=8(ns^2+a^2+b^2)$ by \eqref{eq:nu^(2)} and \eqref{eq:lambda}.)

\item $\pi_{1,1}$ contributes the eigenvalue $\lambda_1^{(1,1)}$ to $\spec(\R P^{4n+3}, \g_{(a,b,c,s)})$ with multiplicity $d_{1,1}= n(2n+3)$. 
\end{enumerate}
This implies \eqref{eq:mult-lambda_1(a,b,c,s)projective}, by
adding the multiplicities of eigenvalues when if coincide.
\end{proof}

\begin{remark}
The largest possible multiplicity of $\lambda_1(\Ss^{4n+3}, \g_{\abcs})$ is $4n^2+12n+7$, and it is attained when $b^2+c^2=s^2$ and $a^2=(2n+3)s^2$. 
For generic $a>b>c$ and $s$ in this situation, 
the full isometry group is $\Iso(\Ss^{4n+3},\g_{\abcs})=\Sp(n+1)\times_{\Z_2}\Sp(1)$, see \cite{ravi}. Meanwhile, the multiplicity of $\lambda_1(\Ss^{4n+3}, \gr)$ is only $4n+4$, although the full isometry group $\Iso(\Ss^{4n+3},\gr)=\Ot(4n+4)$ is much larger. This is yet another counterexample to the fact that larger isometry groups \emph{do not} necessarily correspond to larger multiplicities for the first eigenvalue, cf.~\cite[p.~181]{Berard-BergeryBourguignon82}.
The first counterexample was obtained by Urakawa~\cite{urakawa79}, who noticed that the multiplicity of $\lambda_1(\Ss^3,\g_{(\sqrt6 b,b,b)})$, $b>0$, is $7$, while that of $\lambda_1(\Ss^3,\gr)$ is only $4$.
\end{remark}

\begin{theorem}\label{thm:lambda_1(b,s)}
	Let $n\geq1$, $b>0$, and $s>0$. 
	The smallest positive eigenvalue of the Laplace--Beltrami operator on the homogeneous space $(\C P^{2n+1}, \check\g_{(b,s)})$ is
	\begin{equation}\label{eq:lambda_1(b,s)}
	\lambda_1(\C P^{2n+1}, \check\g_{(b,s)})=
	\min \!\left\{   8ns^2+16b^2,\; 8(n+1)s^2 \right\},
	\end{equation}
	and its multiplicity is
	\begin{equation} \label{eq:mult-lambda_1(b,s)}
	\begin{cases}
	(2n+3)(n+1)	&\quad \text{if }\; 2b^2< s^2,\\
	(2n+3)n 	&\quad \text{if }\; 2b^2>s^2,\\
	(2n+3)(2n+1) &\quad \text{if }\; 2b^2=s^2.
	\end{cases}
	\end{equation} 
\end{theorem}

\begin{proof}
	Let $\check\lambda_{\min}(b,s)$ denote the right-hand side of \eqref{eq:lambda_1(b,s)}. 
	Since,  by Lemma~\ref{lem:implicitSpec},
	\begin{align*}
	\check\lambda^{(2,0)}(b,s) =8ns^2 +16b^2
	\qquad\text{ and }\qquad
	\check\lambda^{(1,1)}(b,s) =8(n+1)s^2
	\end{align*}
	are eigenvalues of $\Delta_{\check\g_{(b,s)}}$, it follows that $\check\lambda_{\min}(b,s) \geq \lambda_1(\C P^{2n+1}, \check\g_{(b,s)})$.  
	
	Conversely, let us show that $\check\lambda^{(p,q)}(b,s)\geq \check\lambda_{\min}(b,s)$ for every $p\geq q\geq 0$ satisfying $p\equiv q\mod 2$ and $(p,q)\neq (0,0)$. 
	This follows since 
	\begin{align*}
	\check\lambda^{(p+2k,p)}(b,s) 
	&= (4(p+2k)n +4p(p+2k+1+n) )s^2 + 8k(k+1)b^2
	\end{align*}
	clearly satisfies $\check\lambda^{(p+2k,p)}(b,s)> \check\lambda^{(p'+2k,p')}(b,s)$ for $p>p'$, and $\check\lambda^{(p+2k,p)}(b,s)> \check\lambda^{(p+2k',p)}(b,s)$ for $k>k'$. This leaves only $\check\lambda^{(1,1)}(b,s)$ and $\check\lambda^{(2,0)}(b,s)$ as candidates for non-zero minimizers, concluding the proof of \eqref{eq:lambda_1(b,s)}.
	
	Regarding the multiplicity of this eigenvalue, from Lemma~\ref{lem:implicitSpec}, we have that
	\begin{enumerate}[$\bullet$]
	\item $\pi_{2,0}$ contributes the eigenvalue $\check\lambda^{(2,0)}(b,s)$ to $\Spec(\C P^{2n+1}, \check\g_{(b,s)})$ with multiplicity $d_{2,0}= (n+1)(2n+3)$.  	
	\item $\pi_{1,1}$ contributes the eigenvalue $\check\lambda^{(1,1)}(b,s)$ to $\Spec(\C P^{2n+1}, \check\g_{(b,s)})$ with multiplicity $d_{1,1}=  n(2n+3)$. 
	\end{enumerate}
	This gives the values in the first two rows of \eqref{eq:mult-lambda_1(b,s)}, and the third row follows by summing them.
\end{proof}

\subsection{Full spectra}\label{subsec:fullspectra}
We conclude this section providing an explicit description of the full spectrum in some particular cases, as a direct consequence of Lemma~\ref{lem:implicitSpec}.

\begin{theorem}\label{thm:SpecCP^2n+1}
For $n\geq 1$, we have that 
\begin{align}\label{eq:spec(S^4n+3,h(t,t,t))}
\Spec(\Ss^{4n+3}, {\mathbf h}(t,t,t))
	&= \bigcup_{\substack{0\leq l\leq k\\ k\equiv l\mod 2 } }  \Big\{ \!\underbrace{{\mu}_{k,l}(t),\dots, {\mu}_{k,l}(t) }_{{ (l+1) } {m}_{k,l}} \Big\},\\
\label{eq:spec(RP^4n+3,h(t,t,t))}
\Spec(\R P^{4n+3}, {\mathbf h}(t,t,t))
	&= \bigcup_{\substack{0\leq l\leq k\\ k\equiv l\equiv 0\mod 2 } }  \Big\{ \!\underbrace{{\mu}_{k,l}(t),\dots, {\mu}_{k,l}(t) }_{{ (l+1) } {m}_{k,l}} \Big\},\\
    \label{eq:spec(CP^2n+1,check-h(t))}
\Spec(\C P^{2n+1},\check{\mathbf h}(t))
	&= \bigcup_{\substack{0\leq l\leq k \\ k\equiv l\equiv 0\mod2 } }  \Big\{ \!\underbrace{{\mu}_{k,l}(t),\dots, {\mu}_{k,l}(t) }_{{m}_{k,l}} \Big\},
\end{align}
where
\begin{align}
{\mu}_{k,l}(t) &=k(k+4n+2) + l(l+2)  \left(\frac{1}{t^2}-1\right), \label{eq:check-mu_kl} 
\\
{m}_{k,l} &=\sum_{\substack{(p,q)\in\Z^2: \, p\geq q\geq 0,\\  p+q=k, \;p-q=l  }} d_{p,q}.
\label{eq:mult-mu_kl}
\end{align}
\end{theorem}

\begin{proof}
From \eqref{eq:different-parameters1}, we have the isometries
$\mathbf h(t,t,t)\cong {\g}_{\left((\sqrt2 t)^{-1},(\sqrt2 t)^{-1},(\sqrt2 t)^{-1},1\right)}$ for metrics on $\Ss^{4n+3}$ and $\R P^{4n+3}$, and 
$\check{\mathbf h}(t)\cong \check{\g}_{\left((\sqrt2 t)^{-1},1\right)}$ for metrics on $\C P^{2n+1}$.
Lemma~\ref{lem:implicitSpec} ensures that any eigenvalue in $\Spec(\Ss^{4n+3}, {\mathbf h}(t,t,t))$ is as in \eqref{eq:lambda}, i.e.,
\begin{equation*}
\begin{aligned}
\lambda^{(p,q)} \big(\tfrac{1}{\sqrt2 t}, \tfrac{1}{\sqrt2 t}, \tfrac{1}{\sqrt2 t},1\big) 
&:= 4pn+{ 4}q(p+n+1)  +  2 \nu_j^{(p-q)}(\tfrac{1}{\sqrt2 t},\tfrac{1}{\sqrt2 t},\tfrac{1}{\sqrt2 t})
\\
&\;= 4pn+4q(p+n+1)  + (p-q)(p-q+2)\tfrac{1}{t^2}
\\
&\;= (p+q)(p+q+4n+2)+ (p-q)(p-q+2)\left(\tfrac{1}{t^2} - 1\right)
\end{aligned}
\end{equation*}
for integers $p,q$ with $p\geq q\geq 0$. 
We have used that $\nu_j^{(k)}(a,a,a)=k(k+2)a^2$ by \eqref{eq:nu_j^k(a,a,a)}. 
The same holds for $\Spec\big(\R P^{4n+3}, {\mathbf h}(t,t,t)\big)$,
if we further assume $p-q$ is even.
Similarly, Lemma~\ref{lem:implicitSpec} gives that $\Spec(\C P^{2n+1},\check{\mathbf h}(t))$ is the collection of eigenvalues $\check{\lambda}^{(p,q)}(\frac{1}{\sqrt2 t},1) = {\lambda}^{(p,q)}\big(\tfrac{1}{\sqrt2 t}, \tfrac{1}{\sqrt2 t}, \tfrac{1}{\sqrt2 t},1\big) $ for integers $p,q$ with $p\geq q\geq 0$ and $p-q$ even. 
Writing $p+q=k$ and $p-q=l$, we obtain that $0\leq l\leq k$, $k\equiv l\mod2$, ${\lambda}^{(p,q)}\big(\tfrac{1}{\sqrt2 t}, \tfrac{1}{\sqrt2 t}, \tfrac{1}{\sqrt2 t},1\big) = {\mu}_{k,l}(t)$, and $k\equiv l\equiv 0\mod 2$ if $p$ and $q$ are both even, proving \eqref{eq:spec(S^4n+3,h(t,t,t))} and \eqref{eq:spec(CP^2n+1,check-h(t))}.  
The claimed multiplicity contribution \eqref{eq:mult-mu_kl} of $\mu_{k,l}(t)$ to both spectra follows also from Lemma~\ref{lem:implicitSpec}, concluding the proof.
\end{proof}

Differently from the above situation, the full spectrum of $(\Ss^{4n+3},\mathbf h(t_1,t_2,t_3))$, or $(\Ss^{4n+3},\g_{\abcs})$ by means of the isometries in Remark~\ref{rem:isometries}, \emph{cannot} be explicitly described with our methods, since 
the eigenvalues $\lambda_j^{(p,q)}(a,b,c,s)$ are only computed in terms of the eigenvalues $\nu_j^{(k)}(a,b,c)$ of the Laplacian on $(\Ss^3, \g_{\abc})$, cf.~\eqref{eq:lambda} and \eqref{eq:lambda2}. A closed formula for all $\nu_j^{(k)}(a,b,c)$, hence for all $\lambda_j^{(p,q)}(a,b,c,s)$, would be highly desirable, but seems to remain out of the reach of current techniques.

Nevertheless, with the aid of \emph{further symmetries}, we can describe the full Laplace spectrum in some special cases. For instance, 
we may enlarge the symmetry group from $\Sp(n+1)$ to $\Sp(n+1)\mathsf U(1)$. This corresponds to requiring that at least two of the parameters $a,b,c$ coincide, say $b=c$, which, by \cite[Lem.~3.1]{Lauret-SpecSU(2)}, implies that
\begin{equation}\label{eq:nu_j^k(a,b,b)}
		\nu_j^{(k)}(a,b,b) = \big(k-2(j-1)\big)^2\, a^2 + 2\big((2j-1)k-2(j-1)^2\big)\, b^2.
\end{equation}
This yields an explicit expression for all $\lambda_j^{(p,q)}(a,b,b,s)$ via \eqref{eq:lambda}, that can be used to determine the full Laplace spectrum of the $\SU(2n+2)$-invariant metrics
\begin{equation}
\begin{aligned}
\label{eq:bf_g(t)=h(t,1,1)}
(\Ss^{4n+3},\mathbf g(t)) &\cong (\Ss^{4n+3},\mathbf h(t,1,1)) \cong \Big(\Ss^{4n+3},\g_{\big(\frac{1}{\sqrt{2}t}, \frac{1}{\sqrt{2}}, \frac{1}{\sqrt{2}},1 \big)} \Big),
\\
(\R P^{4n+3},\mathbf g(t)) &\cong (\R P^{4n+3},\mathbf h(t,1,1)) \cong \Big(\R P^{4n+3},\g_{\big(\frac{1}{\sqrt{2}t}, \frac{1}{\sqrt{2}}, \frac{1}{\sqrt{2}},1 \big)}  \Big), 
\end{aligned}
\end{equation}
for any $t>0$.

\begin{theorem}\label{thm:Spec(g(t))}
For $d=4n+3$ with $n\geq1$, we have that 
\begin{align}
\label{eq:spec(S^d,g(t))}
\Spec(\Ss^{d},\mathbf g(t))&=\bigcup_{\substack{0\leq l\leq k, \\ k\equiv l\!\!\!\mod 2} }  \Big\{ \!\underbrace{\eta_{k,l}(t),\dots, \eta_{k,l}(t) }_{\widetilde m_{k,l}} \Big\},
\\
\label{eq:spec(RP^d,g(t))}
\Spec(\R P^{d},\mathbf g(t))&=\bigcup_{\substack{0\leq l\leq k, \\ k\equiv l\equiv 0\!\!\!\mod 2} }  \Big\{ \!\underbrace{\eta_{k,l}(t),\dots, \eta_{k,l}(t) }_{\widetilde m_{k,l}} \Big\},
\end{align}
where 
\begin{align}
\eta_{k,l}(t)
	&=k(k+d-1) + l^2\, \left(\frac{1}{t^2}-1\right),\\
\label{eq:mu_kl}
\widetilde m_{k,l}
	&=\sum_{\substack{(p,q,j)\in\Z^3:\, p\geq q\geq 0,\\ 1\leq j\leq p-q+1, \; p+q=k, \\ \;p-q-2(j-1)=\pm l  }} d_{p,q}.
\end{align}
\end{theorem}
	
\begin{proof}
From \eqref{eq:lambda}, \eqref{eq:nu_j^k(a,b,b)}, and \eqref{eq:bf_g(t)=h(t,1,1)}, see also~Remark~\ref{rem:isometries}, the eigenvalues in $\Spec(\Ss^{d},\mathbf g(t))$ are of the form 
\begin{align*}
\lambda_j^{(p,q)}\big(\tfrac{1}{\sqrt{2}t}, \tfrac{1}{\sqrt{2}}, \tfrac{1}{\sqrt{2}},1\big) &= \big(4pn+{ 4}q(p+n+1) \big) +  2 \nu_j^{(p-q)}\big(\tfrac{1}{\sqrt{2}t}, \tfrac{1}{\sqrt{2}}, \tfrac{1}{\sqrt{2}}\big)\\
&=(d-3)p+ q(4p+d+1) +2(2j-1)(p-q) \\
&\quad -4(j-1)^2+ \big(p-q-2(j-1)\big)^2\, \tfrac{1}{t^2} \\
&=(p+q)(d-1+p+q) +  \big(p-q-2(j-1)\big)^2\, \left(\tfrac{1}{t^2}-1\right),
\end{align*}
which coincides with $\eta_{p+q,|p-q-2(j-1)|}(t)$.
For integers $0\leq l\leq k$ with $k-l$ even, Lemma~\ref{lem:implicitSpec} implies that $\eta_{k,l}(t)$ contributes to $\Spec(\Ss^{d},\mathbf g(t))$ with multiplicity \eqref{eq:mu_kl}. The statements regarding $\R P^d$ follow by the same arguments, with $p-q$ even.
\end{proof}

\begin{remark}
Although the full spectrum of the Laplacian on $(\Ss^{d},\mathbf g(t))$ had not been previously described in odd dimensions $d\geq 5$, partial results by Tanno~\cite[Lem.~4.1]{tanno1}, see also \cite[\S{}5]{bp-calcvar}, were sufficient to explicitly compute $\lambda_1(\Ss^d,\mathbf g(t))$.

We only analyze dimensions $d\equiv 3 \mod 4$ in Theorem~\ref{thm:Spec(g(t))} for simplicity, as the description of the entire $\Spec(\Ss^d,\mathbf g(t))$ for such $d$ follows directly from Lemma~\ref{lem:implicitSpec} and \eqref{eq:nu_j^k(a,b,b)}. The same methods in Section~\ref{sec:homspectra} can be used to compute $\Spec(\Ss^d,\mathbf g(t))$ in the remaining cases, using $\G=\SU\big(\frac{d+1}{2}\big)$ and $\K=\SU\big(\frac{d-1}{2}\big)$, see~\cite{blp-new}.
\end{remark}

\begin{example}
The $k$th eigenvalue of the Laplacian on $(\C P^{2n+1},\gFS)$ and $(\Ss^{d},\gr)$ can be read
from Theorems~\ref{thm:SpecCP^2n+1} and \ref{thm:Spec(g(t))} respectively, by setting $t=1$ in \eqref{eq:check-mu_kl} and \eqref{eq:mu_kl}, recovering the well-known formulae
\begin{equation*}
\lambda_k(\Ss^d,\gr)=k(k+d-1) \quad \text{and }\quad \lambda_k(\C P^{2n+1},\gFS)=4k(k+2n+1).
\end{equation*}
Recall that, since these are symmetric spaces, the above Laplace eigenvalues can be computed with Freudenthal's formula  \eqref{eq:Casimirscalar}.
Moreover, it can be checked combinatorially that the multiplicity of the $k$th eigenvalue $\lambda_k(\Ss^{d},\gr)$ is equal to
\begin{equation}\label{eq:multsphere}
\sum_{\substack{p+q=k \\ p\geq q\geq 0}} (p-q+1)\,  d_{p,q} = \binom{k+d}{d}- \binom{k+d-2}{d},
\end{equation}
where we use the convention that $\binom{a}{b}=0$ if $a<b$.
\end{example}

\section{Spectral uniqueness}\label{sec:specunique}

In this section, we prove that the spectrum of the Laplace--Beltrami operator distinguishes homogeneous CROSSes up to isometries, proving Theorem~\ref{thm:rigidity} in the Introduction. We begin showing that two isospectral $\Sp(n+1)$-invariant metrics on $\Ss^{4n+3}$ or $\R P^{4n+3}$ must be isometric.

\subsection{\texorpdfstring{Spectral uniqueness of homogeneous metrics on $\Ss^{4n+3}$}{Spectral uniqueness on spheres}}
Given real numbers $a\geq b\geq c>0$, consider the elementary symmetric polynomials in their squares,
\begin{equation}\label{eq:sigma_i}
	\begin{aligned}
		\sigma_1&:=\sigma_1\big(a^2,b^2,c^2\big)=a^2+b^2+c^2, \\
		\sigma_2&:=\sigma_2\big(a^2,b^2,c^2\big)=a^2b^2+a^2c^2+b^2c^2, \\
		\sigma_3&:=\sigma_3\big(a^2,b^2,c^2\big)=a^2b^2c^2.
	\end{aligned}
\end{equation}
In the sequel, we repeatedly use the elementary fact that
\begin{equation}\label{eq:sigma_idetermine}
	(\sigma_1,\sigma_2,\sigma_3) \quad\text{ determines }\quad\abc. 
\end{equation}
Indeed, $ x^3-\sigma_1x^2+\sigma_2x-\sigma_3=(x-a^2)(x-b^2)(x-c^2)$ determines $a^2,b^2,c^2$ up to permutations, hence $\abc$ are completely determined since $a\geq b\geq c>0$.

Recall that, by Lemma~\ref{lem:implicitSpec}, eigenvalues in $\Spec(\Ss^{4n+3}, \g_{\abcs})$ are of the form
\begin{equation*}
	\lambda_j^{(p,q)}(a,b,c,s) = 4\big((p+q)n+q(p+1)\big)s^2 + 2\nu_{j}^{(p-q)}(a,b,c)
\end{equation*}
for some $p\geq q\geq0$ and $1\leq j\leq p-q+1$, where $\{\nu_{j}^{(k)}(a,b,c):1\leq j\leq k+1\}$ is the spectrum of the operator \eqref{eq:tau_k}.
We assume that $\nu_{1}^{(k)}\leq\dots\leq \nu_{k+1}^{(k)}$, thus $\lambda_1^{(p,q)} \leq \dots\leq \lambda_{p-q+1}^{(p,q)}$. 

\begin{lemma}\label{lem:nu_1^4}
	The smallest eigenvalue of $\tau_{4}(-a^2X_1^2 -b^2X_2^2 -c^2X_3^2)$ on $V_{\tau_4}$, see \eqref{eq:tau_k}, is given by
	\begin{equation*}
		\nu_1^{(4)}(a,b,c) = 8(a^2+b^2+c^2) -8 \sqrt{a^4+b^4+c^4 -a^2b^2-a^2c^2-b^2c^2}.
	\end{equation*}
	Moreover, the multiplicity of this eigenvalue is $1$ if and only if $a>b$.
\end{lemma}

\begin{proof}
	From \cite[Lem.~3.1]{Lauret-SpecSU(2)}, the matrix representing $\tau_{4}(-a^2X_1^2 -b^2X_2^2 -c^2X_3^2)$ 
	is similar to a block diagonal matrix $\diag(\tau_4^1,\tau_4^2)$, with blocks given by
	\begin{align*}
		\tau_4^1&=\begin{pmatrix}
			16a^2 + 4(b^2 + c^2) & 2(b^2-c^2) & 0 \\
			12(b^2-c^2) & 12(b^2+c^2) & 12(b^2-c^2)\\
			0 & 2(b^2-c^2) & 16a^2 + 4(b^2 + c^2)
		\end{pmatrix}, \\
		\tau_4^2&=\begin{pmatrix}
			4a^2 + 10(b^2 + c^2) & 6(b^2-c^2) \\
			6(b^2-c^2) & 4a^2 + 10(b^2 + c^2)\\
		\end{pmatrix}.
	\end{align*}
	Note that, although \eqref{eq:tau_k} is self-adjoint, the above $\tau_4^1$ is not symmetric because 
	the basis we used to represent it as a matrix is only orthogonal, and not orthonormal.
	The eigenvalues of $\tau_4^2$ are $4a^2+16b^2+4c^2$ and $4a^2+4b^2+16c^2$, while the eigenvalues of $\tau_4^1$ are 
	$16a^2+4b^2+4c^2$, and 
	$8(a^2+b^2+c^2) \pm 8\sqrt{a^4+b^4+c^4 -a^2b^2-a^2c^2-b^2c^2}.$
	The minimum $\nu_1^{(4)}(a,b,c) $ of these five numbers is as claimed in the statement, since
	\begin{equation*}
		8(a^2+b^2+c^2) -8\sqrt{a^4+b^4+c^4 -a^2b^2-a^2c^2-b^2c^2}\leq 4a^2+4b^2+16c^2,
	\end{equation*}
	as easily shown with routine computations. 
	Since equality in the above holds if and only if $a=b$, the assertion regarding multiplicity also follows.
\end{proof}

We set 
\begin{equation}\label{eq:beta}
	\beta(a,b,c) =\sigma_1- \sqrt{\sigma_1^2-3\sigma_2}.
\end{equation}
Lemma~\ref{lem:nu_1^4} tells us that $\nu_1^{(4)}(a,b,c) =8\beta(a,b,c)$ via \eqref{eq:sigma_i}.
The next estimates will be useful later.

\begin{lemma}\label{lem:beta}
	For $a\geq b\geq c>0$, we have that
	\begin{equation*}
		b^2+c^2<
		\beta(a,b,c) \leq \tfrac{3}2(b^2+c^2).
	\end{equation*}
	Furthermore, the second inequality above is an equality if and only if $b=c$. 
\end{lemma}

\begin{proof}
	From \eqref{eq:nu^(k)>=}, we get that $\beta(a,b,c)= \frac18 \nu_1^{(4)}\geq \frac18 (8b^2+16c^2)> b^2+c^2$. 
	We next prove the inequality at the right. 
	By \eqref{eq:beta}, the assertion is equivalent to $\sigma_1- \frac{3}2(b^2+c^2)<\sqrt{\sigma_1^2-3\sigma_2}$.
	Since the left-hand side is nonnegative, squaring both sides, this becomes equivalent to 
	\begin{equation*}
		- 3\sigma_1(b^2+c^2) + \tfrac{9}4(b^2+c^2)^2\leq -3\sigma_2.
	\end{equation*}
	By replacing $\sigma_1$ and $\sigma_2$ as in \eqref{eq:sigma_i} and simple manipulations, one has that the above condition is equivalent to $4b^2c^2\leq  (b^2+c^2)^2$, which clearly holds, with equality if and only if $b=c$.
\end{proof}

\begin{lemma}\label{lem:vol-scal}
	The volume and scalar curvature of $(\Ss^{4n+3}, \g_{(a,b,c,s)})$ are given by
	\begin{align}
		\Vol(\Ss^{4n+3}, \g_{\abcs}) &= 
		\frac{\Vol(\Ss^{4n+3}, \gr)}{2\sqrt{2\sigma_3} \,s^{4n}} = \dfrac{2\pi^{2n+2}}{(2n+1)!}   \frac{1}{2\sqrt{2\sigma_3} \,s^{4n}}, \label{eq:vol} \\
		\scal(\Ss^{4n+3}, \g_{\abcs})  &= 16n(n+2) s^2+16\sigma_1 -\frac{2n\sigma_2s^4}{\sigma_3} -\frac{4\sigma_2^2}{\sigma_3}. \label{eq:scal}
	\end{align}
\end{lemma}

\begin{proof}
	The proof of \eqref{eq:vol} is left to the reader. 
	(In this article, we will only use the fact that $\Vol(\Ss^{4n+3}, \g_{\abcs})$ depends only on $s$ and $\sigma_3$, which is well-known.)
	We next prove \eqref{eq:scal} using the Gray--O'Neill formula~\eqref{eq:scal4n3}. Recalling the isometries \eqref{eq:different-parameters2}, and Newton's identity $\sigma_2^2-2\sigma_1\sigma_3=a^4b^4+a^4c^4+b^4c^4$, we have
		\begin{align*}
			\scal(\Ss^{4n+3}, \g_{\abcs})  &= \scal\!\big(\Ss^{4n+3}, \tfrac{1}{s^2}\mathbf h(\tfrac{s}{\sqrt2 a},\tfrac{s}{\sqrt2 b},\tfrac{s}{\sqrt2 c})\big) \\
			&= s^2 \scal\!\big(\Ss^{4n+3},\mathbf h(\tfrac{s}{\sqrt2 a},\tfrac{s}{\sqrt2 b},\tfrac{s}{\sqrt2 c})\big) \\
			&=16n(n+2)s^2 +8\left( a^2+b^2+c^2 \right)\\
			&\quad  -4\left(\frac{b^2c^2}{a^2}+\frac{a^2c^2}{b^2}+\frac{a^2b^2}{c^2} \right) -2ns^4\left(\frac{1}{a^2}+\frac{1}{b^2}+\frac{1}{c^2}\right)\\
			&= 16n(n+2)s^2 +8\sigma_1 -4\frac{\sigma_2^2-2\sigma_1\sigma_3}{\sigma_3}-2ns^4\frac{\sigma_2}{\sigma_3}\\
			&=16n(n+2) s^2+16\sigma_1 -\frac{2n\sigma_2s^4}{\sigma_3} -\frac{4\sigma_2^2}{\sigma_3}.\qedhere
		\end{align*}
\end{proof}

\begin{lemma}\label{lem:invariants}
	Positive real numbers $a,b,c,s$ satisfying $a\geq b\geq c$ are determined by the volume \eqref{eq:vol}, the scalar curvature \eqref{eq:scal}, and either
\begin{enumerate}[\rm (i)]
\item the quantities $\lambda_1^{(1,0)}(a,b,c,s)$ and $\lambda_1^{(1,1)}(a,b,c,s)$;
\item the quantities $\lambda_1^{(1,1)}(a,b,c,s)$, $\lambda_1^{(2,0)}(a,b,c,s)$, and $\lambda_1^{(4,0)}(a,b,c,s)$. 
\end{enumerate}    
\end{lemma}

\begin{proof}
Let us begin with (i). Since $\lambda_1^{(1,1)}=8(n+1)s^2$, the value of $s>0$ is easily determined.
	The volume then determines $\sigma_3$, and $\lambda_1^{(1,0)}=4ns^2+2\sigma_1$ determines $\sigma_1$. 
	Moreover, $\sigma_2$ is determined by the scalar curvature, since \eqref{eq:scal} gives
	\begin{equation*}
		\frac{4}{\sigma_3} \sigma_2^2 +\frac{2ns^4}{\sigma_3}\sigma_2 +\big(\!\scal(\Ss^{4n+3}, \g_{\abcs})-16n(n+2) s^2-16\sigma_1\big)=0,
	\end{equation*}
	and at most one of the roots of this quadratic polynomial in $\sigma_2$ is positive, because the coefficients of $\sigma_2^2$ and $\sigma_2$ are both positive. 
	Thus, $(\sigma_1,\sigma_2,\sigma_3,s)$ are determined, and hence so are $(a,b,c,s)$ by \eqref{eq:sigma_idetermine}. 

    Let us now turn to (ii).
	Just like in the previous case, $\Vol(\Ss^{4n+3}, \g_{\abcs})$ and $\lambda_1^{(1,1)}$ determine $s$ and $\sigma_3$.
	Furthermore, $\lambda_1^{(2,0)} =8ns^2+8(b^2+c^2)$ determines $b^2+c^2$. 
	From \eqref{eq:lambda} and Lemma~\ref{lem:nu_1^4}, we have 
	$ \lambda_1^{(4,0)} =16ns^2 + 2\nu_1^{(4)} =16ns^2+16\beta(a,b,c),$
	so $\beta:=\beta(a,b,c)$ 
	is also determined.
	
	Thus far, we know the (positive) values of the quantities $s$, $\sigma_3=a^2b^2c^2$, $b^2+c^2$, and $\beta$, and wish to use them to uniquely determine the values of $a\geq b\geq c>0$. We will see that there are two possible options for $\abcs$, and one of them will be excluded using the value of the scalar curvature.	
	From \eqref{eq:beta}, we have that 
	\begin{equation*}
		3\sigma_2- 2\sigma_1 \beta+ \beta^2=0.
	\end{equation*}
	Substituting $\sigma_2=a^2(b^2+c^2)+ \frac{\sigma_3}{a^2}$, this equation can be written as 
	\begin{equation}\label{eq:quadraticequation}
		Aa^4-Ba^2+C=0,
	\end{equation}
	where
	\begin{align}\label{eq:ABCbeta}
		A&= 3(b^2+c^2) - 2 \beta,  &
		B&= \beta\left(2(b^2+c^2) -\beta  \right) ,&
		C&= 3\sigma_3 .
	\end{align}
	Note that $A$, $B$, and $C$ are already determined, since they can be written in terms of the known values $b^2+c^2$, $\sigma_3$, and $\beta$. 	
	Clearly, $C>0$.
	Lemma~\ref{lem:beta} implies that $B>0$ and $A \geq0$, with equality if and only if $b=c$. 
	Let us assume $A>0$, otherwise all parameters can be easily (uniquely) determined using that $b=c$.

	We know that the equation $Ax^2-Bx+C=0$ must have at least one real root, so its discriminant is nonnegative, that is, 
	\begin{equation}\label{eq:discriminant}
		B^2-4AC\geq0. 
	\end{equation}
	Moreover, since $A,B,C$ are all positive, the equation in \eqref{eq:quadraticequation} with respect to the variable $a$ has two positive solutions $a_1<a_2$ satisfying 
	\begin{equation*}
		a_1^2=\frac{B- \sqrt{B^2-4AC}}{2A}, \qquad  \text{and}  \qquad
		a_2^2=\frac{B+ \sqrt{B^2-4AC}}{2A}.
	\end{equation*}
	Setting $a=a_i>0$, $i=1,2$, since we know the values of $b^2+c^2$ and $b^2c^2=\sigma_3/a^2$, it follows that $b>0$ and $c>0$, satisfying $b>c$, become determined. Denote their values by $b_i$ and $c_i$, $i=1,2$, according to the choice $a=a_i$, $i=1,2$. If one of these choices $i=1,2$ violates the inequalities $a_i\geq b_i > c_i>0$, then $\abcs$ is determined, since $\abc$ must then be equal to $(a_i,b_i,c_i)$ for the \emph{other} choice $i=1,2$.
	Thus, suppose that $a_i\geq b_i > c_i>0$ for both $i=1,2$. 
	We will show that $\scal\!\big(\Ss^{4n+3}, \g_{(a_1,b_1,c_1,s)}\big)>\scal\!\big(\Ss^{4n+3}, \g_{(a_2,b_2,c_2,s)}\big)$, which implies that only one of $(a_i,b_i,c_i,s)$ for $i=1,2$ matches all five known quantities from the statement. 
		
	From \eqref{eq:scal}, using that $s$, $b^2+c^2$, $\sigma_3$, $A$, $B$ and $C$ are determined, we compute
	\begin{align*}
			F&:=\frac{\scal\!\big(\Ss^{4n+3}, \g_{(a_2,b_2,c_2,s)}\big) - \scal\!\big(\Ss^{4n+3}, \g_{(a_1,b_1,c_1,s)}\big)}{a_2^2-a_1^2}\\
			&=16 - \frac{2ns^4}{\sigma_3}(b^2+c^2) -2ns^4 \left(\frac{1}{a_2^2}- \frac{1}{a_1^2}\right)\frac{1}{a_2^2-a_1^2}
			\\&\quad
			-\frac{4}{\sigma_3}\!\left(\!(b^2+c^2)+ \frac{\sigma_3}{(a_2^2-a_1^2)}\!\left(\frac{1}{a_2^2}- \frac{1}{a_2^2}\right)\!\right)\!\left(\!(a_2^2+a_1^2)(b^2+c^2)+ \sigma_3\left(\frac{1}{a_2^2} + \frac{1}{a_2^2}\right)\!\right)\\
			&=16 - \frac{2ns^4(b^2+c^2)}{\sigma_3}
			+2ns^4 \frac{1}{a_1^2a_2^2}	\\
            &\quad
            -\frac{4}{\sigma_3} 
			\left( b^2+c^2- \frac{\sigma_3}{a_1^2a_2^2}\right)
			\left( (a_2^2+a_1^2)(b^2+c^2)+ \sigma_3 \frac{a_1^2+a_2^2}{a_1^2a_2^2} \right)\\
			&=16 - \frac{6ns^4(b^2+c^2)}{C}
			+2ns^4 \frac{A}{C}
			-\frac{12}{C} 
			\left( 3(b^2+c^2)- A\right)
			\left( \frac{3B}A(b^2+c^2)+  B \right).
		\end{align*}
	In the last step, we used that $C=3\sigma_3$ and the relations $a_1^2+a_2^2=\frac{B}{A}$ and $a_1^2a_2^2=\frac{C}{A}$ between roots and coefficients of a quadratic equation. 
	Basic manipulations give
		\begin{align*}
			F
			&= - \frac{2ns^4}{C} \Big(3(b^2+c^2)-A\Big)
			- \frac{4}{3AC} \Big( \left(9(b^2+c^2)^2- {A^2} \right) B -12AC \Big)
			\\
			&= - \frac{4n\beta s^4}{C}
			- \frac{4}{3AC} \Big( 2\beta \big(6(b^2+c^2)-2\beta\big) B -12AC \Big),
		\end{align*}
	where the last step uses \eqref{eq:ABCbeta}. 
To prove that $F<0$, since $s$, $\beta$, $A$, $B$, and $C$ are all positive, it remains to show that $G:=2\beta \big(6(b^2+c^2)-2\beta\big) B -12AC $ is positive. 
Since $ B=\beta\left(2(b^2+c^2) -\beta  \right) $ by \eqref{eq:ABCbeta}, we have that
\begin{equation*}
G=2\beta \big(6(b^2+c^2)-3\beta+\beta\big) B -12AC = 6 \big(B+\tfrac13 \beta^2\big) B -12AC >6 \big(B^2 -2AC\big),
\end{equation*}
so the proof is complete by \eqref{eq:discriminant}. 
\end{proof}

\begin{theorem}\label{thm:rigidityS^4n+3}
	Two isospectral $\Sp(n+1)$-invariant metrics on $\Ss^{4n+3}$ are isometric. 
\end{theorem}

\begin{proof}
	In order to show that $\spec(\Ss^{4n+3}, \g_{\abcs})$ determines $\abcs$, we first recall that since $(\Ss^{4n+3}, \g_{\abcs})$ is homogeneous, the first two heat invariants determine $\Vol(\Ss^{4n+3}, \g_{(a,b,c,s)})$ and $\scal(\Ss^{4n+3}, \g_{(a,b,c,s)})$, see e.g.~\cite[Chap.~III, E.IV]{BerGauMaz}. Furthermore, by Lemma~\ref{lem:invariants}, it suffices to show that either $\lambda_1^{(1,0)}(a,b,c,s)$ and $\lambda_1^{(1,1)}(a,b,c,s)$ are also determined by the spectrum.  
	
    	From Theorem~\ref{thm:lambda_1(a,b,c,s)}, there are 7 distinct possible values for the multiplicity of the first eigenvalue $\lambda_1(\Ss^{4n+3}, \g_{(a,b,c,s)})$, see \eqref{eq:mult-lambda_1(a,b,c,s)}, thus the spectrum reveals which among $\lambda_1^{(1,0)}$, $\lambda_1^{(2,0)}$, or $\lambda_1^{(1,1)}$ realizes the minimum in \eqref{eq:lambda_1(a,b,c,s)}. The proof is therefore naturally divided in 7 cases, corresponding to the 7 rows in \eqref{eq:mult-lambda_1(a,b,c,s)}.
	We proceed with a case-by-case analysis.
	
	\smallskip
	
	\noindent\textbf{Row 1:} $\lambda_1^{(1,0)}< \min\!\big\{ \lambda_1^{(2,0)}, \lambda_1^{(1,1)}\big\}$.
	The quantity $\lambda_1^{(1,0)}$ is determined, since it is equal to $\lambda_1(\Ss^{4n+3}, \g_{(a,b,c,s)})$, so it suffices to determine $\lambda_1^{(1,1)}$ by Lemma~\ref{lem:invariants}. This is achieved searching for it among larger eigenvalues in the spectrum. 
	
	Let us determine the second eigenvalue $\lambda_2(\Ss^{4n+3}, \g_{(a,b,c,s)})$ under the current assumptions.
	Note that $\lambda_1(\Ss^{4n+3}, \g_{(a,b,c,s)})= \lambda_1^{(1,0)}= \lambda_2^{(1,0)}$, thus the second eigenvalue must come from $\pi_{p,q}$ with $(p,q)\notin\{(0,0),(1,0)\}$, that is, 
	\begin{equation*}
		\lambda_2(\Ss^{4n+3}, \g_{(a,b,c,s)}) = \!\min_{\substack{ p\geq q\geq 0\\ (p,q)\notin \{(0,0),(1,0)\}}}  \lambda_1^{(p,q)}(a,b,c,s).
	\end{equation*}
	Lemma~\ref{lem:lambda^(1,1)} implies that $\lambda_2(\Ss^{4n+3}, \g_{(a,b,c,s)})=\min\big\{\lambda_1^{(2,0)},\lambda_1^{(1,1)}\big\}$. 
	Note that $(p,q)=(3,0)$ when $n=1$ is excluded, since, by \eqref{eq:nu^(k)>=},
\begin{align*}
\lambda_1^{(3,0)} &= 12ns^2+ 2\nu_{1}^{(3)}(a,b,c)\\
&\geq 12ns^2+ 2(a^2+5b^2+9c^2)\\
&>8ns^2+8(b^2+c^2)\\
&= \lambda_1^{(2,0)}.
\end{align*}
	In order to determine its multiplicity, we must take into account that $\lambda_2^{(2,0)}$ and $\lambda_3^{(2,0)}$ may also contribute if they coincide with $\lambda_1^{(2,0)}$.
	Analyzing each possibility, one obtains the following table:
	\begin{equation}\label{eq:table2ndeigenvalue}
		\begin{array}{lrl}
			\lambda_2(\Ss^{4n+3}, \g_{(a,b,c,s)}) & \multicolumn{1}{c}{\text{multiplicity}} & \text{condition} \\
			\hline 
			\rule{0pt}{13pt}
			\lambda_1^{(1,1)} & n(2n+3)& \lambda_1^{(2,0)}>\lambda_1^{(1,1)}\\
			\lambda_1^{(2,0)} & (n+1)(2n+3) & \lambda_1^{(2,0)}<\min\{ \lambda_1^{(1,1)} , \lambda_2^{(2,0)} \}\\
			\lambda_1^{(2,0)} & 2(n+1)(2n+3) & \lambda_1^{(2,0)}= \lambda_2^{(2,0)}\!\!< \min\{ \lambda_1^{(1,1)}\! ,  \lambda_3^{(2,0)}\}\\
			\lambda_1^{(2,0)} & 3(n+1)(2n+3) & \lambda_1^{(2,0)}= \lambda_2^{(2,0)}= \lambda_3^{(2,0)}<\lambda_1^{(1,1)} \\
			\lambda_1^{(2,0)}=\lambda_1^{(1,1)} & (2n+1)(2n+3) & \lambda_1^{(2,0)}=\lambda_1^{(1,1)} <\lambda_2^{(2,0)} \\
			\lambda_1^{(2,0)}=\lambda_1^{(1,1)} & (3n+2)(2n+3) & \lambda_1^{(2,0)}= \lambda_2^{(2,0)}=\lambda_1^{(1,1)}<\lambda_3^{(2,0)} \\
			\lambda_1^{(2,0)}=\lambda_1^{(1,1)} & (4n+3)(2n+3) & \lambda_1^{(2,0)}= \lambda_2^{(2,0)}= \lambda_3^{(2,0)}=\lambda_1^{(1,1)}
		\end{array}
	\end{equation}
	As the multiplicities in the rows of \eqref{eq:table2ndeigenvalue} are all distinct, we \emph{hear} the expression for $\lambda_2(\Ss^{4n+3}, \g_{(a,b,c,s)})$. Thus, the cases in rows 1 and 5--7 are settled, since $\lambda_1^{(1,1)}$ is determined.
	In row 4, i.e., if $\lambda_1^{(2,0)}= \lambda_2^{(2,0)}= \lambda_3^{(2,0)}<\lambda_1^{(1,1)}$, then, by \eqref{eq:nu^(2)}, we have $a=b=c$, so $\lambda_1^{(1,0)}$ and $\lambda_1^{(2,0)}$ determine $\abcs$, settling this case as well. 
	
	In row 3, i.e., if $\lambda_1^{(2,0)}= \lambda_2^{(2,0)}<\min\{ \lambda_1^{(1,1)} , \lambda_3^{(2,0)} \}$, then  $a=b>c$ by \eqref{eq:nu^(2)}, since $\lambda_1^{(2,0)}= \lambda_2^{(2,0)}$; and $b^2+c^2<s^2$, since $\lambda_1^{(2,0)}<\lambda_1^{(1,1)}$.
	Again from Lemma~\ref{lem:lambda^(1,1)}, the third eigenvalue is given as follows:
	\begin{equation}\label{eq:table3ndeigenvalue2}
		\begin{array}{lrl}
			\lambda_3(\Ss^{4n+3}, \g_{(a,b,c,s)})  & \multicolumn{1}{c}{\text{multiplicity}} & \text{condition} \\
			\hline 
			\rule{0pt}{13pt}
			\lambda_1^{(1,1)} & n(2n+3)& \lambda_3^{(2,0)}>\lambda_1^{(1,1)}\\
			\lambda_3^{(2,0)} & (n+1)(2n+3) & \lambda_3^{(2,0)}<\lambda_1^{(1,1)} \\
			\lambda_3^{(2,0)}=\lambda_1^{(1,1)} & (2n+1)(2n+3) & \lambda_2^{(2,0)}=\lambda_1^{(1,1)}\\
		\end{array}
	\end{equation}
	As in \eqref{eq:table2ndeigenvalue}, the quantity $\lambda_1^{(3,0)}$ does not appear, since, using that $a=b <s$,
\begin{align*}
\lambda_1^{(3,0)} &= 12ns^2+ 2\nu_{1}^{(3)}(a,b,c) \\
&\geq 12ns^2+ 2(a^2+5b^2+9c^2)\\
&= 12ns^2+ 12b^2+18c^2\\
&>8ns^2+16b^2\\
&= \lambda_3^{(2,0)}.
\end{align*}
	Since the multiplicities in the rows of \eqref{eq:table3ndeigenvalue2} are all distinct, the expression for $\lambda_3(\Ss^{4n+3}, \g_{(a,b,c,s)})$ can be heard from the spectrum. The value $\lambda_1^{(1,1)}$ is determined in rows 1 and 3 of \eqref{eq:table3ndeigenvalue2}, hence 
	these cases are settled by Lemma~\ref{lem:invariants}.
	
	Suppose now that $\lambda_3^{(2,0)}<\lambda_1^{(1,1)}$, as indicated in row 2 of \eqref{eq:table3ndeigenvalue2}. 
		At this point, the strategy is to keep searching for the next eigenvalue until we find $\lambda_1^{(1,1)}$, which settles this case by Lemma~\ref{lem:invariants}. Since $\lambda_1^{(1,0)}=\lambda_2^{(1,0)}$, $\lambda_1^{(2,0)}$, $\lambda_2^{(2,0)}$, and $\lambda_3^{(2,0)}$ are all strictly smaller than $\lambda_1^{(1,1)}$, Lemma~\ref{lem:lambda^(1,1)} ensures that the next eigenvalue is $\lambda_1^{(1,1)}$, unless $n=1$, in which case $\lambda_j^{(3,0)}$ for $1\leq j\leq 4$ are the remaining candidates that might be smaller than $\lambda_1^{(1,1)}$. 
Assume $n=1$.
As $\lambda_1^{(1,1)}$ and $\lambda_j^{(3,0)}$ contribute to the spectrum with multiplicities $d_{1,1}=5$ and $d_{3,0}=20$ respectively, $\lambda_1^{(1,1)}$ is determined since it is the next occurring eigenvalue (after the third eigenvalue) with odd multiplicity, no matter where it is located among $\lambda_1^{(3,0)}\leq \lambda_2^{(3,0)}\leq \lambda_3^{(3,0)}\leq \lambda_4^{(3,0)}$. It is worth mentioning that the situation $\lambda_1^{(1,1)}>\lambda_4^{(3,0)}$ may in fact occur, provided $s$ is sufficiently large.

	It only remains to analyze row 2 of \eqref{eq:table2ndeigenvalue}, i.e.,
	the case $\lambda_1^{(2,0)}<\min\!\big\{ \lambda_1^{(1,1)}\! , \lambda_2^{(2,0)}\big\}$, which is only possible if $a>b$. 
	Suppose, for now, that $n\geq2$. 
	Then, by Lemma~\ref{lem:lambda^(1,1)}, the third eigenvalue is given as follows:
	\begin{equation*}
		\begin{array}{lrl}
			\lambda_3(\Ss^{4n+3}, \g_{(a,b,c,s)})  & \multicolumn{1}{c}{\text{multiplicity}} & \text{condition} \\
			\hline 
			\rule{0pt}{13pt}
			\lambda_1^{(1,1)} & n(2n+3)& \lambda_2^{(2,0)}>\lambda_1^{(1,1)}\\
			\lambda_2^{(2,0)} & (n+1)(2n+3) & \lambda_2^{(2,0)}<\min\{ \lambda_1^{(1,1)}\! , \lambda_3^{(2,0)} \}\\
			\lambda_2^{(2,0)} & 2(n+1)(2n+3) & \lambda_2^{(2,0)}= \lambda_3^{(2,0)}<\lambda_1^{(1,1)}\\
			\lambda_2^{(2,0)}=\lambda_1^{(1,1)} & (2n+1)(2n+3) & \lambda_2^{(2,0)}=\lambda_1^{(1,1)} <\lambda_3^{(2,0)}\\
			\lambda_2^{(2,0)}=\lambda_1^{(1,1)} & (3n+2)(2n+3) & \lambda_2^{(2,0)}= \lambda_3^{(2,0)}=\lambda_1^{(1,1)}
		\end{array}
	\end{equation*}
	As above, since none of the multiplicities coincide, the spectrum determines the expression for $\lambda_3(\Ss^{4n+3}, \g_{(a,b,c,s)})$. We are done (by Lemma~\ref{lem:invariants}) whenever $\lambda_1^{(1,1)}$ is determined, which does not happen with $\lambda_3(\Ss^{4n+3}, \g_{(a,b,c,s)})$ only if $\lambda_2^{(2,0)}<\lambda_1^{(1,1)}$. In that case, the next two eigenvalues need to be analyzed, in a totally analogous way, to show that $\lambda_1^{(1,1)}$ is eventually determined by the spectrum because the possible multiplicities are again all distinct. 
		The case $n=1$ is slightly longer, as any of $\lambda_1^{(3,0)},\dots,\lambda_4^{(3,0)}$ may occur as the next distinct eigenvalue. However, since this is also completely analogous to the above cases, the proof is omitted. 
	
	\smallskip
	
	\noindent\textbf{Row 2:} $\lambda_1^{(1,1)}< \min\!\big\{ \lambda_1^{(1,0)}, \lambda_1^{(2,0)}\big\}$.
	Since $\lambda_1^{(1,1)}=8(n+1)s^2$ is determined, so are $s>0$ and $\sigma_3$, the latter through \eqref{eq:vol}.
	Moreover, since $\lambda_1^{(q,q)} = 4\big(2qn+q(q+1)\big)s^2$ for any $q\geq0$, the value of $s$ determines the following infinite subset of the spectrum:
	\begin{equation*}
		\calS_0:=\Big\{ \underbrace{\lambda_1^{(q,q)} ,\dots,\lambda_1^{(q,q)}}_{d_{q,q}\text{-times}}:q\geq0 \Big\}\subset\spec\!\big(\Ss^{4n+3}, \g_{\abcs}\big).
	\end{equation*}
	In fact, $\calS_0=\spec\!\big(\Hr P^{n}, \frac{1}{s^2}\, \gFS\big)$ are precisely the basic eigenvalues, see Remark~\ref{rem:basic}.
	
	Consider the smallest eigenvalue in $\Spec(\Ss^{4n+3}, \g_{\abcs}) \smallsetminus \calS_0$, which is given by the minimum of $\lambda_1^{(p,q)}$, $p>q\geq0$.
	Since $\lambda_1^{(1,0)} <\lambda_1^{(q+1,q)}$ for all $q>0$, and $\lambda_1^{(2,0)}<\lambda_1^{(p,q)}$ for all $p\geq q\geq0$ with $p-q\geq2$ and $(p,q)\neq (2,0)$, this eigenvalue is
	\begin{equation}\label{eq:Row2-2ndeigenvalue}
		\begin{array}{lrl}
			\min\!\left(\Spec\!\big(\Ss^{4n+3}, \g_{\abcs}\big) \!\smallsetminus\! \calS_0\right) & \multicolumn{1}{c}{\text{multiplicity}} & \text{condition} \\
			\hline 
			\rule{0pt}{13pt}
			\lambda_1^{(1,0)} & 4(n+1)& \lambda_1^{(1,0)}< \lambda_1^{(2,0)}\\
			\lambda_1^{(2,0)} & (n+1)(2n+3) & \lambda_1^{(1,0)}> \lambda_1^{(2,0)} \\
			\lambda_1^{(1,0)} & (n+1)(2n+7) &\lambda_1^{(1,0)}= \lambda_1^{(2,0)} \\
		\end{array}
	\end{equation}
	For the multiplicity computation in the last two rows, we used that $\lambda_1^{(2,0)}<\lambda_2^{(2,0)}$ whenever $\lambda_1^{(1,0)}\geq \lambda_1^{(2,0)}$, since $a>b$, and hence $\pi_{2,0}$ contributes to the spectrum with multiplicity $d_{2,0}=(n+1)(2n+3)$. 
	Since none of the multiplicities in \eqref{eq:Row2-2ndeigenvalue} coincide, the spectrum determines the expression for the smallest nonbasic eigenvalue.
	In rows 1 and 3 of \eqref{eq:Row2-2ndeigenvalue}, the value of $\lambda_1^{(1,0)}$ is determined, so we are done, by Lemma~\ref{lem:invariants}.

	We now deal with the remaining row 2 as a particular case of the following setup:
	\begin{equation}\label{eq:weakerassumption}
		\lambda_1^{(1,1)} \text{ and } \lambda_1^{(2,0)}\text{ are known, and }\max\!\big\{\lambda_1^{(1,1)} ,\lambda_1^{(2,0)} \big\} < \lambda_1^{(1,0)}.
	\end{equation}
	In other words, we \emph{will not} use the fact that, in row 2, $\lambda_1^{(1,1)}<\lambda_1^{(2,0)}$, since proving the result under these weaker assumptions will simplify later parts of the proof.

	Given that, under these assumptions, both $s$ and $\lambda_1^{(2,0)}=8ns^2+8(b^2+c^2)$ are known, so is $b^2+c^2$. Then, since $\lambda_1^{(q+2,q)} = 4\big((2q+2)n+q(q+3)\big)s^2 +8(b^2+c^2)$, the following infinite subset of the spectrum is also determined:
	\begin{equation*}
		\calS_1:=
		\Big\{
		\underbrace{ \lambda_1^{(q+2,q)},\dots, \lambda_1^{(q+2,q)}}_{d_{q+2,q} \text{-times}} : q\geq0
		\Big\}.
	\end{equation*}
	The smallest eigenvalue in $\Spec(\Ss^{4n+3}, \g_{\abcs}) \smallsetminus (\calS_0\cup \calS_1)$ is 
	the minimum among the following union of sets:
	\begin{equation*}
		\big\{\lambda_1^{(q+k,q)}:q\geq 0,\, k\geq1\text{ odd}\big\} \cup
		\big\{\lambda_2^{(q+2,q)}:q\geq 0\big\} \cup \big\{\lambda_1^{(q+k,q)}:q\geq 0,\, k\geq4\text{ even}\big\}.
	\end{equation*}
	One can check that $\lambda_1^{(1,0)} < \lambda_1^{(q+k,q)}$ for all $k$ odd and $q\geq0$, with $(q,k)\neq (0,1)$, by \eqref{eq:nu^(k)>=}; $\lambda_1^{(1,0)} < \lambda_2^{(q+2,q)}$ for all $q\geq0$ since $a>b$; and $\lambda_1^{(4,0)} < \lambda_1^{(q+k,q)}$ for all $k\geq4$ even and $q\geq0$, with $(q,k)\neq (0,4)$, by \eqref{eq:nu^(k)>=}. 
	This implies that this minimum is
	\begin{equation*}
		\begin{array}{lrl}
			\min\!\left(\Spec\!\big(\Ss^{4n+3}, \g_{\abcs}\big)\! \smallsetminus\! (\calS_0\cup  \calS_1)\right) & \multicolumn{1}{c}{\text{multiplicity}} & \text{condition} \\
			\hline 
			\rule{0pt}{13pt}
			\lambda_1^{(1,0)} & 4(n+1)& \lambda_1^{(1,0)}< \lambda_1^{(4,0)} \\
			\lambda_1^{(4,0)} & \binom{2n+5}{4} &\lambda_1^{(1,0)}> \lambda_1^{(4,0)} \\
			\lambda_1^{(1,0)}=\lambda_1^{(4,0)} & 4(n+1)+\binom{2n+5}{4} & \lambda_1^{(1,0)}= \lambda_1^{(4,0)} \\
		\end{array}
	\end{equation*}
	The computation of multiplicities is done using that $\lambda_1^{(1,0)}=\lambda_2^{(1,0)}$ and $\pi_{1,0}$ contributes with multiplicity $2d_{1,0}=4(n+1)$, while $\lambda_1^{(4,0)} < \lambda_2^{(4,0)}$ and $\pi_{4,0}$ contributes with multiplicity $d_{4,0}=\binom{2n+5}{4}$.
	
	Once more, since the above multiplicities are pairwise different, the expression for this eigenvalue can be read from the spectrum. Furthermore, in rows 1 and 3, the proof follows from Lemma~\ref{lem:invariants} since $\lambda_1^{(1,0)}$ is determined. 
	In row 2, the proof follows from Lemma~\ref{lem:invariants} since $\lambda_1^{(1,1)}$, $\lambda_1^{(2,0)}$, and $\lambda_1^{(4,0)}$ are determined. 
	
	\smallskip
	
	\noindent\textbf{Row 3:} $\lambda_1^{(2,0)}< \min\big\{ \lambda_1^{(1,0)}, \lambda_1^{(1,1)}\big\}$.
	Lemma~\ref{lem:lambda^(1,1)} implies that the second eigenvalue is $\lambda_2(\Ss^{4n+3}, \g_{(a,b,c,s)})=\min\big\{\lambda_1^{(1,0)}, \lambda_1^{(1,1)}, \lambda_1^{(3,0)}\big\}$, and, as $\lambda_1^{(1,0)}<\lambda_1^{(3,0)}$, we have
	\begin{equation*}
		\begin{array}{lrl}
			\lambda_2(\Ss^{4n+3}, \g_{(a,b,c,s)})  & \multicolumn{1}{c}{\text{multiplicity}} & \text{condition} \\
			\hline 
			\rule{0pt}{13pt}
			\lambda_1^{(1,0)} & 4(n+1) &\lambda_1^{(1,0)}< \lambda_1^{(1,1)} \\
			\lambda_1^{(1,1)} & n(2n+3) & \lambda_1^{(1,0)}> \lambda_1^{(1,1)} \\
			\lambda_1^{(1,0)}=\lambda_1^{(1,1)}&2n^2+7n+4 & \lambda_1^{(1,0)}= \lambda_1^{(1,1)} \\
		\end{array}
	\end{equation*}
	As before, since the possible multiplicities are all distinct, the spectrum determines the expression for the second eigenvalue.
	
		If $\lambda_1^{(1,0)}= \lambda_1^{(1,1)}$, then both quantities are determined, thus so is $\abcs$ by Lemma~\ref{lem:invariants}.
		The case $\lambda_1^{(1,1)} < \lambda_1^{(1,0)}$ satisfies \eqref{eq:weakerassumption}, hence was settled in Row 2.
		
		Suppose $\lambda_1^{(1,0)}<\lambda_1^{(1,1)}$. 
		Note that $a^2> 2ns^2$, since $\lambda_1^{(2,0)}<\lambda_1^{(1,0)}$.
		Thus, $\lambda_3^{(2,0)}\geq \lambda_2^{(2,0)}= 8ns^2+8(a^2+c^2) >8ns^2+16ns^2>\lambda_1^{(1,1)}$ and $\lambda_j^{(3,0)}\geq 12ns^2+2(a^2+5b^2+9c^2) >12ns^2+4ns^2\geq 8(n+1)s^2=\lambda_1^{(1,1)}$ by \eqref{eq:nu^(k)>=}. 
		Consequently, Lemma~\ref{lem:lambda^(1,1)} implies that the third eigenvalue is $\lambda_1^{(1,1)}$, which settles this case.

	\smallskip
	
	\noindent\textbf{Row 4:} $\lambda_1^{(1,0)}= \lambda_1^{(1,1)}< \lambda_1^{(2,0)}$.
	Both $\lambda_1^{(1,0)}$ and $\lambda_1^{(1,1)}$ are determined by the spectrum, as they are equal to $\lambda_1(\Ss^{4n+3}, \g_{(a,b,c,s)})$, so the result follows from  Lemma~\ref{lem:invariants}. 
	
	\smallskip	
	
	\noindent\textbf{Row 5:} $\lambda_1^{(1,0)}= \lambda_1^{(2,0)}< \lambda_1^{(1,1)}$.
	The condition $\lambda_1^{(1,0)}= \lambda_1^{(2,0)}$ implies $a^2>2ns^2$, which, in turn, implies  that $\lambda_2(\Ss^{4n+3}, \g_{(a,b,c,s)})=\lambda_1^{(1,1)}$, similarly to the last case in Row 3. The desired conclusion then follows from Lemma~\ref{lem:invariants}.

	\smallskip
	
	\noindent\textbf{Row 6:} $\lambda_1^{(1,1)}= \lambda_1^{(2,0)}< \lambda_1^{(1,0)}$.
	Since  \eqref{eq:weakerassumption} holds, 
	this case was settled in Row 2.
	
	\smallskip
	
	\noindent\textbf{Row 7:} $\lambda_1^{(1,0)}= \lambda_1^{(2,0)}= \lambda_1^{(1,1)}$.
	Similarly to Row 4, as $\lambda_1^{(1,0)}$ and $\lambda_1^{(1,1)}$ are known, the result follows from Lemma~\ref{lem:invariants}.
\end{proof}

We now prove spectral uniqueness of $\Sp(n+1)$-invariant metrics on $\R P^{4n+3}$. The proof strategy is very similar to that of Theorem~\ref{thm:rigidityS^4n+3}, so many details are omitted. 

\begin{theorem}\label{thm:rigidityRP^4n+3}
Two isospectral $\Sp(n+1)$-invariant metrics on $\R P^{4n+3}$ are isometric. 
\end{theorem}

\begin{proof}
	Similarly to the proof of Theorem~\ref{thm:rigidityS^4n+3}, by homogeneity, the spectrum of $(\R P^{4n+3},\g_{(a,b,c,s)})$ determines $\Vol(\R P^{4n+3},\g_{(a,b,c,s)})= \frac12\Vol(\Ss^{4n+3},\g_{(a,b,c,s)})$ and $\scal(\R P^{4n+3},\g_{(a,b,c,s)})=\scal(\Ss^{4n+3},\g_{(a,b,c,s)})$. 
	
	First, let us determine $\abcs$ from $\Spec(\R P^{4n+3},\g_{(a,b,c,s)})$ assuming that:
	\begin{equation}\label{eq:weakerassumption2}
			\text{The values of } \lambda_1^{(1,1)} \text{ and } \lambda_1^{(2,0)}
			\text{ are known.}
	\end{equation}
	Additionally, suppose $\lambda_1^{(2,0)}<\lambda_2^{(2,0)}<\lambda_3^{(2,0)}$, which is equivalent to $a>b>c$. The special cases $a=b$ and $b=c$ are much simpler, and left to the reader. 

	By Lemma~\ref{lem:vol-scal}, $\lambda_1^{(1,1)}=8(n+1)s^2$ and the volume determine $s>0$, $\sigma_3$ and 
	\begin{equation}\label{eq:calS_0}
		\calS_0:=\Big\{ \underbrace{\lambda_1^{(q,q)} ,\dots,\lambda_1^{(q,q)}}_{d_{q,q}\text{-times}}:q\geq0\text{ even} \Big\}\subset\spec\!\big(\R P^{4n+3}, \g_{\abcs}\big).
	\end{equation}
	Similarly, $\lambda_1^{(1,1)}$ together with $\lambda_1^{(2,0)}$ determine $b^2+c^2$, and consequently
	\begin{equation}\label{eq:calS_1}
		\calS_1:=
		\Big\{
		\underbrace{ \lambda_1^{(q+2,q)},\dots, \lambda_1^{(q+2,q)}}_{d_{q+2,q} \text{-times}} : q\geq0\text{ even}
		\Big\}.
	\end{equation}
	By Theorem~\ref{thm:lambda_1(a,b,c,s)}, the smallest eigenvalue in $\Spec(\R P^{4n+3}, \g_{\abcs}) \smallsetminus (\calS_0\cup \calS_1)$ is the minimum of
	\begin{equation*}
		\big\{\lambda_2^{(q+2,q)}:q\geq 0 \text{ even}\big\} \cup \big\{\lambda_1^{(q+k,q)}:q\geq 0,\, k\geq4,\text{ both even}\big\}.
	\end{equation*}
	We have $\lambda_1^{(4,0)}<\lambda_2^{(4,0)}$ by Lemma~\ref{lem:nu_1^4} and the assumption $a>b$.
	For even integers $k\geq6$ and $q\geq 0$, the inequality \eqref{eq:nu^(k)>=} gives
	\begin{align*}
			\lambda_1^{(q+k,q)} 
			&\geq 4\big((k+2q)n+q(k+q+1)\big)s^2 + 2\big(2kb^2+k^2c^2\big)
			\\
			&\geq 24ns^2 + 24b^2+72c^2
			\\
			&>16ns^2 + 16\, \tfrac32(b^2+c^2) \geq16ns^2+16\beta(a,b,c) =\lambda_1^{(4,0)}.
	\end{align*}
	The last inequality follows from Lemma~\ref{lem:beta}. 
	Furthermore, we have $\lambda_1^{(q+4,q)} >\lambda_1^{(4,0)}$ for all $q>0$ even. 
	Similarly, one may check that $\lambda_2^{(2,0)}< \lambda_2^{(q+2,q)}$ for all $q>0$ even. 
	
The above facts imply the following:
	\begin{equation*}
		\begin{array}{lrl}
			\min\!\left(\Spec\!\big(\R P^{4n+3}, \g_{\abcs}\big)\! \smallsetminus\! (\calS_0\cup  \calS_1)\right) & \multicolumn{1}{c}{\text{multiplicity}} & \text{condition} \\
			\hline 
			\rule{0pt}{13pt}
			\lambda_2^{(2,0)} & (n+1)(2n+3)& \lambda_2^{(2,0)}< \lambda_1^{(4,0)}\\
			\lambda_1^{(4,0)} & \binom{2n+5}{4} &\lambda_2^{(2,0)}> \lambda_1^{(4,0)} \\
			\lambda_2^{(2,0)}=\lambda_1^{(4,0)} & \text{the sum of both} & \lambda_2^{(2,0)}= \lambda_1^{(4,0)}\\
		\end{array}
	\end{equation*}
	
	Since the above multiplicities are pairwise different, the expression for this eigenvalue can be read from the spectrum. 
	In rows 2 and 3, the expression for $\lambda_1^{(4,0)}$ is determined, thus $\abcs$ is determined by Lemma~\ref{lem:invariants}.
	Note that the hypotheses in Lemma~\ref{lem:invariants} are satisfied because the volume and scalar curvature of $(\Ss^{4n+3},\g_{(a,b,c,s)})$ are determined by the spectrum of $(\R P^{4n+3},\g_{(a,b,c,s)})$, as explained above.

	We now assume $\lambda_2^{(2,0)}<\lambda_1^{(4,0)}$, as in row 1.
	Thus, $\lambda_2^{(2,0)}$ is determined, and so are $a^2+c^2$, $\nu_2^{(2)}(a,b,c)$ and 
	\begin{equation}\label{eq:calS_2}
		\calS_2:=
		\Big\{
		\underbrace{ \lambda_2^{(q+2,q)},\dots, \lambda_2^{(q+2,q)}}_{d_{q+2,q} \text{-times}} : q\geq0\text{ even}
		\Big\}.
	\end{equation}
	Reasoning before, the smallest eigenvalue in $\Spec(\R P^{4n+3}, \g_{\abcs}) \smallsetminus (\calS_0\cup \calS_1\cup \calS_2)$ is given as in the next table:
	\begin{equation*}
		\begin{array}{lrl}
			\min\!\left(\Spec\!\big(\R P^{4n+3}, \g_{\abcs}\big)\! \smallsetminus\! (\calS_0\cup  \calS_1\cup \calS_2)\right) & \multicolumn{1}{c}{\text{multiplicity}} & \text{condition} \\
			\hline 
			\rule{0pt}{13pt}
			\lambda_3^{(2,0)} & (n+1)(2n+3)& \lambda_2^{(3,0)}< \lambda_1^{(4,0)}\\
			\lambda_1^{(4,0)} & \binom{2n+5}{4} &\lambda_3^{(2,0)}> \lambda_1^{(4,0)} \\
			\lambda_3^{(2,0)}=\lambda_1^{(4,0)} & \text{the sum of both} & \lambda_3^{(2,0)}= \lambda_1^{(4,0)}\\
		\end{array}
	\end{equation*}
	
	Once again, the multiplicity distinguishes the situation in each of the three rows. 
	In rows 2 and 3, $\lambda_1^{(4,0)}$ is determined, so are $\abcs$ by Lemma~\ref{lem:invariants}.
	In row 1, $\lambda_3^{(2,0)}$ is determined, and so is $a^2+b^2$, which together with the already known values of $a^2+c^2$ and $b^2+c^2$, determine $\abc$. 
	This completes the proof that $\Spec(\R P^{4n+3}, \g_{\abcs})$ determines $\abcs$ under the assumption \eqref{eq:weakerassumption2}. 
	
	\medskip
	
	It remains to show that no loss of generality is incurred by assuming \eqref{eq:weakerassumption2}; that is, we must prove that $\lambda_1^{(1,1)}$ and $\lambda_1^{(2,0)}$ are determined by the spectrum of $(\R P^{4n+3}, \g_{\abcs})$.
	According to Theorem~\ref{thm:lambda_1(a,b,c,s)}, the multiplicity of the first eigenvalue of $(\R P^{4n+3}, \g_{\abcs})$ can assume 7 different values, listed in \eqref{eq:mult-lambda_1(a,b,c,s)projective}.
	Thus, the proof is naturally divided in seven cases corresponding to the rows in \eqref{eq:mult-lambda_1(a,b,c,s)projective}.
	
	\noindent\textbf{Row 1:} $\lambda_1^{(1,1)}<  \lambda_1^{(2,0)}$.
	Since the expression for $\lambda_1^{(1,1)}$ is determined, so are $s$ and $\calS_0$, see \eqref{eq:calS_0}. 
	One can easily check that $\lambda_1^{(2,0)}<\lambda_1^{(p,q)}$ for all $p>q\geq0$ with $p-q$ even and strictly greater than $2$. 
	It follows that the smallest eigenvalue in $\Spec(\R P^{4n+3}, \g_{\abcs}) \smallsetminus \calS_0$ is $\lambda_1^{(2,0)}=8ns^2+8(b^2+c^2)$, and \eqref{eq:weakerassumption2} holds.

	\noindent\textbf{Row 2:} $\lambda_1^{(2,0)}< \lambda_1^{(1,1)}$ and $a>b$.
	Lemma~\ref{lem:lambda^(1,1)} implies that the second eigenvalue is given by $\min\!\big\{\lambda_1^{(1,1)}, \lambda_2^{(2,0)}\big\}$. Straightforward multiplicity computations give:
	\begin{equation*}
		\begin{array}{lrl}
			\lambda_2(\R P^{4n+3}, \g_{(a,b,c,s)})  & \multicolumn{1}{c}{\text{multiplicity}} & \text{conditions} \\
			\hline 
			\rule{0pt}{13pt}
			\lambda_2^{(2,0)} & (n+1)(2n+3) &\lambda_2^{(2,0)}< \lambda_1^{(1,1)} \text{ and } b>c\\
			\lambda_2^{(2,0)} &2(n+1)(2n+3) &\lambda_2^{(2,0)}< \lambda_1^{(1,1)} \text{ and } b=c\\
			\lambda_1^{(1,1)} & n(2n+3) & \lambda_2^{(2,0)}> \lambda_1^{(1,1)} \\
			\lambda_2^{(2,0)}=\lambda_1^{(1,1)}&(2n+1)(2n+3) & \lambda_2^{(2,0)}= \lambda_1^{(1,1)}  \text{ and } b>c\\
			\lambda_2^{(2,0)}=\lambda_1^{(1,1)}&(3n+2)(2n+3) & \lambda_2^{(2,0)}= \lambda_1^{(1,1)}  \text{ and } b=c\\
		\end{array}
	\end{equation*}
	Since none of the multiplicities coincide, the expression for  this eigenvalue can be heard. In rows 3, 4 and 5, the value $\lambda_1^{(1,1)}$ is determined, thus the proof is complete since \eqref{eq:weakerassumption2} holds. 
	The case in row 2 is simple and left to the reader.
	
	We now assume $\lambda_2^{(2,0)}< \lambda_1^{(1,1)}$ and $b<c$, as in row 1. 
	The expression for $\lambda_2^{(2,0)}$ determines $a^2+c^2$.
	Lemma~\ref{lem:lambda^(1,1)} ensures that the next eigenvalue is $\min\!\big\{\lambda_1^{(1,1)}, \lambda_3^{(2,0)}\big\}$, with distinct multiplicities given by
	\begin{equation*}
		\begin{array}{lrl}
			\lambda_3(\R P^{4n+3}, \g_{(a,b,c,s)})  & \multicolumn{1}{c}{\text{multiplicity}} & \text{conditions} \\
			\hline 
			\rule{0pt}{13pt}
			\lambda_3^{(2,0)} & (n+1)(2n+3) &\lambda_3^{(2,0)}< \lambda_1^{(1,1)}\\
			\lambda_1^{(1,1)} & n(2n+3) & \lambda_2^{(2,0)}> \lambda_1^{(1,1)} \\
			\lambda_2^{(2,0)}=\lambda_1^{(1,1)}&(2n+1)(2n+3) & \lambda_2^{(2,0)}= \lambda_1^{(1,1)}\\
		\end{array}
	\end{equation*}
	If $\lambda_1^{(1,1)}\leq \lambda_3^{(2,0)}$, then $\lambda_1^{(1,1)}$ is determined, and \eqref{eq:weakerassumption2} holds. 
	
	Suppose that $\lambda_3^{(2,0)}<\lambda_1^{(1,1)}$. 
	Since $\lambda_j^{(2,0)}$, $j=1,2,3$, are determined, so are $s^2+a^2+b^2$, $s^2+a^2+c^2$, and $s^2+b^2+c^2$, which uniquely determine the positive values of $\abc$ in terms of $s$. 
	Lemma~\ref{lem:lambda^(1,1)} implies that the fourth eigenvalue is given by $\lambda_1^{(1,1)}$, which determines $s$, and the proof of this case is complete. 
	
	\noindent\textbf{Rows 3--4:} $\lambda_1^{(2,0)}< \lambda_1^{(1,1)}$ and $a=b$.
	These cases are simpler than Row 2 and left to the reader. 
	
	\noindent\textbf{Rows 5--7:} $\lambda_1^{(1,1)}=  \lambda_1^{(2,0)}$.
	Since both expressions are determined, \eqref{eq:weakerassumption2} holds.
\end{proof}

\subsection{Spectral uniqueness among homogeneous CROSSes}
We first prove that an $\Sp(n+1)$-invariant metric on $\Ss^{4n+3}$ cannot be isospectral to an $\Sp(n+1)$-invariant metric on $\R P^{4n+3}$. For this, we need the following:

\begin{lemma}\label{claim:nu}
Suppose $a>b\geq c>0$. 
\begin{enumerate}[\rm (i)]
\item If $b^2<11c^2$, then $\nu_{1}^{(2k)}\abc<\nu_{1}^{(2k+2)}\abc$ for all $k\geq0$. 
\item $\nu_2^{(2k)}\abc> \max\left\{\nu_1^{(1)}\abc, \nu_2^{(2)}\abc, \nu_{1}^{(2k)}\abc \right\}$ for all $k\geq2$. 
\end{enumerate}
\end{lemma}

\begin{proof}
It is well-known that the $(k+1)$-dimensional irreducible representation $(\tau_k,V_{\tau_k})$ of $\SU(2)$ can be realized as the space of complex homogeneous polynomials of degree $k$ in two variables, with the action given by $(g\cdot P) \left(\begin{smallmatrix}z\\ w\end{smallmatrix}\right) =P(g^{-1}\left(\begin{smallmatrix}z\\ w\end{smallmatrix}\right))$, where $g^{-1}\left(\begin{smallmatrix}z\\ w\end{smallmatrix}\right)$ denotes matrix multiplication. 

We fix the basis $\{P_j:0 \leq j\leq k\}$, with $P_j \left(\begin{smallmatrix}z\\ w\end{smallmatrix}\right) = z^jw^{k-j}$. 
It is important to note that this basis is \emph{orthogonal} but not \emph{orthonormal} with respect to the $\G$-invariant inner product. 
Thus, the matrix $M_k=M_k(a,b,c)$ of $\tau_k(-a^2X_1^2- b^2X_2^2 - c^2X_3^2)$ with respect to this basis is not symmetric, but is similar to a positive-definite symmetric matrix. 

According to the proof of \cite[Lem.~3.1]{Lauret-SpecSU(2)}, we have that the only non-zero coefficients of $M_k=[m_{i,j}^{(k)}]_{i,j=0,\dots,k}$ are given by 
\begin{equation}\label{eq:m_ij^k}
\begin{aligned}
m_{j,j}^{(k)} &= (k-2j)^2a^2+ \big( (2j+1)k-2j^2 \big)(b^2+c^2) 
&&\text{for }0\leq j\leq k,\\
m_{j-2,j}^{(k)} &= -(j-1)j (b^2-c^2)
&&\text{for }2\leq j\leq k,\\
m_{j+2,j}^{(k)} &= -(k-1-j)(k-j) (b^2-c^2)
&&\text{for }0\leq j\leq k-2.
\end{aligned}
\end{equation}
(Although in the statement of \cite[Lem.~3.1]{Lauret-SpecSU(2)} a negative sign is missing  in the expressions for the second and third rows, as displayed above, this typo does not have any impact because the spectra of these two matrices coincide.)

Let $D_{2k}=\diag\big(d_0^{(2k)},\dots,d_{2k}^{(2k)}\big)$, where $d_j^{(2k)}=\sqrt{j!(2k-j)!}$. 
It is easy to check that $D_{2k}M_{2k}D_{2k}^{-1}$ is symmetric and has the same spectra as $M_{2k}$. Let
\begin{equation*}
U_{k} := D_{2k+2}M_{2k+2}D_{2k+2}^{-1} -
\begin{pmatrix}
\mu \\ & D_{2k}M_{2k}D_{2k}^{-1} \\ && \mu
\end{pmatrix} 
= [u_{i,j}^{(k)}]_{i,j=0,\dots,2k+2},
\end{equation*}
where $\mu=2k(k+1)(b^2+c^2)$.
We claim that 
\begin{equation}\label{eq:lambda_min(extended)}
\lambda_{\min} 
\left( \diag(\mu , D_{2k}M_{2k}D_{2k}^{-1} ,\mu)\right) 
= \lambda_{\min}(M_{2k}). 
\end{equation}
Of course, the left-hand side is equal to $\min\{\mu, \lambda_{\min}(M_{2k})\}$, so it is sufficient to show that $\mu\geq \lambda_{\min}(M_{2k})$. 
Clearly,
\begin{equation*}
\lambda_{\min}(M_{2k}) = \lambda_{\min}( D_{2k}M_{2k}D_{2k}^{-1})\leq (D_{2k}M_{2k}D_{2k}^{-1})_{j,j}=m_{j,j}^{(2k)}
\end{equation*}
for all $j$, thus $\lambda_{\min}(M_{2k})\leq m_{k,k}^{(2k)}= 2k(k+1)(b^2+c^2)=\mu$, as desired. 

Now, by \eqref{eq:lambda_min(extended)}, (i) holds if and only if $\lambda_{\min}(U_k)>0$.
It is a simple task to check that the only non-zero coefficients of the $(2k+3)\times(2k+3)$-matrix $U_{k}$ are:
\begin{align*}
u_{0,0}^{(k)} &= u_{2k+2,2k+2}^{(k)} = (2k+2)^2a^2-2(k-1)(k+1)(b^2+c^2),\\
u_{0,2}^{(k)} &= u_{0,2}^{(k)}= u_{2k,2k+2}^{(k)}= u_{2k+2,2k}^{(k)}= -(b^2-c^2)\sqrt{2(2k+1)(2k+2)},\\
u_{j,j}^{(k)} &= 4(k+1)(b^2+c^2), \quad \text{ for }1\leq j\leq 2k+1, \text{ and }\\
u_{j-2,j}^{(k)} &= u_{j,j-2}^{(k)} = \frac{-4(k+1)(b^2-c^2) \; \sqrt{(j-1)(2k+3-j)}}{\sqrt{j(2k+4-j)} + \sqrt{(j-2)(2k+2-j)}},
\quad \text{ for }3\leq j\leq 2k+1.
\end{align*}

By the Gershgorin Circle Theorem, see e.g.~\cite{varga}, all eigenvalues $\lambda$ of $U_{k}$ satisfy
\begin{equation}\label{eq:gersh-goal}
\lambda \geq \min_{0\leq j\leq 2k+2} \left\{ 
\beta(k,j):=u_{j,j}^{(k)}- |u_{j+2,j}^{(k)}|- |u_{j-2,j}^{(k)}|
\right\}, 
\end{equation}
where the coefficients with index outside the range $\{0,\dots,2k+2\}$ are conventioned to be $0$; e.g., $u_{-2,0}^{(k)}=0$. 
One can easily check that $\beta(k,0)=\beta(k,2k+2)>0$ for $k\geq2$, and $\beta(k,1)=\beta(k,2k+1)>0$. 
Furthermore, 
\begin{align*}
\beta(k,2)=\beta(k,2k)&=
4(k+1)(b^2+c^2)
-\tfrac{4(k+1)(b^2-c^2) \; \sqrt{3(2k-1)}}{\sqrt{8k} + \sqrt{2(2k-2)}}
\\ &\quad
- (b^2-c^2)\sqrt{2(2k+1)(2k+2)}
\\
&\geq 
4(k+1)(b^2+c^2)-4(k+1)\left(\tfrac{ \sqrt{3}}{2 + \sqrt{2}} +\tfrac{1}{\sqrt{2}}\right)(b^2-c^2)
\\
&\geq 4(k+1)(b^2+c^2)-4(k+1) \tfrac{6}{5}(b^2-c^2),
\end{align*}
which is positive, since the hypothesis $11c^2>b^2$ is equivalent to $b^2+c^2>\frac{6}{5}(b^2-c^2)$. 
For $3\leq j\leq 2k-1$, one has that
	\begin{align*}
		\beta(k,j)&=
		4(k+1)(b^2+c^2)
		-\tfrac{4(k+1)(b^2-c^2) \; \sqrt{(j+1)(2k+1-j)}}{\sqrt{(j+2)(2k+2-j)} + \sqrt{j(2k-j)}}
		\\ &\quad
		-\tfrac{4(k+1)(b^2-c^2) \; \sqrt{(j-1)(2k+3-j)}}{\sqrt{j(2k+4-j)} + \sqrt{(j-2)(2k+2-j)}}. 
	\end{align*}
Moreover, a direct computation gives
\begin{multline*}
	\beta(k,j)\geq \beta(k,3) =4(k+1)(b^2+c^2)
	\\
	-4(k+1)(b^2-c^2) 
	\left( \tfrac{\sqrt{4(2k-2)}}{\sqrt{5(2k-1)} + \sqrt{3(2k-3)}} + \tfrac{\sqrt{4k}}{\sqrt{3(2k+1)} + \sqrt{2k-1}} 
	\right)
	\\
	>4(k+1)(b^2+c^2)
	-4(k+1)(b^2-c^2) 
	\left( \tfrac{2}{\sqrt{5} + \sqrt{3}} + \tfrac{\sqrt{2}}{\sqrt{3} + 1} 
	\right),
\end{multline*}
which is positive, since  $\tfrac{2}{\sqrt{5} + \sqrt{3}} + \tfrac{\sqrt{2}}{\sqrt{3} + 1}\approx 1.021<\tfrac{6}{5}$. 
Therefore, the right-hand side of \eqref{eq:gersh-goal} is positive, and hence so is $\lambda_{\min}(U_k)$, which concludes the proof of (i).

We now turn to (ii). We have that $\nu_{2}^{(2k)} \geq 4kb^2+4k^2c^2>4(b^2+c^2)=\nu_{1}^{(2)} $ for all $k\geq2$, by \eqref{eq:nu^(k)>=}. Before proceeding, note that \eqref{eq:m_ij^k} can be used to check that
\begin{equation}\label{eq:alpha(2k,j)}
	\alpha(2k,j):=m_{j,j}^{(2k)}- |m_{j-2,j}^{(2k)}|- |m_{j+2,j}^{(2k)}| 
	=4(k-j)^2(a^2-b^2)+ 4kb^2+4k^2c^2
\end{equation}
for all $0\leq j\leq 2k$, where, by convention, $m_{i,j}=0$ if $i<0$ or $i>2k$; see also the proof of \cite[Lem.~3.4]{Lauret-SpecSU(2)}.

Next, let us show that $\nu_1^{(2k)} <\nu_2^{(2k)}$ for all $k\geq1$. The matrix $M_{2k}$ is similar to 
\begin{equation}
\diag \left( [m_{2i,2j}^{(2k)}]_{0\leq i,j\leq k},\, [m_{2i+1,2j+1}^{(2k)}]_{0\leq i,j\leq k-1}\right).
\end{equation}
Both blocks in the above block-diagonal matrix are \emph{tridiagonal matrices}; the first one is $(k+1)\times (k+1)$ and the second is $k\times k$. 
We shall only consider the case in which $k$ is even, since the case of odd $k$ is analogous and left to the reader.

Using the Gershgorin Circle Theorem again, we have that the smallest eigenvalue of the $k\times k$-block is greater than the minimum of $\alpha(2k,j)$ for $0\leq j\leq 2k$ with $j$ odd, which is realized when $j=k\pm1$ by \eqref{eq:alpha(2k,j)}; namely, 
\begin{equation}\label{eq:alpha(2k,k+1)}
	\alpha(2k,k\pm1)=4a^2+ 4(k-1)b^2 + 4k^2c^2.
\end{equation}
On the one hand, since $\nu_1^{(2k)} \leq m_{k,k}^{(2k)}=  4kb^2+4k^2c^2<\alpha(2k,k\pm1)$ because $a>b$, we deduce that $\nu_{1}^{(2k)}$ coincides with the smallest eigenvalue of the $(k+1)\times(k+1)$-block, and it is strictly smaller than every eigenvalue of the $k\times k$-block. 
On the other hand, the $(k+1)\times(k+1)$-block is a tridiagonal matrix with non-zero non-diagonal entries, thus it has simple spectrum, and, therefore, $\nu_{1}^{(2k)}$ is strictly smaller than the second eigenvalue of the first block. 
We conclude that $\nu_1^{(2k)}<\nu_2^{(2k)}$.

It only remains to show that $a^2+b^2+c^2=\nu_1^{(1)} <\nu_{2}^{(2k)}$ for every $k\geq2$. 
This has actually already been proven, since $\nu_{2}^{(2k)}$ and $\nu_1^{(2k)}$ are realized in different blocks, so the previous case shows that $\nu_{2}^{(2k)} \geq \alpha(2k,k\pm1) =4a^2+4(k-1)b^2+4k^2c^2$ by \eqref{eq:alpha(2k,k+1)}, which gives $\nu_2^{(2k)} >a^2+b^2+c^2=\nu_1^{(1)}$. 
\end{proof}

\begin{proposition}\label{lem:isospS-RP}
For all $n\geq 0$, an $\Sp(n+1)$-invariant metric on $\Ss^{4n+3}$ cannot be isospectral to an $\Sp(n+1)$-invariant metric on $\R P^{4n+3}$. 
\end{proposition}

\begin{proof}
Suppose that $(\Ss^{4n+3},\g_{(a_1,b_1,c_1,s_1)})$ and $(\R P^{4n+3},\g_{(a_2,b_2,c_2,s_2)})$ are isospectral for some positive real numbers $a_i\geq b_i\geq c_i$ and $s_i$, for $i=1,2$. 
We assume $n\geq1$ since the case $n=0$ is very similar (essentially, one has to set $s_1=s_2=0$).

The multiplicity of the first Laplace eigenvalue in both manifolds must coincide. By Theorem~\ref{thm:lambda_1(a,b,c,s)},  such multiplicities are given by \eqref{eq:mult-lambda_1(a,b,c,s)} and \eqref{eq:mult-lambda_1(a,b,c,s)projective}, respectively.
Hence, we have that the multiplicity is equal to either
$n(2n+3)$, $(n+1)(2n+3)$, or $(2n+1)(2n+3)$. 
		
We first assume it is $n(2n+3)$.
The smallest positive eigenvalues of each spectra coincide, that is, 
$\lambda_1^{(1,1)}(a_1,b_1,c_1,s_1)= \lambda_1^{(1,1)}(a_2,b_2,c_2,s_2)$, which gives $s_1=s_2$. 
We set $\calS_0$ as in \eqref{eq:calS_0}, which is contained simultaneously in both spectra. 
		
We have already seen in the proof of Theorem~\ref{thm:rigidityRP^4n+3} that the smallest eigenvalue in $\spec(\R P^{4n+3},\g_{(a_2,b_2,c_2,s_2)})\smallsetminus\calS_0$ is $\lambda_1^{(2,0)}(a_2,b_2,c_2,s_2)$, with multiplicity $(n+1)(2n+3)$ if $a_2>b_2$, $2(n+1)(2n+3)$ if $a_2=b_2>c_2$, and $3(n+1)(2n+3)$ if $a_2=b_2=c_2$. 
Similarly, an almost identical procedure to that done for Row 2 in the proof of Theorem~\ref{thm:rigidityS^4n+3} gives that the smallest eigenvalue in $\spec(\Ss^{4n+3},\g_{(a_1,b_1,c_1,s_1)})\smallsetminus\calS_0$ is given as in \eqref{eq:Row2-2ndeigenvalue}. 
Since the only common value among their multiplicities is $(n+1)(2n+3)$, we have that $\lambda_1^{(2,0)}(a_1,b_1,c_1,s_1)= \lambda_1^{(2,0)}(a_2,b_2,c_2,s_2)$.

Let us now assume that the multiplicity is $(n+1)(2n+3)$. We have that $a_i>b_i$ for $i=1,2$. 
Since the first eigenvalues coincide, we obtain that $\lambda_1^{(2,0)}(a_1,b_1,c_1,s_1)= \lambda_1^{(2,0)}(a_2,b_2,c_2,s_2)$.
The second eigenvalue with its corresponding multiplicity on $(\Ss^{4n+3},\g_{(a_1,b_1,c_1,s_1)})$ (resp.\ $(\R P^{4n+3},\g_{(a_2,b_2,c_2,s_2)})$) has been explicitly determined at the beginning of the case Row 3 (resp.\ Row 2) in the proof of Theorem~\ref{thm:rigidityS^4n+3} (resp.\ Theorem~\ref{thm:rigidityRP^4n+3}).
A simple inspection shows that the only possible coincidence among their multiplicities is $n(2n+3)$, when the corresponding eigenvalue is $\lambda_1^{(1,1)}(a_1,b_1,c_1,s_1)= \lambda_1^{(1,1)}(a_2,b_2,c_2,s_2)$.
Furthermore, $a_i>b_i$ for $i=1,2$. 

When the multiplicity is $(2n+1)(2n+3)$, one has that $\lambda_1^{(1,1)}(a_1,b_1,c_1,s_1)= \lambda_1^{(2,0)}(a_1,b_1,c_1,s_1)=  \lambda_1^{(2,0)}(a_2,b_2,c_2,s_2)= \lambda_1^{(1,1)}(a_2,b_2,c_2,s_2)$ and $a_i>b_i$, $i=1,2$. 

\smallskip

Summing up, we have proved thus far that:
\begin{align*}
&\lambda_1^{(1,1)}(a_1,b_1,c_1,s_1)= \lambda_1^{(1,1)}(a_2,b_2,c_2,s_2), \\ 
&\lambda_1^{(2,0)}(a_1,b_1,c_1,s_1)=  \lambda_1^{(2,0)}(a_2,b_2,c_2,s_2), \\[2pt]
&a_i>b_i, \text{ for }i=1,2. 
\end{align*}
This implies that 
\begin{equation}\label{eq:identities1}
s:=s_1=s_2 \quad \text{and} \quad b_1^2+c_1^2=b_2^2+c_2^2.
\end{equation}

By \eqref{eq:lambda}, \eqref{eq:identities1} forces $\nu_1^{(2)}(a_1,b_1,c_1)= \nu_1^{(2)}(a_2,b_2,c_2)$. 
Consequently, the set $\calS_1$ defined as in \eqref{eq:calS_1} is simultaneously contained in both spectra.  
From the proofs of Theorems~\ref{thm:rigidityS^4n+3} and \ref{thm:rigidityRP^4n+3}, we easily see that the only possible coincidence among multiplicities of the smallest eigenvalues in $\spec(\Ss^{4n+3},\g_{(a_1,b_1,c_1,s)})\smallsetminus (\calS_0\cup\calS_1)$ and $\spec(\R P^{4n+3},\g_{(a_2,b_2,c_2,s)}) \smallsetminus (\calS_0\cup\calS_1)$ is $\dim V_{4,0}=\binom{2n+5}{4}$, thus 
\begin{equation}\label{eq:lambda^40}
\lambda_1^{(4,0)}(a_1,b_1,c_1,s) = \lambda_1^{(4,0)}(a_2,b_2,c_2,s).
\end{equation}
This situation occurs only if 
$
\lambda_1^{(1,0)}(a_1,b_1,c_1,s) >\lambda_1^{(4,0)}(a_1,b_1,c_1,s),
$
which gives $4ns^2+2(a_1^2+b_1^2+c_1^2) \geq 16ns^2+2\nu_1^{(4)}(a_2,b_2,c_2) \geq 16ns^2+2(8b_2^2+16c_2^2)>16ns^2+16(b_1^2+c_1^2)$ by \eqref{eq:nu^(k)>=} and \eqref{eq:identities1}, thus $a_1^2>6ns^2+7(b_1^2+c_1^2)$, and only if 
$
\lambda_2^{(2,0)}(a_2,b_2,c_2,s) >\lambda_1^{(4,0)}(a_2,b_2,c_2,s),
$
which gives $8ns^2+8a^2+8c^2 >  16ns^2+2\nu_1^{(4)}(a_2,b_2,c_2)\geq 16ns^2+16b^2+32c^2$ by \eqref{eq:nu^(k)>=}, thus 
\begin{equation}\label{eq:a_2lowerbound}
a_2^2\geq ns^2 +2b_2^2 +3c_2^2. 
\end{equation}

At this point, we divide the proof according to whether $b_i^2<11c_i^2$ holds or not. 

\smallskip

\noindent
\textbf{First case}:
Assume that $b_i^2<11c_i^2$, for both $i=1,2$. 

From \eqref{eq:lambda^40}, we obtain that $\nu_1^{(4)}(a_1,b_1,c_1)= \nu_1^{(4)}(a_2,b_2,c_2)$. 
Therefore, 
the following subset is simultaneously contained in both spectra:
\begin{equation}
	\calS_2:=
	\Big\{
	\underbrace{ \lambda_1^{(q+4,q)},\dots, \lambda_1^{(q+4,q)}}_{d_{q+4,q} \text{-times}} : q\geq0\text{ even}
	\Big\}.
\end{equation}

From Lemma~\ref{lem:implicitSpec}, the smallest eigenvalues in  $\spec(\Ss^{4n+3},\g_{(a_1,b_1,c_1,s)})\smallsetminus (\calS_0\cup\calS_1\cup\calS_2)$ is given by 
\begin{multline*}
\min\left(
	\begin{array}{l}
	\big\{\lambda_1^{(k+q,q)}(a_1,b_1,c_1,s):k\geq1\text{ odd},\, q\geq 0\big\} \,\cup\,
	\\
	\big\{\lambda_2^{(2+q,q)}(a_1,b_1,c_1,s):q\geq 0\big\} \,\cup\,
	\\
	\big\{\lambda_2^{(4+q,q)}(a_1,b_1,c_1,s):q\geq 0\big\} \,\cup\,
	\\
	\big\{\lambda_1^{(q+k,q)}(a_1,b_1,c_1,s):k\geq6\text{ even},\, q\geq 0,\big\}	
	\end{array}
\right)
\\ 
=\min\left( \lambda_1^{(1,0)}(a_1,b_1,c_1,s), \lambda_1^{(6,0)}(a_1,b_1,c_1,s) \right).
\end{multline*}
The last equality follows from the following facts, which, in turn, rely on \eqref{eq:lambda}:
\begin{enumerate}[$\bullet$]
\item $\lambda_1^{(6,0)}(a_1,b_1,c_1,s)< \lambda_1^{(q+k,q)}(a_1,b_1,c_1,s)$ if $k\geq8$ is even, by Lemma~\ref{claim:nu};
\item  $\lambda_1^{(1,0)}(a_1,b_1,c_1,s)< \lambda_2^{(k+q,q)}(a_1,b_1,c_1,s)$ if $k\geq0$ is even and $q\geq0$, by Lemma~\ref{claim:nu};
\item $\lambda_1^{(1,0)}(a_1,b_1,c_1,s)< \lambda_1^{(k+q,q)}(a_1,b_1,c_1,s)$ if $k\geq1$ is odd and $q\geq0$ with $(k,q)\neq (1,0)$, by \eqref{eq:nu^(k)>=}.
\end{enumerate}
Likewise, the smallest eigenvalue in $\spec(\R P^{4n+3},\g_{(a_2,b_2,c_2,s_2)}) \smallsetminus (\calS_0\cup\calS_1\cup\calS_2)$ is given by 
\begin{multline*}
\min\left(
\begin{array}{l}
	\big\{\lambda_2^{(2+q,q)}(a_2,b_2,c_2,s):q\geq 0\big\} \,\cup\,
	\\
	\big\{\lambda_2^{(4+q,q)}(a_2,b_2,c_2,s):q\geq 0\big\} \,\cup\,
	\\
	\big\{\lambda_1^{(q+k,q)}(a_2,b_2,c_2,s):k\geq6\text{ even},\, q\geq 0,\big\}	
\end{array}
\right)
\\ 
	=\min\left( \lambda_2^{(2,0)}(a_2,b_2,c_2,s), \lambda_1^{(6,0)}(a_2,b_2,c_2,s) \right).
\end{multline*}
The last equality follows from the following facts, where, once again, \eqref{eq:lambda} is used:
\begin{enumerate}[$\bullet$]
\item $\lambda_1^{(6,0)}(a_2,b_2,c_2,s)< \lambda_1^{(q+k,q)}(a_2,b_2,c_2,s)$ if $k\geq8$ is even, by Lemma~\ref{claim:nu};
\item  $\lambda_2^{(2,0)}(a_2,b_2,c_2,s)< \lambda_2^{(k+q,q)}(a_2,b_2,c_2,s)$ if $k\geq4$ is even and $q\geq0$, by Lemma~\ref{claim:nu}.
\end{enumerate}

The multiplicities of the first eigenvalues are clearly given by:
\begin{equation*}
\begin{aligned}
&
\begin{cases}
2\dim V_{1,0} &\text{ if } \lambda_1^{(1,0)}(a_1,b_1,c_1,s)< \lambda_1^{(6,0)}(a_1,b_1,c_1,s),
\\
\dim V_{6,0} &\text{ if } \lambda_1^{(1,0)}(a_1,b_1,c_1,s)> \lambda_1^{(6,0)}(a_1,b_1,c_1,s),
\\
2\dim V_{1,0}+\dim V_{6,0} &\text{ if } \lambda_1^{(1,0)}(a_1,b_1,c_1,s)= \lambda_1^{(6,0)}(a_1,b_1,c_1,s),
\end{cases}
\\
&
\begin{cases}
\dim V_{2,0} &\text{ if } \lambda_2^{(2,0)}(a_2,b_2,c_2,s)< \lambda_1^{(6,0)}(a_2,b_2,c_2,s) \text{ and } b_2>c_2,
\\
2\dim V_{2,0} &\text{ if } \lambda_2^{(2,0)}(a_2,b_2,c_2,s)< \lambda_1^{(6,0)}(a_2,b_2,c_2,s) \text{ and } b_2=c_2,
\\
\dim V_{6,0} &\text{ if } \lambda_2^{(2,0)}(a_2,b_2,c_2,s)> \lambda_1^{(6,0)}(a_2,b_2,c_2,s),
\\
\dim V_{2,0}+\dim V_{6,0} &\text{ if } \lambda_2^{(2,0)}(a_2,b_2,c_2,s)= \lambda_1^{(6,0)}(a_2,b_2,c_2,s) \text{ and } b_2>c_2,
\\
2\dim V_{2,0}+\dim V_{6,0} &\text{ if } \lambda_2^{(2,0)}(a_2,b_2,c_2,s)= \lambda_1^{(6,0)}(a_2,b_2,c_2,s) \text{ and } b_2=c_2,
\end{cases}
\end{aligned}
\end{equation*}
respectively. 
Since the only possible coincidence among multiplicities is $\dim V_{6,0}$, we have that $\lambda_1^{(6,0)}(a_1,b_1,c_1,s)= \lambda_1^{(6,0)}(a_2,b_2,c_2,s)$, which occurs only if
\begin{align*}
\lambda_1^{(1,0)}(a_1,b_1,c_1,s) &
> \lambda_1^{(6,0)}(a_1,b_1,c_1,s) 
=24ns^2+2 \nu_1^{(6)}(a_2,b_2,c_2)
\\&
> 24ns^2+24b_2^2+72c_2^2 \qquad\text{(by \eqref{eq:nu^(k)>=})}
\\&
> 24ns^2+24(b_1^2+c_1^2) \qquad\text{(by \eqref{eq:identities1})},
\end{align*}
thus $a_1^2>10ns^2+11(b_1^2+c_1^2)$, because $\lambda_1^{(1,0)}(a_1,b_1,c_1,s)= 4ns^2+2(a_1^2+b_1^2+c_1^2)$.
Furthermore, we have that $\nu_1^{(6)}(a_1,b_1,c_1) = \nu_1^{(6)}(a_2,b_2,c_2)$. 

Repeating this procedure, we deduce from the multiplicity of the smallest eigenvalue in 
$\spec(\Ss^{4n+3},\g_{(a_1,b_1,c_1,s_1)}) \smallsetminus (\calS_0\cup\dots\cup \calS_k)$ and $\spec(\R P^{4n+3},\g_{(a_2,b_2,c_2,s_2)}) \smallsetminus (\calS_0\cup\dots\cup\calS_k)$, where 
\begin{equation*}
\calS_i = \Big\{
\underbrace{ \lambda_1^{(q+2i,q)},\dots, \lambda_1^{(q+2i,q)}}_{d_{q+2i,q} \text{-times}} : q\geq0\text{ even}
\Big\},
\end{equation*}
that $\lambda_1^{(2k,0)}(a_1,b_1,c_1,s)= \lambda_1^{(2k,0)}(a_2,b_2,c_2,s)$, which occurs only if
\begin{equation}
	\begin{aligned}
 4ns^2+2(a_1^2+b_1^2+c_1^2)&=
		\lambda_1^{(1,0)}(a_1,b_1,c_1,s) \\&
		> \lambda_1^{(2k,0)}(a_1,b_1,c_1,s) 
		=8kns^2+2 \nu_1^{(2k)}(a_2,b_2,c_2)
		\\&
		> 8kns^2+8kb_2^2+8k^2c_2^2 \qquad\text{(by \eqref{eq:nu^(k)>=})}
		\\&
		> 8kns^2+8k(b_1^2+c_1^2) \qquad\text{(by \eqref{eq:identities1})}.
	\end{aligned}
\end{equation} 
Hence
\begin{equation}
a_1^2> 2(2k-1)ns^2+(4k-1)(b_1^2+c_1^2)
\end{equation}
for every positive integer $k$, which gives the required contradiction.

\bigskip

\noindent
\textbf{Second case}:
Assume that either
\begin{equation}\label{eq:assumption2ndcase}
b_1^2\geq 11c_1^2, \qquad \text{or} \qquad b_2^2\geq 11c_2^2. 
\end{equation}

So far, we have shown that $s:=s_1=s_2$ and $b_1^2+c_1^2=b_2^2+c_2^2$ from \eqref{eq:identities1}, and 
\begin{equation}
\beta(a_1,b_1,c_1)=\beta(a_2,b_2,c_2)=:\beta
\end{equation}
from $\lambda_1^{(4,0)}(a_1,b_1,c_1,s)= \lambda_1^{(4,0)}(a_2,b_2,c_2,s)$, 
where $\beta(a,b,c)$ is given as in \eqref{eq:beta}. 
Furthermore, since
\begin{equation*}
\Vol(\Ss^{4n+3},\g_{(a_1,b_1,c_1,s_1)})= \Vol(\R P^{4n+3},\g_{(a_2,b_2,c_2,s_2)}) = \tfrac12\Vol(\Ss^{4n+3}, \g_{(a_2,b_2,c_2,s_2)}),
\end{equation*}
Lemma~\ref{lem:vol-scal} implies that $\sigma_3(a_2,b_2,c_2)=4\, \sigma_3(a_1,b_1,c_1)$, that is,
\begin{equation}\label{eq:volS=2volRP}
	a_2^2b_2^2c_2^2=4\, a_1^2b_1^2c_1^2. 
\end{equation}
Also, the proof of Lemma~\ref{lem:invariants} ensures that 
\begin{equation}\label{eq:quadraticequations}
	\begin{aligned}
		Aa_1^4-Ba_1^2+C_1&=0,\\
		Aa_2^4-Ba_2^2+C_2&=0,
	\end{aligned}	
\end{equation}
where $A=3(b_i^2+c_i^2)-2\beta$, $B=\beta(2(b_i^2+c_i^2)-\beta)$, $C_1=3\sigma_3(a_1,b_1,c_1)=3\,a_1^2b_1^2c_1^2$,  $C_2=3\sigma_3(a_2,b_2,c_2)=3\,a_2^2b_2^2c_2^2=4C_1$, and, moreover, $A\geq0$, and $B,C_1,C_2$ are all positive. 
Actually, also $A>0$ by Lemma~\ref{lem:beta} and \eqref{eq:assumption2ndcase}. 
Consequently, $a_i^2=\tfrac{1}{2A}\big(B\pm\sqrt{B^2-4AC_i}\big)$. 
We claim that only the larger real root occurs if $b_i\geq 11c_i$:

\begin{claim}\label{claim:largestroot}
If $b_i^2\geq 11 c_i^2$, then $a_i^2=\tfrac{1}{2A}\big(B+\sqrt{B^2-4AC_i}\big)$. 
\end{claim}

\begin{proof}
\renewcommand{\qedsymbol}{$\blacksquare$}
Clearly, it is sufficient to show that $a_i^2> \frac{B}{2A}$. 
First, note that $b_i^2\geq 11c_i^2$ implies $\frac{5}{6}(b_i^2+c_i^2)\leq b_i^2-c_i^2$. 
By straightforward manipulations, one has that 
\begin{equation*}
\frac{B}{2A} \!= \!\frac{a_i^2(b_i^2+c_i^2)^2 -b_i^2c_i^2\big(\beta-(b_i^2+c_i^2)\big)} {2(b_i^2-c_i^2)^2}
\!\leq\! \frac{a_i^2(b_i^2+c_i^2)^2 -b_i^2c_i^2\big(\beta-(b_i^2+c_i^2)\big)} {\frac{25}{18}(b_i^2+c_i^2)^2} \!< \! a_i^2.
\end{equation*}
Since $\beta>b_i^2+c_i^2$ by Lemma~\ref{lem:beta}, the assertion follows. 
\end{proof}

Since $(\R P^{4n+3}, \g_{(a_2,b_2,c_2,s_2)})$ and $(\Ss^{4n+3},\g_{(a_1,b_1,c_1,s_1)})$ were assumed to be isospectral, their scalar curvatures must coincide. 
Thus, by Lemma~\ref{lem:vol-scal}, 
	\begin{align*}
		0
		&=\scal(\R P^{4n+3}, \g_{(a_2,b_2,c_2,s_2)})- \scal(\Ss^{4n+3},\g_{(a_1,b_1,c_1,s_1)})
		\\
		&= 
		16(a_2^2-a_1^2)
		-2ns^4 \left(\frac{a_2^2(b_2^2+c_2^2) + b_2^2c_2^2}{a_2^2b_2^2c_2^2} -\frac{a_1^2(b_1^2+c_1^2) + b_1^2c_1^2}{a_1^2b_1^2c_1^2} \right)
		\\ &\quad
		-4 \left(
		\frac{\big(a_2^2(b_2^2+c_2^2) + b_2^2c_2^2\big)^2} {a_2^2b_2^2c_2^2} - 
		\frac{\big(a_1^2(b_1^2+c_1^2) + b_1^2c_1^2\big)^2} {a_1^2b_1^2c_1^2} 
		\right).
	\end{align*}
Combining \eqref{eq:identities1} and \eqref{eq:volS=2volRP}, tedious but straightforward computations give 
\begin{equation}\label{eq:Ftilde}
	\begin{aligned}
		0
		&=
		\frac{ns^4}{2a_1^2b_1^2c_1^2} \left(
		(4a_1^2-a_2^2)(b_1^2+c_1^2) 
		- 4b_1^2c_1^2\tfrac{a_1^2-a_2^2}{a_2^2}
		\right)
		\\&\quad 
		+ \frac{ \big((b_1^2+c_1^2)^2 - \tfrac{4b_1^4c_1^4}{a_2^4}\big) }{a_1^2b_1^2c_1^2}
		(2a_1^2+a_2^2)(2a_1^2-a_2^2)
		-16(a_1^2-a_2^2).
	\end{aligned}
\end{equation}

The following technical (but simple) facts will be used in the sequel.

\begin{claim}\label{claim:RHS>0}
	If $a_2^2<a_1^2$, then the right-hand side of \eqref{eq:Ftilde} is positive. 
\end{claim}

\begin{proof}
	\renewcommand{\qedsymbol}{$\blacksquare$}
	Concerning the first term, we have that
	\begin{equation}\label{eq:first-row-positive}
		\begin{aligned}
			(4a_1^2-a_2^2)(b_1^2+c_1^2) 
			- 4b_1^2c_1^2\tfrac{a_1^2-a_2^2}{a_2^2}
			&
			>(a_1^2-a_2^2) \left( (b_1^2+c_1^2) - \tfrac{4b_1^2c_1^2}{a_2^2} \right).
		\end{aligned}
	\end{equation}
	By \eqref{eq:a_2lowerbound}, we get that $a_2^2> 2(b_2^2+c_2^2)=2(b_1^2+c_1^2)>4b_1c_1$, thus $\tfrac{4b_1^2c_1^2}{a_2^2}<b_1c_1<b_1^2+c_1^2 $, which shows that \eqref{eq:first-row-positive} is positive. 
	
	To prove that the remaining terms in \eqref{eq:Ftilde} are positive, it suffices to show that 
	\begin{equation*}
		(a_2^2+2a_1^2)
		\left((b_1^2+c_1^2)^2 - \tfrac{4b_1^4c_1^4}{a_2^4}\right) >16\,{a_1^2b_1^2c_1^2}. 
	\end{equation*}
	We already saw that $a_2^2>4b_1c_1$, thus 
	$(b_1^2+c_1^2)^2 - \tfrac{4b_1^4c_1^4}{a_2^4}
	>(b_1^2+c_1^2)^2 - \tfrac{b_1^2c_1^2}{4} 
	=\tfrac78(b_1^2+c_1^2)^2+ \tfrac18(b_1^2-c_1^2)^2
	>\tfrac78(b_1^2+c_1^2)^2$.
	Using, in addition, that $a_2^2+2a_1^2>2a_1^2$, the above is verified if 
	$\tfrac74(b_1^2+c_1^2)^2>16\,{b_1^2c_1^2}$,
	which holds thanks to the fact that $(b_1^2+c_1^2)^2=(b_2^2+c_2^2)^2\geq 4b_2^2c_2^2 = 16\frac{a_1^2}{a_2^2}b_1^2c_1^2>16b_1^2c_1^2$, by \eqref{eq:volS=2volRP}.
\end{proof}

\begin{claim}\label{claim:RHS<0}
If $a_2^2>7a_1^2$, then the right-hand side of \eqref{eq:Ftilde} is negative. 
\end{claim}

\begin{proof}
\renewcommand{\qedsymbol}{$\blacksquare$}
The first term in \eqref{eq:Ftilde} is negative if and only if 
\begin{equation*}
(a_2^2-4a_1^2)(b_1^2+c_1^2) 
	> 4b_1^2c_1^2\tfrac{a_2^2-a_1^2}{a_2^2}. 
\end{equation*}
By noting that $\frac{b_1c_1}{a_2^2}< \frac{b_1c_1}{7a_1^2}<\frac17$ and $b_1^2+c_1^2\geq 2b_1c_1$, it is sufficient to show that $2(a_2^2-4a_1^2) > \frac47 (a_2^2-a_1^2),$ which is clearly true because $a_2^2>7a_1^2$. 

The remaining terms in \eqref{eq:Ftilde} are negative if and only if 
\begin{equation*}
	(a_2^4-4a_1^4) \left((b_1^2+c_1^2)^2 - \tfrac{4b_1^4c_1^4}{a_2^4}\right)
	>16(a_2^2-a_1^2)a_1^2b_1^2c_1^2.
\end{equation*}
Since $a_2^4>49a_1^4$, we have that $\frac{4b_1^4c_1^4}{a_2^4}<\frac{4b_1^4c_1^4}{49a_1^4}< \frac{1}{12} b_1^2c_1^2$, and so
$
(b_1^2+c_1^2)^2-\frac{4b_1^4c_1^4}{a_2^4} 
>4b_1^2c_1^2- \frac{1}{12}b_1^2c_1^2
=\frac{47}{12} b_1^2c_1^2.
$
Consequently, it is sufficient to show that 
\begin{equation*}
	\tfrac{47}{12} 
	(a_2^4-4a_1^4)
	>16(a_2^2-a_1^2)a_1^2.
\end{equation*}
The above identity can be easily verified keeping in mind that  $a_2^2>7a_1^2$.
\end{proof}

We are now in position to finish the proof, seeking the desired contradiction under the assumption \eqref{eq:assumption2ndcase}, that is, $b_i^2\geq 11c_i^2$ for some $i=1,2$. 

We first suppose that $b_1^2\geq 11c_1^2$, thus $a_1^2=\tfrac{1}{2A}(B+\sqrt{B^2-4AC_1})$ by Claim~\ref{claim:largestroot}. 
Thus, $a_2^2=\tfrac{1}{2A}(B\pm\sqrt{B^2-4AC_2})= \tfrac{1}{2A}(B\pm\sqrt{B^2-16AC_1}) <a_1^2$, so Claim~\ref{claim:RHS>0} yields the desired contradiction. 

Suppose now that $b_2^2\geq 11c_2^2$.  Then Claim~\ref{claim:largestroot} forces 
\begin{equation}\label{eq:a2}
a_2^2= \tfrac{1}{2A}\left(B+\sqrt{B^2-4AC_2}\right) =\tfrac{1}{2A}\left(B+\sqrt{B^2-16AC_1}\right). 
\end{equation} 
We recall that $a_1^2=\tfrac{1}{2A}\big(B\pm\sqrt{B^2-4AC_1}\big)$. If $a_1^2=\tfrac{1}{2A}\big(B+\sqrt{B^2-4AC_1}\big)$, then $a_1^2>a_2^2$, thus Claim~\ref{claim:RHS>0} gives a contradiction. 
Therefore,
\begin{equation}\label{eq:a1}
a_1^2=\tfrac{1}{2A}\left(B-\sqrt{B^2-4AC_1}\right). 
\end{equation}
According to Claim~\ref{claim:RHS<0}, it is sufficient to show that $a_2^2>7a_1^2$. 
From \eqref{eq:a2} and \eqref{eq:a1}, it follows that this is equivalent to $6B<\sqrt{B^2-16AC_1} + 7\sqrt{B^2-4AC_1}$. 
Thus, it is sufficient to show that 
\begin{equation*}
36B^2<B^2-16AC_1+49(B^2-4AC_1)=50B^2-212AC_1,
\end{equation*}
which holds since $B^2>16AC_1$. 
\end{proof}

Finally, we are in position to prove Theorem~\ref{thm:rigidity} in the Introduction.

\begin{proof}[Proof of Theorem~\ref{thm:rigidity}]
	Consider two homogeneous metrics on CROSSes that are isospectral. Since the \emph{dimension} of a manifold is one of its spectral invariants, we may assume that these manifolds have the same dimension $d$. 
	
	We divide the proof in cases according to the congruence of $d$ modulo $4$. In each case, we prove that homogeneous metrics are determined (up to isometry) by the spectrum. We will make frequent use of the classification of homogeneous metrics on CROSSes, discussed in the Introduction, that can be found e.g.~in \cite[Ex.~6.16, 6.21]{mybook} or \cite{Ziller82}, and of Table~\ref{tab:eigenvalues}.
	Recall also that, just like its scalar curvature, each eigenvalue of the Laplacian on a closed Riemannian manifold $(M,\g)$ satisfies $\lambda_j(M,\alpha\,\g)=\frac{1}{\alpha}\lambda_j(M,\g)$ for all $\alpha>0$, and the corresponding eigenspaces are the same, so
	$\Spec(M,\alpha\, \g) = \frac{1}{\alpha}\Spec(M,\g)$.
	
	We recall from the proof of Theorem~\ref{thm:rigidityS^4n+3} that the volume and the scalar curvature of a homogeneous Riemannian manifold are spectral invariants; this fact will be also frequently used in the sequel without explicit mention.
	
	\smallskip
	\noindent\textbf{$\bullet$ Case $d\equiv 0\mod 4$:} 
	The only $d$-dimensional CROSSes are $\Ss^d$, $\R P^d$, $\C P^{d/2}$, $\Hr P^{d/4}$, and, if $d=16$, also $\Ca P^2$. Up to homotheties and isometries, there exists a \emph{unique} homogeneous metric on each of these manifolds. 
	According to Tables~\ref{tab:eigenvalues}--\ref{tab:eigenvaluesRPd}, we have that
	\begin{gather*}
		\frac{\scal(\Ss^d)}{\lambda_1(\Ss^d)}= d-1, \qquad
		\frac{\scal(\R P^d)}{\lambda_1(\R P^d)}= \frac{d(d-1)}{2(d+1)},\qquad
		\frac{\scal(\C P^{d/2})}{\lambda_1(\C P^{d/2})}= \frac{d}{2},\\
		\frac{\scal(\Hr P^{d/4})}{\lambda_1(\Hr P^{d/4})}= \frac{d(d+8)}{2d+8}, \qquad
		\frac{\scal(\Ca P^2)}{\lambda_1(\Ca P^2)}=12.
	\end{gather*}
	For $d>4$, the above quantities are all distinct, leading to a contradiction if there were two isospectral but non-isometric $d$-dimensional CROSSes. If $d=4$, then the above invariant distinguishes every possibility excepting the pair $\Ss^4$ and $\Hr P^1$, which indeed are homothetic, and therefore isometric as their volumes are the same.

	\smallskip
	\noindent\textbf{$\bullet$ Case $d\equiv 1\mod 4$:} 
	The only $d$-dimensional CROSSes are $\Ss^d$ and $\R P^d$.
	Up to homotheties and isometries, the only homogeneous metrics in each of them are $\mathbf g(t)$. 
	It is easy to see, using the explicit formulas in Tables~\ref{tab:eigenvalues}--\ref{tab:eigenvaluesRPd}, that the volume and scalar curvature of $(\Ss^d,\alpha\, \mathbf g(t_1))$ and $(\R P^d,\beta\,\mathbf g(t_2))$, $\alpha,\beta>0$, cannot coincide.
	
	We now prove that two isospectral homogeneous metrics on $\Ss^d$ are isometric. 
	According to \cite[Prop.~5.3]{bp-calcvar}, cf.~Table~\ref{tab:eigenvalues}, the first eigenvalue of $(\Ss^d,\alpha\,\mathbf g(t))$ is
	\begin{equation}\label{eq:eigenvaluesS2n+1}
		\begin{array}{lll}
			\lambda_1(\Ss^{d}, \alpha\,\mathbf g(t) )  & \multicolumn{1}{c}{\text{multiplicity}} & \text{condition} \\
			\hline 
			\rule{0pt}{13pt}
			\tfrac{2}{\alpha} (d+1)  &  \tfrac14(d-1)(d+3) & t<\frac{1}{\sqrt{d+3}} \\
			\tfrac{2}{\alpha} (d+1) & \tfrac14 (d^2+6d+1) &  t=\frac{1}{\sqrt{d+3}} \\
			\tfrac{1}{\alpha} \big(d-1+\frac{1}{t^2}\big) & d+1 & t>\frac{1}{\sqrt{d+3}} \\
		\end{array}
	\end{equation}
	Since the above multiplicities are all distinct, the expression for this eigenvalue can be read from the spectrum. Clearly, in row 2 of \eqref{eq:eigenvaluesS2n+1}, the values of $\alpha$ and $t$ are determined.
	In row 1, the value of $\alpha$ is determined from the first eigenvalue itself, and then 
	the value of $t$ can be determined by examining another spectral invariant:
	\begin{equation}\label{eq:volumeSdalphagt}
		\Vol(\Ss^d,\alpha\,\mathbf g(t)) = \frac{2\pi^{(d+1)/2}}{\big(\frac{d-1}{2}\big)!}\,t\,\alpha^{d/2}.
	\end{equation}
	
	Now assume $t>\frac{1}{\sqrt{d+3}}$, as in row 3. 
	We claim the second distinct eigenvalue~is:
	\begin{equation}\label{eq:2nd-eigenvaluesS2n+1}
		\begin{array}{lll}
			\lambda_2(\Ss^{d}, \alpha\,\mathbf g(t) )  & \multicolumn{1}{c}{\text{multiplicity}} & \text{condition} \\
			\hline 
			\rule{0pt}{13pt}
			\tfrac{2}{\alpha} (d+1)  &  \tfrac14(d-1)(d+3) & \frac{1}{\sqrt{d+3}}<t<1 \\[1mm]
			\tfrac{2}{\alpha} (d+1) & \tfrac d2(d+3) &  t=1 \\[1mm]
			\tfrac{1}{\alpha} \big(2d-2+\frac{4}{t^2}\big) & \tfrac14(d+1)(d+3) & t>1 \\
		\end{array}
	\end{equation}
	Since the above multiplicities are all distinct, the spectrum determines the expression for this second eigenvalue. In row 2 of \eqref{eq:2nd-eigenvaluesS2n+1}, both $\alpha$ and $t$ are immediately determined. In row 1, the value of $\alpha$ can be read from the eigenvalue itself, and then the value of $t$ is determined by the volume \eqref{eq:volumeSdalphagt}. 
	In row 3, the quantity
	\begin{equation*}
		\tfrac12 \lambda_2(\Ss^{d}, \alpha\,\mathbf g(t) )- \lambda_1(\Ss^{d}, \alpha\,\mathbf g(t) ) 
		= \tfrac{2}{\alpha t^2}
	\end{equation*}
	is known, as well as $t^{2}\alpha^{d}$ by the volume \eqref{eq:volumeSdalphagt}, thus $t$ and $\alpha$ are again determined.

	We now prove that \eqref{eq:2nd-eigenvaluesS2n+1} holds, using  the partial description of $\Spec(\Ss^{d}, \mathbf g(t) )$ in \cite[\S4]{tanno1} and \cite[\S{}5]{bp-calcvar}, which states that 
	every eigenvalue is of the form 
	\begin{equation*}
		\mu_{k,l}(t) = k(k+d-1)+(\tfrac{1}{t^2}-1)l^2,
	\end{equation*}
	for integers $0\leq l\leq k$ with $k\equiv l\mod 2$.
	Note that $\lambda_1(\Ss^{d}, \mathbf g(t) ) = \mu_{1,1}(t)$ under the assumption $t>\frac{1}{\sqrt{d+3}}$.  
	It is easy to see that $\lambda_2(\Ss^{d}, \mathbf g(t) )  =\min\left\{ \mu_{2,0}(t), \mu_{2,2}(t) \right\}$.
	In the notation of \cite{bp-calcvar}, its multiplicity is  $\dim E_2^0$ if $\mu_{2,0}(t)<\mu_{2,2}(t)$, $\dim E_2^2$ if $\mu_{2,0}(t)>\mu_{2,2}(t)$, and $\dim E_2=\dim (E_2^0\oplus E_2^2)$ if $\mu_{2,0}(t)=\mu_{2,2}(t)$,
	where $E_2$ is the space of complex harmonic homogeneous quadratic polynomials in $d+1$ variables.
	Thus, $\dim E_2= \frac{d(d+3)}{2}$, and $\dim E_{2}^2 = \dim E_2-\dim E_2^0 = \frac{1}{4}(d+1)(d+3)$, since $\dim E_{2,0}= \tfrac14(d-1)(d+3)$ by \cite[\S5(a)]{tanno1}, concluding the proof of \eqref{eq:2nd-eigenvaluesS2n+1}.
	
	A very similar procedure shows that any isospectral homogeneous metrics on $\R P^d$ must be isometric.

	\smallskip
	\noindent\textbf{$\bullet$ Case $d\equiv 2\mod 4$:} 
	The case $d=2$ is easy and left to the reader. 
	Assume $d\geq 6$. 
	The only $d$-dimensional CROSSes are $\Ss^d$, $\R P^d$, and $\C P^{d/2}$.
	Up to homotheties, the only homogeneous metrics are
	$\gr$ on $\Ss^d$ and $\R P^d$, and $\check{\mathbf h}(t)$ on $\C P^{d/2}$.
	By Theorem~\ref{thm:lambda_1(b,s)}, the first eigenvalue of $(\C P^{d/2}, \alpha\,\check{\mathbf h}(t))$ is  as follows, see also ~\eqref{eq:check-mu_kl}:
	\begin{equation}\label{eq:eigenvaluesCP2n+1}
		\begin{array}{lll}
			\lambda_1(\C P^{d/2}, \alpha\,\check{\mathbf h}(t) )  & \multicolumn{1}{c}{\text{multiplicity}} & \text{condition} \\
			\hline 
			\rule{0pt}{13pt}
			\tfrac{2}{\alpha} (d+2)  &  \tfrac18 (d+4)(d-2)  & t<1 \\[3pt]
			\tfrac{2}{\alpha} (d+2) & \tfrac14 d(d+4) &  t=1 \\[3pt]
			\tfrac{2}{\alpha} \big(d-2+\frac{4}{t^2}\big) & \tfrac18 (d+4)(d+2)  & t>1 \\
		\end{array}
	\end{equation}
	The multiplicities of $\lambda_1(\Ss^d,\beta\,\gr)=\frac{d}{\beta}$ and $\lambda_1(\R P^d,\beta'\,\gr)=\frac{2(d+1)}{\beta}$ are $d+1$ and $\binom{d+2}{2}-1=\frac{d(d+2)}2$ respectively, which are different from each other and from all the multiplicities in \eqref{eq:eigenvaluesCP2n+1}. Thus, for any positive numbers $\alpha$, $\beta$, and $\beta'$, we have that $(\Ss^d,\beta\,\gr)$, $(\R P^d,\beta'\,\gr)$ and $(\C P^{d/2}, \alpha\,\check{\mathbf h}(t))$ are pairwise non-isospectral for any fixed $t>0$. It is only left to show that there are no isospectral non-isometric members in the latter family.

	Since none of the multiplicities in \eqref{eq:eigenvaluesCP2n+1} coincide for $d\geq6$, the expression for this eigenvalue is determined by the spectrum. In row 2, there is nothing to be done, since the values of $\alpha$ and $t$ are determined. In row 1, the value of $\alpha$ is determined by the first eigenvalue, and then the value of $t$ can be determined through another spectral invariant, such as
	\begin{equation}\label{eq:volumeCPd2}
		\Vol(\C P^{d/2}, \alpha\,\check{\mathbf h}(t) )= \frac{\pi^{d/2}}{\big(\frac{d}{2}\big)!} t^2 \, \alpha^{d/2}.
	\end{equation}
	Now suppose $t>1$, as in row 3. 
	From the description of $\Spec(\C P^{d/2}, \alpha\,\check{\mathbf h}(t))$ in Theorem~\ref{thm:SpecCP^2n+1}, it is straightforward to check that the second distinct eigenvalue is $\lambda_2(\C P^{d/2},\alpha\,\check{\mathbf h}(t))=\frac{2}{\alpha}(d+2)$, with multiplicity $\frac18(d+4)(d-2)$, since 
	$\check\lambda^{(p,q)}\big((\sqrt2 t)^{-1},1\big) > \check\lambda^{(1,1)}\big((\sqrt2 t)^{-1},1\big) =2(d+2)$ for all $p,q$ satisfying $p\geq q\geq 0$, with $p-q$ is even, and $(p,q)\notin\{ (0,0), (2,0), (1,1)\}$.
	Similarly to row 1, the values of $\alpha$ and $t$ are uniquely determined by this expression together with \eqref{eq:volumeCPd2}.

	\smallskip
	\noindent\textbf{$\bullet$ Case $d\equiv 3\mod 4$:} 
	The only $d$-dimensional CROSSes are $\Ss^{d}$ and $\R P^d$. 
	Up to homotheties and isometries, the only homogeneous metrics on either $\Ss^{d}$ or $\R P^d$ are $\mathbf h(t_1,t_2,t_3)$, and also $\mathbf k(t)$ if $d=15$. 
	Indeed, recall that $(\Ss^d,\mathbf g(t))$ and $(\R P^d,\mathbf g(t))$ are isometric to $(\Ss^d,\mathbf h(t,1,1))$ and $(\R P^d,\mathbf h(t,1,1))$, respectively, so we may disregard the family of metrics $\mathbf g(t)$. 
	
For $d=3$, the non-existence of isospectral and non-isometric pairs of $\Sp(n+1)$-invariant metrics on $\Ss^{d}$ (resp.\ $\R P^d$) has been proved independently in \cite[Thm.~1.5]{Lauret-SpecSU(2)} and \cite[Thm.~1.3]{LSS2}.
Furthermore, Proposition~\ref{lem:isospS-RP} shows that a homogeneous $\Ss^3$ cannot be isospectral to a homogeneous $\R P^3$.

Assume henceforth that $d>3$. By Theorems~\ref{thm:rigidityS^4n+3} and \ref{thm:rigidityRP^4n+3}, two isospectral $\Sp(\tfrac{d+1}{4})$-invariant metrics on either $\Ss^{d}$ or $\R P^{d}$ are in fact isometric. Furthermore, Proposition~\ref{lem:isospS-RP} implies that any $\Sp(\tfrac{d+1}{4})$-invariant metric on $\Ss^d$ is not isospectral to any $\Sp(\tfrac{d+1}{4})$-invariant metric on $\R P^{d}$. 
Consequently, the result follows for $d\neq 15$.

From now on, we work exclusively in dimension $d=15$. 	
We first show that the spectrum of $(\Ss^{15}, \beta\, \mathbf k(t) )$ determines $\beta$ and $t$. 
	We analyze its first eigenvalue, see~\cite[\S 7]{bp-calcvar}.
	\begin{equation}\label{eq:eigenvaluesS15}
		\begin{array}{lll}
			\lambda_1(\Ss^{15}, \beta\, \mathbf k(t) )  & \multicolumn{1}{c}{\text{multiplicity}} & \text{condition} \\
			\hline 
			\rule{0pt}{13pt}
			\tfrac{32}{\beta}  & 9  & t<\sqrt{\frac{7}{24}} \\[3pt]
			\tfrac{32}{\beta} & 25 &  t=\sqrt{\frac{7}{24}} \\[3pt]
			\frac{1}{\beta}\big(8+\tfrac{7}{t^2}\big) & 16 & t>\sqrt{\frac{7}{24}} \\
		\end{array}
	\end{equation}
	Since the above multiplicities are all distinct, the spectrum determines the expression for this first eigenvalue. In row 2, both $\beta$ and $t$ are automatically determined. In row 1, the value of $\beta$ can be read from the first eigenvalue, and then the value of $t$ can be determined through another spectral invariant such as
	\begin{equation}\label{eq:volS15}
		\Vol(\Ss^{15}, \beta\, \mathbf k(t) ) = \frac{2\pi^8}{7!} t^7\,\beta^{15/2}.
	\end{equation}
	Now assume $t>\sqrt{\frac{7}{24}}$, as in row 3. 
	We claim that the second distinct eigenvalue is
	\begin{equation}\label{eq:2nd-eigenvaluesS15}
		\begin{array}{lll}
			\lambda_2(\Ss^{15}, \beta\, \mathbf k(t) )  & \multicolumn{1}{c}{\text{multiplicity}} & \text{condition} \\
			\hline 
			\rule{0pt}{13pt}
			\tfrac{32}{\beta}   &  9 &\sqrt{\frac{7}{24}}<t<1\\[1mm]
			\tfrac{32}{\beta}  & 135 &  t=1 \\[1mm]
			\tfrac{16}{\beta} \big(1+\frac{1}{t^2}\big) & 126 & t>1 \\
		\end{array}
	\end{equation}
	Since the above multiplicities are all distinct, the spectrum once again determines the expression for this second eigenvalue. In row 2, both $\beta$ and $t$ are immediately determined. In row 1, the value of $\beta$ can be read from $\lambda_2(\Ss^{15}, \beta\, \mathbf k(t) )$, and then the value of $t$ is determined by the volume \eqref{eq:volS15}. 
	In row 3, the quantity
	\begin{equation*}
		\tfrac12 \lambda_2(\Ss^{15}, \beta\, \mathbf k(t) )- \lambda_1(\Ss^{15}, \beta\, \mathbf k(t) ) = \tfrac{1}{\beta t^2}
	\end{equation*}
	is determined, as well as $t^{14}\beta^{15}$ by the volume, hence $t$ and $\beta$ are both determined. 
	
	We now prove \eqref{eq:2nd-eigenvaluesS15} using the partial description of $\Spec(\Ss^{15}, \mathbf k(t) )$ given in \cite[\S{}7.1]{bp-calcvar}. 
	According to \cite[Lem.~7.1]{bp-calcvar}, every eigenvalue is of the form 
	\begin{equation}
		\widetilde\mu_{k,l}(t) = k(k+14)+(\tfrac{1}{t^2}-1)l(l+6)
	\end{equation}
	for integers $0\leq l\leq k$ with $k\equiv l\mod 2$.
	Note that $\lambda_1(\Ss^{15}, \mathbf k(t) ) = \widetilde\mu_{1,1}(t)$ under the assumption $t>\sqrt{\frac{7}{24}}$.  
	One easily sees that $\lambda_2(\Ss^{15}, \mathbf k(t) )  =\min\left\{ \widetilde\mu_{2,0}(t), \widetilde\mu_{2,2}(t) \right\}$.
	Moreover, with the notation of \cite[\S7]{bp-calcvar}, its multiplicity is equal to $\dim E_2^0$ if $\widetilde\mu_{2,0}(t)<\widetilde\mu_{2,2}(t)$, $\dim E_2^2$ if $\widetilde\mu_{2,0}(t)>\widetilde\mu_{2,2}(t)$, and $\dim E_2=\dim (E_2^0\oplus E_2^2)$ if $\widetilde\mu_{2,0}(t)=\widetilde\mu_{2,2}(t)$, where $E_2$ is the vector space of complex harmonic homogeneous quadratic polynomials in $16$ variables. Thus, $\dim E_2= 135$, and $\dim E_{2}^2 = \dim E_2-\dim E_2^0 = 135-9=126$, since $\dim E_{2,0}= 9$, concluding the proof of \eqref{eq:2nd-eigenvaluesS15}. 
In a very similar way one shows that the spectrum of $(\R P^{15}, \beta\, \mathbf k(t) )$ determines $\beta$ and $t$.

We next show that $(\Ss^{15},\alpha\,\mathbf h(t_1,t_2,t_3))$ is not isospectral to $(\Ss^{15},\beta\,\mathbf k(t))$, unless $t=t_1=t_2=t_3=\alpha/\beta=1$; that is, unless both metrics have constant sectional curvature. 
The only way in which the multiplicity of $\lambda_1(\Ss^{15}, \beta\, \mathbf k(t))$, listed above in \eqref{eq:eigenvaluesS15}, may coincide with the multiplicity of $\lambda_1(\Ss^{15}, \alpha\, \mathbf h(t_1,t_2,t_3))$, obtained setting $n=3$ in \eqref{eq:mult-lambda_1(a,b,c,s)}, is if they are both equal to $16$. 
Namely, this is the case in row 3 of \eqref{eq:eigenvaluesS15}
and row 1 of \eqref{eq:mult-lambda_1(a,b,c,s)}. 
In this situation, consider the second eigenvalue of both manifolds, which for $(\Ss^{15}, \beta\, \mathbf k(t) )$ is 
given in \eqref{eq:2nd-eigenvaluesS15}, and for $(\Ss^{15}, \alpha\,\mathbf h(t_1,t_2,t_3))$ is given in \eqref{eq:table2ndeigenvalue} by setting $n=3$ and multiplying the values (in the first column) by $\frac{1}{\alpha}$. 
In particular, the only case where the multiplicities of $\lambda_2(\Ss^{15}, \beta\, \mathbf k(t))$ and $\lambda_2(\Ss^{15}, \alpha\, \mathbf h(t_1,t_2,t_3))$ could possibly coincide  is if they are equal to  $135$, in which case $t=1$ by \eqref{eq:2nd-eigenvaluesS15}, and $t_1=t_2=t_3=1$, from $\lambda_1^{(2,0)} = \lambda_2^{(2,0)} = \lambda_3^{(2,0)} = \lambda_1^{(1,1)}$ in \eqref{eq:table2ndeigenvalue}. 
Comparing the volumes, one easily obtains that $\alpha=\beta$, so $(\Ss^{15}, \beta\, \mathbf k(t))$ and $(\Ss^{15}, \alpha\, \mathbf h(t_1,t_2,t_3))$ are isometric round spheres.
Once more, similar arguments show that $(\R P^{15},\alpha\,\mathbf h(t_1,t_2,t_3))$ is not isospectral to $(\R P^{15},\beta\,\mathbf k(t))$, unless $t=t_1=t_2=t_3=\alpha/\beta=1$. The last remaining cases; namely,
showing that $(\Ss^{15},\alpha\,\mathbf h(t_1,t_2,t_3))$ and $(\R P^{15},\alpha\,\mathbf h(t_1,t_2,t_3))$ are not isospectral to $(\R P^{15},\beta\,\mathbf k(t))$ and $(\Ss^{15},\beta\,\mathbf k(t))$, respectively, 
are also analogous to the above, and their proofs are omitted.
\end{proof}

\section{Stability in the Yamabe Problem}\label{sec:yamabe}

As another application of Theorem~\ref{thm:A}, we now analyze which homogeneous metrics on $\Ss^{4n+3}$ and $\R P^{4n+3}$ are \emph{stable} solutions to the Yamabe problem, proving Theorem~\ref{mainthm:stability3param}. Combined with results in \cite{bp-calcvar,Lauret-SpecSU(2)} and Remark~\ref{rem:cp2n1stability}, this completes the classification of Yamabe stable homogeneous CROSSes, see Table~\ref{tab:stable}.

\subsection{Yamabe problem}\label{subsec:yamabe}
In order to keep the paper as self-contained as possible, we now briefly recall a few basic facts about the Yamabe problem; for more details see, e.g., \cite{aubin-book,bp-calcvar,LPZ12,lee-parker}.

Given a closed Riemannian manifold $(M,\g_0)$ of dimension $n\geq3$, the Yamabe problem consists of finding metrics $\g$ in the conformal class $[\g_0]$ with constant scalar curvature, which is equivalent to finding critical points of the (normalized) total scalar curvature functional
\begin{equation}\label{eq:HilbertEinstein}
\mathcal A\colon [\g_0]\to\R, \quad \mathcal A(\g)=\Vol(M,\g)^{\frac{2-n}{n}}\int_M \scal(\g)\,\vol_\g.
\end{equation}
A homogeneous metric $\g_0$ is clearly a solution to the Yamabe problem in its conformal class. 
Moreover, homogeneous metrics (invariant under the same transitive group action) that are conformal must be homothetic, so any other solutions to the Yamabe problem in $[\g_0]$ that have the same volume as $\g_0$ must be inhomogeneous.

The second variation of \eqref{eq:HilbertEinstein} at a solution $\g\in[\g_0]$ with $\Vol(M,\g)=1$ is
\begin{equation*}
\dd^2 \mathcal A(\g)(\psi,\psi)=\frac{n-2}{2}\int_M \big((n-1)\Delta_\g\psi-\scal(\g) \psi\big)\psi\,\vol_\g,
\end{equation*}
which is hence represented by the \emph{Jacobi operator} $J_\g\colon L^2(M,\g)\to L^2(M,\g)$
\begin{equation}\label{eq:jacobi-operator}
J_\g=\Delta_\g -\frac{\scal(\g)}{n-1}.
\end{equation}
Thus, $\g$ is a \emph{nondegenerate} solution if $\ker(J_\g)=\{0\}$, that is, if $\frac{\scal(\g)}{n-1}$ is not an eigenvalue of the Laplacian on $(M^n,\g)$; and $\g$ is a \emph{stable} nondegenerate solution if $\lambda_1(J_\g)>0$, that is, if $\lambda_1(\Delta_\g)>\frac{\scal(\g)}{n-1}$. In this case, $\g$ is a strict local minimum for the functional \eqref{eq:HilbertEinstein}, hence locally the unique solution to the Yamabe problem. More generally, the \emph{Morse index} of a solution $\g$ is
\begin{equation}\label{eq:iMorse}
i_{\textrm{Morse}}(\g)=\#\big\{\lambda\in\Spec(\Delta_\g)\smallsetminus\{0\}:(n-1)\lambda<\scal(\g)\!\big\},
\end{equation}
where nonzero eigenvalues $\lambda\in\Spec(\Delta_\g)$ are counted with multiplicity.
In particular, stable solutions $\g$ are precisely those with $i_{\textrm{Morse}}(\g)=0$.

\subsection{\texorpdfstring{Permutation action on $\R^3_{>0}$}{Permutation action}}
Let us collect some elementary facts that will be used in the sequel on the representation of the permutation group $\mathfrak S_3$ of three letters on the positive octant $\R^3_{>0}=\{(x,y,z)\in\R^3 : x>0, \, y>0,\,  z>0\}$, given by permuting the coordinates $(x,y,z)$.
Consider the open fundamental domain $$\mathcal D=\{(x,y,z)\in\R^3_{>0}:0<x<y<z\}$$ for this orthogonal $\mathfrak S_3$-action, and the polynomial map $\Phi\colon\mathcal D\to\R^3_{>0}$ given by
\begin{equation}\label{eq:ElSymPol}
\Phi(x,y,z)=\big(x+y+z, \, xy+xz+yz,\, xyz\big),
\end{equation}
that is, $\Phi(x,y,z)=(\sigma_1,\sigma_2,\sigma_3)$, where $\sigma_i=\sigma_i(x,y,z)$ is the $i$th elementary symmetric polynomial in $(x,y,z)$. Recall that $\Phi(x,y,z)$ are the coefficients, with alternating sign, of the monic univariate polynomial $m(r)=r^3-\sigma_1 r^2+\sigma_2 r-\sigma_3$ whose roots are $x,y,z$. In particular, the image $\Phi(\mathcal D)\subset\R^3_{>0}$ is the subset where the discriminant $\Delta=(x-y)^2(x-z)^2(y-z)^2$ of the cubic polynomial $m(r)$ is positive,
\begin{equation*}\label{eq:ImagePhi}
\Phi(\mathcal D)=\big\{(\sigma_1,\sigma_2,\sigma_3)\in\R^3_{>0}:\Delta=\sigma_1^2\sigma_2^2-4\sigma_2^3-4\sigma_1^3\sigma_3-27\sigma_3^2+18\sigma_1\sigma_2\sigma_3>0\big\},
\end{equation*}
cf.~Procesi~\cite{procesi}, keeping in mind that a $3\times 3$ Bezoutiant matrix is positive-definite if and only if its determinant (which equals the discriminant $\Delta$) is positive.

Since $\det(\dd\Phi(x,y,z))=(x-y)(x-z)(y-z)<0$ on $\mathcal D$, it follows that \eqref{eq:ElSymPol} is a diffeomorphism onto its image $\Phi(\mathcal D)$.
Finally, any closed subset $C\subset\R^3_{>0}$ with nonempty interior and invariant under the $\mathfrak S_3$-action can be decomposed as
\begin{equation}\label{eq:DecompC}
C=\overline{\bigcup_{g\in\mathfrak S_3} C\cap g(\mathcal D)}=\overline{\bigcup_{g\in\mathfrak S_3} g(C\cap \mathcal D)}.
\end{equation}

\subsection{Stability}
Henceforth, we assume that $n\geq1$.
The Riemannian submersion $\big(\Ss^{4n+3},\mathbf h(t_1,t_2,t_3)\big)\to (\Hr P^n,\gFS)$ has totally geodesic fibers and its \emph{$A$-tensor} (see e.g.~\cite[Def.~9.20]{Besse}) has square norm $\|A\|^2=4n  \left(t_1^2 + t_2^2 + t_3^2\right)$. Thus, by the Gray--O'Neill formula \cite[Prop.~9.70]{Besse}, we have
\begin{align}\label{eq:scal4n3}
\scal\!\big(\Ss^{4n+3},\mathbf h(t_1,t_2,t_3)\big)&=\scal(\Hr P^n,\gFS)+\scal\!\big(\Ss^{3},\mathbf h(t_1,t_2,t_3)\big)-\|A\|^2\nonumber \\
&= 16n(n+2) + 4\left(\frac{1}{t_1^2}+\frac{1}{t_2^2}+\frac{1}{t_3^2}\right) \\
&\quad - 2\left(\frac{t_1^2}{t_2^2t_3^2} +\frac{t_2^2}{t_1^2t_3^2}+\frac{t_3^2}{t_1^2t_2^2}\right) - 4n  \left(t_1^2 + t_2^2 + t_3^2\right).\nonumber
\end{align}

The scalar curvature of $\big(\R P^{4n+3},\mathbf h(t_1,t_2,t_3)\big)$ is identical, since these manifolds are locally isometric.
We are now ready to prove Theorem~\ref{mainthm:stability3param} in the Introduction.

\begin{proof}[Proof of Theorem~\ref{mainthm:stability3param}]
First, let us consider the case of $\Ss^{4n+3}$.
As discussed above, $\mathbf h(t_1,t_2,t_3)$ is a stable nondegenerate solution to the Yamabe problem if and only~if
\begin{equation}\label{eq:YamabeStability}
\lambda_1(\Ss^{4n+3},\mathbf h(t_1,t_2,t_3))-\frac{\scal(\Ss^{4n+3},\mathbf h(t_1,t_2,t_3))}{4n+2}>0.
\end{equation}
Our computations are significantly simplified by making the change of variables
\begin{equation}\label{eq:squarefree}
(x,y,z)=\big(t_1^2,t_2^2,t_3^2\big),
\end{equation}
which is a diffeomorphism of $\R^3_{>0}$. In terms of these  variables, by \eqref{eq:scal4n3}, we have
\begin{equation}\label{eq:scal4n3xyz}
\begin{aligned}
\scal(x,y,z)&:=\scal\big(\Ss^{4n+3},\mathbf h(t_1,t_2,t_3)\big) \\
&\;= 16n(n+2) + 4\left(\frac{1}{x}+\frac{1}{y}+\frac{1}{z}\right) - 2\left(\frac{x}{yz} +\frac{y}{xz}+\frac{z}{xy}\right)\\ &\;\quad - 4n  \left(x+y+z\right),
\end{aligned}
\end{equation}
and, from Theorem~\ref{thm:A}, we have
\begin{equation}\label{eq:lambda1xyz}
\lambda_1(x,y,z):=\lambda_1\big(\Ss^{4n+3},\mathbf h(t_1,t_2,t_3) \big)= 
\min \left\{ \lambda^{(1,0)},\; \lambda^{(2,0)},\; \lambda^{(1,1)}\right\},
\end{equation}
where
\begin{equation}\label{eq:lambdasxyz}
\begin{aligned}
\lambda^{(1,0)}(x,y,z)&=4n +\frac{1}{x}+\frac{1}{y}+\frac{1}{z},&\\
\lambda^{(2,0)}(x,y,z)&=8n+ \frac{4}{y} + \frac{4}{z},& (\text{if }x<y<z)\\
\lambda^{(1,1)}(x,y,z)&=8(n+1).&
\end{aligned}
\end{equation}

First, we claim that $\scal(x,y,z)\leq (4n+2)\lambda^{(1,0)}(x,y,z)$, with equality holding if and only if $(x,y,z)=(1,1,1)$. Indeed, let us find the infimum of $\phi\colon\R^3_{>0}\to\R$,
\begin{equation*}
\begin{aligned}
\phi(x,y,z)&:=\tfrac12\big((4n+2)\lambda^{(1,0)}(x,y,z)-\scal(x,y,z)\big)xyz\\
&\;=x^2+y^2+z^2+(2n-1)(xy+xz+yz)+2n(x+y+z-6)xyz,
\end{aligned}
\end{equation*}
which is clearly invariant under the permutation action of $\mathfrak S_3$ on $\R^3_{>0}$, and extends to a polynomial map $\phi\colon\R^3_{\geq0}\to\R$.
Rewriting $\phi(x,y,z)$ in terms of elementary symmetric polynomials $\sigma_i$, that is, precomposing with the inverse $\Phi^{-1}\colon\Phi(\mathcal D)\to\mathcal D$ of the diffeomorphism \eqref{eq:ElSymPol}, we have
\begin{equation*}
(\phi\circ\Phi^{-1})(\sigma_1,\sigma_2,\sigma_3)=\sigma_1^2+(2n-3)\sigma_2-12n\sigma_3+2n\sigma_1\sigma_3,
\end{equation*}
which clearly has no critical points in $\Phi(\mathcal D)\subset\R^3_{>0}$, since its partial derivative with respect to $\sigma_2$ never vanishes. Therefore, $\phi(x,y,z)$ does not have any critical points in $\mathcal D$, or in $g(\mathcal D)$ for any $g\in\mathfrak S_3$, since $\mathfrak S_3$ acts by diffeomorphisms. Moreover, since
\begin{equation*}
\R^3_{>0}\!\smallsetminus\!\bigcup_{g\in\mathfrak S_3} g(\mathcal D)=\{x=y>0,\, z>0\}\cup\{x=z>0,\, y>0\}\cup\{y=z>0,\, x>0\},
\end{equation*}
it follows that any interior critical points $(x_0,y_0,z_0)\in\R^3_{>0}$ of $\phi(x,y,z)$ must have at least two equal coordinates. Restricting $\phi$ to the above subsets, it is easy to see that there are only two such critical points: the saddle point $(\frac12,\frac12,\frac12)$, and the local minimum $(1,1,1)$, where $\phi(1,1,1)=0$. Finally, it is straightforward that $\phi(x,y,z)\geq0$ on the boundary of $\R^3_{\geq0}$, and also $\phi(x,y,z)\geq0$ for all $(x,y,z)\in\R^3_{\geq0}$ with $x+y+z\geq6$, so $\phi\colon\R^3_{\geq0}\to\R$ attains its minimum on the compact set $\{(x,y,z)\in\R^3_{\geq0}:x+y+z\leq 6\}$, namely, at $(1,1,1)$.
Thus, $\phi(x,y,z)\geq0$ in $\R^3_{>0}$, with equality if and only if $(x,y,z)=(1,1,1)$, proving the claim above.

Second, we claim that $\scal(x,y,z)<(4n+2)\lambda^{(2,0)}(x,y,z)$ for all $(x,y,z)\in\R^3_{>0}$. This follows easily since $\psi(x,y,z):=\tfrac12\big((4n+2)\lambda^{(2,0)}(x,y,z)-\scal(x,y,z)\big)xyz$ satisfies
\begin{equation*}
\begin{aligned}
\psi(x,y,z)&=x^2+(y-z)^2+2x(y+z)\\
&\quad+2n(x+y+z)xyz+8nx(y+z+(n-1)yz)>0.
\end{aligned}
\end{equation*}

Therefore, \eqref{eq:YamabeStability} is equivalent to $(x,y,z)\neq(1,1,1)$ and
\begin{equation*}
\scal(x,y,z)<(4n+2)\lambda^{(1,1)}(x,y,z)=16(2n+1)(n+1).
\end{equation*}
In turn, by \eqref{eq:scal4n3xyz}, the above inequality is equivalent to $p(x,y,z)>0$, where
\begin{equation}\label{eq:sigma-rootfree}
\begin{aligned}
p(x,y,z)&:=x^2+y^2+z^2-2(xy+xz+yz)\\
&\quad +2n(x+y+z)xyz+8(n^2+n+1)xyz.
\end{aligned}
\end{equation}
This algebraically characterizes which spheres $\big(\Ss^{4n+3},\mathbf h(t_1,t_2,t_3)\big)$, $n\geq1$, are stable nondegenerate solutions to the Yamabe problem; after the change of variables \eqref{eq:squarefree}, this is precisely the characterization claimed in Theorem~\ref{mainthm:stability3param}.

This characterization carries over \emph{verbatim} to $\big(\R P^{4n+3}, \mathbf h(t_1,t_2,t_3)\big)$, $n\geq1$. Indeed, $\mathbf h(t_1,t_2,t_3)$ is stable and nondegenerate on $\R P^{4n+3}$ if and only if
\begin{equation}\label{eq:stabRPn}
\lambda_1\big(\R P^{4n+3},\mathbf h(t_1,t_2,t_3) \big)-\frac{\scal\big(\R P^{4n+3},\mathbf h(t_1,t_2,t_3) \big)}{4n+2}>0,
\end{equation}
cf.~\eqref{eq:YamabeStability}; and, since $\big(\R P^{4n+3},\mathbf h(t_1,t_2,t_3) \big)$ is locally isometric to $\big(\Ss^{4n+3},\mathbf h(t_1,t_2,t_3) \big)$, they have the same scalar curvature. Moreover, from Theorem~\ref{thm:A},
\begin{equation*}
\lambda_1\big(\R P^{4n+3},\mathbf h(t_1,t_2,t_3) \big)= \min \left\{ \lambda^{(2,0)},\; \lambda^{(1,1)}\right\},
\end{equation*}
where $\lambda^{(2,0)}$ and $\lambda^{(1,1)}$ are again as in \eqref{eq:lambdasxyz}. 
If $\lambda^{(1,0)} < \min \left\{ \lambda^{(2,0)},\; \lambda^{(1,1)}\right\}$, then \eqref{eq:stabRPn} holds because its left-hand side is $> \phi(x,y,z) / (2n+1)xyz \geq0$.
Meanwhile, if $\lambda^{(1,0)}\geq\min \left\{ \lambda^{(2,0)},\; \lambda^{(1,1)}\right\}$, then $\lambda_1\big(\R P^{4n+3},\mathbf h(t_1,t_2,t_3) \big)=\lambda_1\big(\Ss^{4n+3},\mathbf h(t_1,t_2,t_3) \big)$, so \eqref{eq:stabRPn} holds if and only if \eqref{eq:YamabeStability} holds, i.e., if and only if $p(x,y,z)>0$.

We now analyze the (topological) boundary 
\begin{equation}\label{eq:sigman}
\Sigma_n:=p^{-1}(0)
\end{equation}
of the semialgebraic open subset $\{(x,y,z)\in\R^3_{>0}:p(x,y,z)>0\}$. All claims in Theorem~\ref{mainthm:stability3param} about $\partial\mathcal S_n$ will be proved in terms of $\Sigma_n$, since these sets are mapped to one another by the (orientation-preserving) diffeomorphism \eqref{eq:squarefree} of $\R^3_{>0}$.

Since $p(x,y,z)$ is clearly invariant under the action of the permutation group $\mathfrak S_3$ on $\R^3_{>0}$, so is its zero set $\Sigma_n$. Rewriting \eqref{eq:sigma-rootfree} in terms of $\sigma_i$, one easily sees that the image $\Phi(\Sigma_n\cap\mathcal D)\subset\R^3_{>0}$ under the diffeomorphism \eqref{eq:ElSymPol} is the portion inside $\Phi(\mathcal D)$ of the graph of a smooth function of $\sigma_1$ and $\sigma_3$, namely,
\begin{equation}\label{eq:SigmaSigmas}
\sigma_2=\sigma_2(\sigma_1,\sigma_3)=\frac{\sigma_1^2}{4}+\frac{n}{2}\sigma_1 \sigma_3+2(n^2+n+1)\sigma_3,
\end{equation}
and hence a smooth, connected, embedded surface in the open subset $\Phi(\mathcal D)\subset \R^3_{>0}$, diffeomorphic to $\R^2_{>0}$. Therefore, also $\Sigma_n\cap\mathcal D$, as well as $\Sigma_n\cap g(\mathcal D)=g(\Sigma_n\cap\mathcal D)$, for any $g\in\mathfrak S_3$, are smooth, connected, embedded surfaces in~$\R^3_{>0}$, diffeomorphic to $\R^2_{>0}$. 
Since the $\mathfrak S_3$-action on $\R^3_{>0}$ is generated by reflections across the planes $x=y$, $x=z$, and $y=z$, in order to conclude that $\Sigma_n$ itself is a smooth, connected, embedded surface in $\R^3_{>0}$, using \eqref{eq:DecompC} with $C=\Sigma_n$, it suffices to show the following:
\begin{enumerate}
\item $\Sigma_n\cap\overline{\mathcal D}=\overline{\Sigma_n\cap\mathcal D}$ in $\R^3_{>0}$;
\item $\Sigma_n\cap\overline{\mathcal D}$ meets the planes $x=y$ and $y=z$ orthogonally; 
\item The planar curves determined by intersecting $\Sigma_n\cap\overline{\mathcal D}$ with $x=y$ and $y=z$ arrive orthogonally at the diagonal line $x=y=z$ in each of these planes.
\end{enumerate}
All of the above can be directly verified by elementary methods, using \eqref{eq:sigma-rootfree}. In particular, it follows that the complement $\R^3_{>0}\smallsetminus\Sigma_n$ has two connected components.

Finally, let us prove that $\Sigma_n\subset\R^3_{>0}$ is bounded. 
Using $\mathfrak S_3$-invariance once again, it suffices to show that there exists $\rho>0$ such that $\Phi(\Sigma_n\cap\mathcal D)\subset \Phi(B_\rho\cap\mathcal D)$, where $B_\rho\subset\R^3_{>0}$ is the (portion in the positive octant of the) ball of radius $\rho$ around the origin. Indeed, this implies that $\Sigma_n\cap\mathcal D\subset B_\rho\cap\mathcal D$, and hence by \eqref{eq:DecompC}, since both $\Sigma_n$ and $B_\rho$ are invariant under the $\mathfrak S_3$-action, that $\Sigma_n\subset \overline{B_\rho}$.
Clearly,
\begin{equation*}
\Phi(B_\rho\cap\mathcal D)=
\big\{(\sigma_1,\sigma_2,\sigma_3)\in\R^3_{>0}:  \sigma_1^2-2\sigma_2 <\rho^2, \, \Delta>0\big\},
\end{equation*}
while, from \eqref{eq:SigmaSigmas}, the points $(\sigma_1,\sigma_2,\sigma_3)\in\Phi(\Sigma_n\cap\mathcal D)\subset\R^3_{>0}$ satisfy
\begin{equation*}
\sigma_1^2-2\sigma_2 =2\sigma_2 -2n\sigma_1 \sigma_3-8(n^2+n+1)\sigma_3< 2\sigma_2,
\end{equation*}
in addition to $\Delta>0$. In particular, it is enough to show that there exists $\rho>0$ so that $2\sigma_2\leq \rho^2$ for all $(\sigma_1,\sigma_2,\sigma_3)\in\Phi(\Sigma_n\cap\mathcal D)$, i.e., that the quadratic function $\sigma_2(\sigma_1,\sigma_3)$ defined in \eqref{eq:SigmaSigmas} is bounded in the region of $(\sigma_1,\sigma_3)\in\R^2_{>0}$ 
such that $(\sigma_1,\sigma_2(\sigma_1,\sigma_3),\sigma_3)$ satisfies $\Delta>0$.
If $\sigma_3>0$ and $\Delta>0$, then 
\begin{equation*}
\begin{aligned}
\frac{\Delta}{\sigma_3}&=\frac{1}{\sigma_3}\big(\sigma_1^2\sigma_2(\sigma_1,\sigma_3)^2-4\sigma_2(\sigma_1,\sigma_3)^3-4\sigma_1^3\sigma_3-27\sigma_3^2+18\sigma_1\sigma_2(\sigma_1,\sigma_3)\sigma_3\big)\\
&\stackrel{\eqref{eq:SigmaSigmas}}{=}- \tfrac12\!\left(n (\sigma_1+4)+4 n^2+4\right)^3\sigma_3^2 \\
&\quad -\left(\tfrac{1}{2} n^2 \sigma_1^4 -9 n \sigma_1^2   +4 \left(n^2+n+1\right)\left(n \sigma_1^2+2 \left(n^2+n+1\right) \sigma_1-9\right) \sigma_1+ 27\right)\sigma_3\\
&\quad - \left(\tfrac{1}{8} n \sigma_1^2 +\tfrac12\!\left(n^2+n+1\right) \sigma_1 -\tfrac12\right)\sigma_1^3
\end{aligned}
\end{equation*}
is also positive. For all $\sigma_1>0$, the above is a concave quadratic function of $\sigma_3$, since its leading coefficient is $< -32$.
Thus, for each $\sigma_1>0$, the quantity $\frac{\Delta}{\sigma_3}$ can only be positive for $\sigma_3$ in a bounded interval, whose endpoints depend continuously on $\sigma_1$. Moreover, such interval is nonempty if and only if the discriminant
\begin{equation*}
\left( 9 - 8 \left(n^2+n+1\right) \sigma_1 -2 n \sigma_1^2 \right)^3
\end{equation*}
of the above quadratic form in $\sigma_3$ is nonnegative, and, since $n>0$, a necessary condition for this is $0<\sigma_1<\frac{9}{8}$. Therefore, the (topological) closure of the region of $(\sigma_1,\sigma_3)\in \R^2_{>0}$ such that $(\sigma_1,\sigma_2(\sigma_1,\sigma_3),\sigma_3)$ satisfies $\Delta>0$ is compact, and hence the quadratic function $\sigma_2(\sigma_1,\sigma_3)$ is bounded in this region, as desired.
\end{proof}

\section{Bifurcation in the Yamabe Problem}\label{sec:bifurcations}

As an application of the characterization of stable homogeneous solutions to the Yamabe problem in the previous section, we now establish nonuniqueness results via Bifurcation Theory, along the lines of \cite{bp-calcvar,bp-pacific,frankenstein,LPZ12}.
Following these references, solutions to the Yamabe problem are said to \emph{bifurcate} from a curve $\g(t)$ of solutions on $M$ at $t=t_*$ if there exist a sequence of parameters $t_q$ converging to $t_*$, and constant scalar curvature metrics $\g_q\in[\g(t_q)]$ converging to $\g(t_*)$, such that $\Vol(M,\g_q)=\Vol(M,\g(t_q))$ and $\g_q\neq\g(t_q)$, for all $q\in\mathds N$.

The bifurcating solutions $\g_{q}$ typically have less symmetries than $\g(t_q)$ and are harder to find by direct methods. Standard variational bifurcation results applied to the functional \eqref{eq:HilbertEinstein} imply that bifurcation of solutions along $\g(t)$ can be detected by jumps in the Morse index \eqref{eq:iMorse} of $\g(t)$, see \cite[Thm.~3.3]{LPZ12}. 

\subsection{Bifurcations}
Let us now prove Corollary~\ref{maincor:bifurcation} in the Introduction.

\begin{proof}[Proof of Corollary~\ref{maincor:bifurcation}]
We use the notation from the proof of Theorem~\ref{mainthm:stability3param}, in terms of the variables \eqref{eq:squarefree}.
Let $\alpha\colon[-\varepsilon,\varepsilon]\to\R^3_{>0}$, 
 $\alpha(s)=\big(x(s),y(s),z(s)\big)$, be a continuous curve that crosses the surface $\Sigma_n\subset \R^3_{>0}$, see \eqref{eq:sigman}, and assume it does so only once. By Theorem~\ref{mainthm:stability3param}, the Morse index of $\mathbf h\big(\alpha(s)\big)$ jumps as $s$ goes from $-\varepsilon$ to $\varepsilon$; namely
\begin{equation*}
\Big|i_{\textrm{Morse}}\big(\mathbf h(\alpha(-\varepsilon))\big)-i_{\textrm{Morse}}\big(\mathbf h(\alpha(\varepsilon))\big)\Big|\geq 2n^2+3n\geq5,
\end{equation*}
is at least as large as the multiplicity of the eigenvalue $\lambda^{(1,1)}-\scal/(4n+2)$ of $J_{\mathbf h(\alpha(s))}$ that changes sign when $\alpha(s)$ crosses $\Sigma_n$, see \eqref{eq:mult-lambda_1(a,b,c,s)} or \eqref{eq:mult-lambda_1(a,b,c,s)projective}, and the proof of Theorem~\ref{mainthm:stability3param}.
Furthermore, we may assume without loss of generality that $\mathbf h\big(\alpha(\pm\varepsilon)\big)$ are nondegenerate, as this corresponds to $\alpha(\pm\varepsilon)\in \R^3_{>0}$ belonging to an open and dense subset (contained in the complement of $\Sigma_n$) and $i_{\textrm{Morse}}(\cdot)$ is locally constant on this set.
Under these conditions, bifurcation of solutions from $\mathbf h\big(\alpha(s)\big)$ follows from \cite[Thm.~3.3]{LPZ12}. 
Finally, the solutions bifurcating from $\mathbf h\big(\alpha(s)\big)$ are inhomogeneous since conformal homogeneous metrics are homothetic, see Subsection~\ref{subsec:yamabe}.
\end{proof}

\begin{remark}
Earlier results in \cite{bp-pacific,OtobaPetean1} imply that if $t_i>0$ are such that $\scal\!\big(\Ss^3,\mathbf h(t_1,t_2,t_3)\big)>0$, then there exists a sequence of sufficiently small $\varepsilon_k>0$, that converges to $0$, such that inhomogeneous solutions to the Yamabe problem bifurcate from $\big(\Ss^{4n+3},\mathbf h(\varepsilon_k t_1,\varepsilon_k t_2,\varepsilon_k t_3)\big)$ for all $k\in\mathds N$. However, this collapsing bifurcation result does not imply Corollary~\ref{maincor:bifurcation}.
\end{remark}

Regarding homogeneous metrics on $\C P^{2n+1}$, we have the following result:

\begin{proposition}\label{prop:cp2n1}
There are infinitely many branches of inhomogeneous solutions to the Yamabe problem on $\C P^{2n+1}$, $n\geq 1$, that bifurcate from $\check{\mathbf h}(t)$ as $t\searrow0$.
\end{proposition}

\begin{proof}
This is an instance of a general result of Otoba and Petean~\cite[Thm~1.1]{OtobaPetean1}, see~Proposition~\ref{prop:collapsebif2}. Alternatively, it can be proven using \cite[Thm.~3.3]{LPZ12} and 
Theorem~\ref{thm:SpecCP^2n+1},
to directly show that $i_{\textrm{Morse}}\big(\check{\mathbf h}(t)\big)\nearrow\infty$ as $t\searrow0$, as in \cite{bp-calcvar}.
\end{proof}

\begin{remark}\label{rem:cp2n1stability}
There is usually considerable interest in the \emph{first} bifurcation instant, which corresponds to the \emph{transition between stability and instability}, such as crossing the surface $\partial\mathcal S_n$ in Corollary~\ref{maincor:bifurcation} about $\Ss^{4n+3}$. In the case of $\big(\C P^{2n+1},\check{\mathbf h}(t)\big)$, since the equality $(4n+1)\lambda_1\big(\C P^{2n+1},\check{\mathbf h}(t)\big)=\scal\big(\C P^{2n+1},\check{\mathbf h}(t)\big)$ is only possible if the minimum in the formula for $\lambda_1\big(\C P^{2n+1},\check{\mathbf h}(t)\big)$ in Theorem~\ref{thm:B} is achieved at $8(n+1)$,
this transition happens when $t$ crosses the (first bifurcation) value
\begin{equation*}
t_*=\sqrt{\frac{\sqrt{(2n^2+n+1)^2+4n}-(2n^2+n+1)}{2n}}.
\end{equation*}
More precisely, $\check{\mathbf h}(t)$ is a stable nondegenerate solution if and only if $t>t_*$.
\end{remark}

\subsection{Degenerations}
In this subsection, we analyze the (Yamabe) stability of $\mathbf h(t_1,t_2,t_3)$ as it \emph{degenerates}, i.e., as some $t_i$ converge to either $0$ or $\infty$. 
Note that degenerations where some $t_i\nearrow\infty$ are  stable, since the subset $\R^3_{>0}\smallsetminus\mathcal S_n$ of parameters corresponding to unstable metrics is bounded, as a consequence of Theorem~\ref{mainthm:stability3param}.
Thus, we restrict ourselves to the case in which all $t_i$ remain finite, and call the number of $t_i$ that converge to $0$ the \emph{codimension} of the degeneration.

\begin{proposition}\label{prop:stability0}
The following hold about degenerations along $1$-parameter subfamilies of homogeneous metrics $\mathbf h(t_1,t_2,t_3)$ on $\Ss^{4n+3}$ and $\R P^{4n+3}$, $n\geq1$:
\begin{enumerate}[\rm (1)]
\item Degenerations of codimension $1$ or $3$ may occur through degenerate, stable, and unstable solutions, or through a combination of these;
\item Degenerations of codimension $2$ occur only through stable solutions.
\end{enumerate}
\end{proposition}

\begin{proof}
Once again, we use the notation from the proof of Theorem~\ref{mainthm:stability3param}, in terms of variables \eqref{eq:squarefree}. 
We claim that the (topological) closure of $\Sigma_n$ inside $\R^3_{\geq0}$, see \eqref{eq:sigman}, consists of the union of $\Sigma_n$ with a diagonal line segment inside each of the $3$ coordinate hyperplanes that form the boundary $\partial\R^3_{\geq0}$.
Given the $\mathfrak S_3$-symmetries, without loss of generality, we consider only the part of $\partial\R^3_{\geq0}$ where $z=0$. From \eqref{eq:sigma-rootfree}, we have that
\begin{equation}\label{eq:xy0}
p(x,y,0)=(x-y)^2,
\end{equation}
however, the accumulation points of $\Sigma_n$ only lie in a finite segment along the diagonal $x=y$,
since $\Sigma_n\subset\R^3_{>0}$ is bounded.
Solving for $z$ in the polynomial equation $p(x,x,z)=0$, and then finding its zeroes in $x$, one sees that the accumulation points of $\Sigma_n$ on the plane $z=0$ are precisely $L=\big\{(x,x,0)\in\R^3_{\geq0}: 0\leq x\leq \ell_n\big\}$, where
\begin{equation*}
\ell_n=\frac{\sqrt{(n^3+n^2+2n+1)(n+1)}-(n^2+n+1)}{n}.
\end{equation*}
Thus, the accumulation points of $\Sigma_n$ on $\partial\R^3_{\geq0}$ are the $3$ diagonal line segments of length $\ell_n$ starting at the origin, i.e., the $\mathfrak S_3$-orbit of $L$, proving the above claim.

Claim (2) now follows, as the coordinate axes only intersect this accumulation set at the origin. Claim (1) also follows, since 
$\Sigma_n$ and both connected components of its complement in $\R^3_{>0}$ have accumulation points in the complement of the coordinate axes in $\partial\R^3_{\geq0}$, as well as at the origin.
\end{proof}

\begin{remark}\label{rem:degcollapse}
Degenerations do not always correspond to \emph{collapse}, in the sense of Gromov--Hausdorff convergence to limit metric space with \emph{lower} Hausdorff dimension. 
As an illustration, consider $\big(\Ss^{3},\mathbf h(t_1,t_2,t_3)\big)$, with $0<t_1\leq t_2\leq t_3$. Since this is a class of uniformly doubling metric spaces~\cite{gordinagafa}, any sequence along which the diameter remains bounded has a Gromov--Hausdorff convergent subsequence~\cite[Prop.~11.1.12]{petersen-3ed}.
It can be shown that $\diam\!\big(\Ss^{3},\mathbf h(t_1,t_2,t_3)\big)$ remains bounded if and only if $t_2$ remains bounded, see~\cite[Prop.~7.1]{gordinagafa} or \cite[Cor.~4.4]{Lauret-SpecSU(2)}. If $t_2\searrow0$, then also $\diam\!\big(\Ss^{3},\mathbf h(t_1,t_2,t_3)\big)\searrow0$ and hence the Gromov--Hausdorff limit is a point. 
On the one hand, if $t_2$ remains away from $0$ and $t_1\searrow0$, then the limit is a round sphere $\Ss^2(t_2)$ of radius $t_2$, in which case \emph{there is collapse}. Note that, unless $t_2=t_3$, there is no uniform lower bound on the Ricci curvature as $t_1\searrow0$.
On the other hand, if $t_1$ and $t_2$ remain bounded and $t_3\nearrow\infty$, then the limit is $\Ss^3$ endowed with a (homogeneous) sub-Riemannian distance function, which is a metric space of \emph{larger} Hausdorff dimension, equal to $4$.
\end{remark}

\subsection{Bifurcations versus degenerations}
Based on the literature about bifurcation of homogeneous solutions to the Yamabe problem cited above, one intuitively expects close relations between \emph{degenerations} and \emph{accumulating bifurcations}, manifested through the Morse index blowing up. We now discuss a few such relations.

\begin{proposition}\label{prop:collapsebif}
Let $M$ be a closed manifold and
$\pi_t\colon (M,\g(t))\to B$, $\dim B\geq 1$, be a $1$-parameter family of Riemannian submersions with totally geodesic fibers isometric to $F_t$, such that $\scal(\g(t))$ is constant for all $t\in (0,1]$, $\diam(F_t)\searrow 0$ as $t\searrow0$, and $\lim_{t\searrow0}\inf \Ric(F_t)\geq \kappa$ for some $\kappa\in\R$. Then, as $t\searrow0$,
\begin{equation}\label{eq:equivimorsescal}
i_{\rm{Morse}}(\g(t))\nearrow\infty \quad \Longleftrightarrow \quad \scal(F_t)\nearrow+\infty.
\end{equation}
\end{proposition}

\begin{proof}
Suppose that $\scal(F_t)\leq C$ as $t\searrow0$. 
The scalar curvature of $\g(t)$ is given by (see \cite[Prop.~9.70]{Besse})
\begin{equation*}
\scal(\g(t))=\scal(F_t)+\scal(B)\circ\pi_t -\|A_t\|^2,
\end{equation*}
and hence is also bounded from above as $t\searrow0$.
On the other hand,
all eigenvalues of the Laplacian $\Delta_{\g(t)}$ on $(M,\g(t))$ are of the form
\begin{equation}\label{eq:eigenvaluest}
\lambda(t)=\lambda_j(F_t)+\lambda_k(B),
\end{equation}
for some $\lambda_j(F_t)\in\Spec(\Delta_{F_t})$ and $\lambda_k(B)\in\Spec(\Delta_B)$, see \cite[Thm.~3.6]{Berard-BergeryBourguignon82}. 
Although not all combinations \eqref{eq:eigenvaluest} of eigenvalues of $F_t$ and $B$ occur, there is an inclusion $\Spec(\Delta_B)\subset\Spec(\Delta_{\g(t)})$, since lifting an eigenfunction of $\Delta_B$ with eigenvalue $\lambda_k(B)$ gives an eigenfunction of $\Delta_{\g(t)}$ with same eigenvalue. These eigenvalues of $\Delta_{\g(t)}$ are called \emph{basic} and are independent of $t$.
Since $\diam(F_t)\searrow0$ as $t\searrow0$ and $F_t$ have a uniform lower bound on Ricci curvature, the L\'evy-Gromov isoperimetric inequality~\cite[Cor.~17]{bbg} implies  that 
$\lambda_1(F_t)\nearrow\infty$.
Thus, by \eqref{eq:eigenvaluest}, all non-basic eigenvalues satisfy $\lambda(t)\nearrow\infty$ as $t\searrow0$.
Therefore, if $t>0$ is sufficiently small, only \emph{basic} eigenvalues contribute to the Morse index of $\g(t)$, because $\scal(\g(t))$ is bounded, cf.~\eqref{eq:iMorse}. For the same reason, there are at most \emph{finitely many} basic eigenvalues $\lambda_k(B)$ that satisfy $(n-1)\lambda_k(B)<\scal(\g(t))$, which implies that $i_{\rm{Morse}}(\g(t))$ is bounded as $t\searrow0$.

The converse implication follows from Otoba and Petean~\cite[Thm.~4.1]{OtobaPetean1}.
\end{proof}

\begin{remark}\label{rem:diamvol}
In Proposition~\ref{prop:collapsebif}, the hypothesis $\diam(F_t)\searrow0$ as $t\searrow0$ cannot be relaxed to $\Vol(F_t)\searrow0$, as exemplified by letting $F_t$ be the Berger sphere $(\Ss^3,\mathbf g(t))$ or a flat torus $\Ss^1(t)\times\Ss^1$.
In these examples, $\lambda_1(F_t)$ remains bounded as $t\searrow0$, $\Ric(F_t)\geq0$, and $\Vol(F_t)\searrow0$, but $\diam(F_t)\to\diam(F_0)>0$. Roughly speaking, this corresponds to the fact that $\diam(F_t)\searrow0$ detects whether $F_t$ collapses in \emph{all} directions, while $\Vol(F_t)\searrow0$ only detects if $F_t$ collapses in \emph{some} direction. If the collapse $F_t\to F_0$ is sufficiently controlled (e.g., with upper and lower bounds on the sectional curvature), then $\lambda_1(F_t)\to\lambda_1(F_0)$, see \cite{fukaya}.
\end{remark}

\begin{remark}
A compact homogeneous space $M=\G/\H$ admits $\G$-invariant metrics $\g$ with $\scal>0$ if and only if $M$ is not a torus.
In this case, $M$ also admits many $1$-parameter families $\g(t)$, $t\in (0,1]$ of $\G$-invariant metrics such that, as $t\searrow0$, $\scal(\g(t))\nearrow\infty$ and $\Vol(M,\g(t))=1$, e.g., by considering (normalized) Cheeger deformations with respect to any subaction by a non-Abelian subgroup, such as $\SU(2)\subset\G$. 
In this situation, it seems natural to expect that $i_{\rm{Morse}}(\g(t))\nearrow\infty$.
In principle, confirming this would solely rely on a careful analysis of the spectrum of homogeneous spaces. Nevertheless, a proof seems currently elusive, except if $(\G/\H,\g(t))$ admits nontrivial Riemannian submersions, in which case one may use Proposition~\ref{prop:collapsebif}, see also \cite[Thm.~4.1]{bp-pacific}.
\end{remark}

Consider the canonical variation $\g(t)=t^2\g_{\rm {ver}}+\g_{\rm {hor}}$ of a Riemannian submersion $F\to M\to B$ with totally geodesic fibers, where all manifolds are closed.
In this situation, concerning the setting of Proposition~\ref{prop:collapsebif}, $\scal(\g(t))$ is constant for all $t\in (0,1]$ if and only if $\scal(B)$, $\scal(F)$, and $\|A\|^2$ are constant. Moreover, $\scal(F_t)=\frac{1}{t^2}\scal(F)$, $\diam(F_t)=t\diam(F)$, and $\lim_{t\searrow0}\inf\Ric(F_t)\geq\kappa$ for some $\kappa\in\R$ if and only if $\Ric(F)\geq0$; however, since $\lambda_1(F_t)=\frac{1}{t^2}\lambda_1(F)$, the latter is not necessary to prove the following adaptation of Proposition~\ref{prop:collapsebif} along the same lines:

\begin{proposition}\label{prop:collapsebif2}
Let $F\to M\to B$ be a Riemannian submersion with totally geodesic fibers, and $\dim B\geq 1$. Suppose $F$ and $B$ are closed manifolds with constant scalar curvature. Then the canonical variation $\g(t)$ satisfies, as $t\searrow0$,
\begin{equation}\label{eq:equivimorsescal2}
i_{\rm{Morse}}(\g(t))\nearrow\infty \quad \Longleftrightarrow \quad \scal(F)>0.
\end{equation}
\end{proposition}

Note this proves that the converse statement to \cite[Thm.~1.1]{OtobaPetean1} holds.

Let us briefly revisit the possible degenerations of $\big(\Ss^{4n+3},\mathbf h(t_1,t_2,t_3)\big)$, $n\geq1$, under the light of Propositions~\ref{prop:collapsebif} and \ref{prop:collapsebif2}. For all codimension $1$ degenerations $t_1\searrow0$, direct inspection shows the Morse index remains bounded. Note that 
Propositions~\ref{prop:collapsebif} and \ref{prop:collapsebif2} do not apply, since the diameter of $F_t=\big(\Ss^3,\mathbf h(t,t_2,t_3)\big)$
does not converge to $0$, see Remarks~\ref{rem:diamvol} and \ref{rem:degcollapse}, and unless $t_2=t_3$, there is no uniform lower bound on the Ricci curvature.
All codimension $2$ degenerations are stable, and although $\diam(F_t)\searrow0$, there is no uniform lower bound on the Ricci curvature; in fact, $\scal\searrow-\infty$.
Finally, codimension $3$ degenerations may or may not have unbounded Morse index, depending on how the $t_i$'s go to zero.

Infinitely many bifurcations due to unboundedness of the Morse index are only known to occur accompanied by collapse of codimension $\geq2$, cf.~Proposition~\ref{prop:cp2n1};
and Propositions~\ref{prop:collapsebif} and \ref{prop:collapsebif2} provide further evidence that this should always be the case.
It would be interesting to confirm this, 
that is, show that if a
family of Riemannian submersions $\pi_t\colon (M,\g(t))\to B$ with totally geodesic fibers and $\scal(\g(t))$ constant for all $t\in (0,1]$ satisfies $i_{\rm{Morse}}(\g(t))\nearrow\infty$ and the
Gromov--Hausdorff limit of $(M,\g(t))$ as $t\searrow0$ exists and has finite diameter, then its Hausdorff dimension must be $\leq\dim M-2$.

\appendix
\section{First eigenvalue and Yamabe stability in the remaining homogeneous CROSSes}
\label{appendix}

For the convenience of the reader, we now provide formulae (with references) for the first eigenvalue $\lambda_1(M,\g)$ of the Laplacian on all CROSSes $M$, endowed with a homogeneous $\G$-invariant metric $\g$, as presented in Table~\ref{tab:eigenvalues} below.

The (complete) spectrum of a CROSS, endowed with its canonical \emph{symmetric space} metric, can be found in \cite[p.~202]{Besse-closed}. Detailed spectral computations for $\Ss^n$, $\R P^n$, and $\C P^n$ are given in~\cite{BerGauMaz}; for $\Hr P^n$ and $\Ca P^2$, see \cite{cahn-wolf}. Regarding the remaining homogeneous metrics, we have that:
\begin{enumerate}[(i)]
\item The first eigenvalue of $\mathbf g(t)$ on $\Ss^{2n+1}$ is computed in~\cite{tanno1}, and an inspection of which eigenfunctions are $\Z_2$-invariant yields its first eigenvalue on $\R P^{2n+1}$;
\item The first eigenvalue of $\mathbf h(t_1,t_2,t_3)$ on $\Ss^3$ and $\R P^3$ are computed in \cite{Lauret-SpecSU(2)}, and the special cases where two of $t_1,t_2,t_3$ coincide done previously in \cite{urakawa79};
\item The first eigenvalue of $\mathbf h(t_1,t_2,t_3)$ on $\Ss^{4n+3}$ and $\R P^{4n+3}$ are computed in Theorem~\ref{thm:A}, and the special case $t_1=t_2=t_3$ done previously in \cite{tanno2};
\item The first eigenvalue of $\mathbf k(t)$ on $\Ss^{15}$ is computed in \cite[Prop.~7.3]{bp-calcvar}, and an inspection of which eigenfunctions are $\Z_2$-invariant yields its first eigenvalue on $\R P^{15}$;
\item The first eigenvalue of $\check{\mathbf h}(t)$ on $\C P^{2n+1}$ is computed in Theorem~\ref{thm:B}.
\end{enumerate}
As an alternative reference for (i) and the special case $t_1=t_2=t_3$ in (iii) one may use, respectively, \cite[Prop.~5.3 and 6.3]{bp-calcvar}. These homogeneous metrics, together with those in (iv), account for all isometry classes of distance spheres in rank one symmetric space. A unified and explicit description of their \emph{full} spectrum was recently obtained by the authors~\cite[Thm.~A]{blp-new}.

The above computations are carried out in one of two possible ways. The first, and more general, is the Lie-theoretic approach described in Section~\ref{sec:homspectra}, which is used in (ii) and (iii), and generalizes the classical approach developed for canonical symmetric space metrics (see e.g.~\cite{wallach-book,urakawa}).
The second, which relies on the existence of Riemannian submersions with minimal fibers, is explained in detail in \cite{Berard-BergeryBourguignon82} and \cite{besson-bordoni}, building on the earlier works~\cite{urakawa79,tanno1,tanno2}, and is used in (i), in the special case $t_1=t_2=t_3$ in (iii), as well as in (iv) and (v).

\medskip

We also include in Table~\ref{tab:eigenvalues} formulae for the scalar curvature of these CROSSes. The computation for the symmetric space metric on $\Ss^n$, $\R P^n$, $\C P^n$, $\Hr P^n$, and $\Ca P^2$
follows from the computation of their Einstein constants, which are, respectively, $n-1$, $n-1$, $2(n+1)$, $4(n+2)$, and $36$, under the normalization convention that these metrics have $\sec=1$ for $\Ss^n$ and $\R P^n$, and $1\leq\sec\leq 4$ in the remaining cases.
The computation for the other homogeneous metrics uses the Gray--O'Neill formula~\cite[Prop.~9.70]{Besse}, see also \eqref{eq:scal4n3} and \cite[Prop.~4.2]{bp-calcvar}.
In Table~\ref{tab:stable},
by solving the inequality 
\begin{equation}\label{eq:Yamstability}
\scal(M,\g)<(\dim M-1)\lambda_1(M,\g),
\end{equation}
we present the range of parameters for which these metrics are stable solutions to the Yamabe problem. If equality holds in \eqref{eq:Yamstability},  $\g$ is labeled as \emph{degenerate stable}.

\begin{remark}\label{rem:cahn-wolf}
For the convenience of the reader, we also identify some small imprecisions and typos in the literature.
First, the multiplicity of the $k$th eigenvalue of the round sphere, $\lambda_k(\Ss^d,\gr)=k(k+d-1)$, is given by \eqref{eq:multsphere}. Unfortunately, this formula appears with (the same) typos in \cite[p.~162]{BerGauMaz} and \cite[p.~35]{chavel}.

Second, the computation of some heat invariants of $\Ca P^2$ carried out in \cite{cahn-wolf} is incorrect. 
For instance, the ratio $a_1/a_0$ of the first two heat invariants, which is equal to $\frac{\scal}{6}$, evaluates to a \emph{negative} number according to the formulae in \cite[\S 13]{cahn-wolf}.
The correct values for these invariants are given in~\cite[Thm.~2.1]{awonusika}.
More precisely, in the notation of \cite[\S12]{cahn-wolf}, the values of $\eta_j$ are correct, except for $\eta_6=-{175}/{4}$, $\eta_3=2864323/256$, and $\eta_1=18445239/4096$. 
Furthermore, the second row of $\zeta_{P^2(\textrm{Cay})}$ in \cite[p.~20]{cahn-wolf} should be replaced with
\begin{equation*}
\zeta_{P^2(\textrm{Cay})} (t) 
= \frac{3!}{7! 11!} e^{(121/72)t} \sum_{j=0}^7 \eta_j\, (-1)^j\, g^{(j)}(\tfrac{t}{18}) + O (1), 
\end{equation*}
which gives, for any $0\leq m\leq 7$,
\begin{equation*}
a_m=
 \frac{3!}{7!11!} (4\pi)^8 \sum_{k=0}^m \left(\frac{121}{72}\right)^{\! k} {\eta_{7-m+k}}\, \frac{(7-m+k)!}{k!} \,  18^{8+k-m}.
\end{equation*}
Using the above, one obtains the correct value $a_1/a_0=4/3$, according to the normalization used in \cite{cahn-wolf}, for which the scalar curvature of $\Ca P^2$ is $\scal=8$.
\end{remark}

\vfill

\pagebreak[4]
  \global\pdfpageattr\expandafter{\the\pdfpageattr/Rotate 90}
  
\begin{landscape}
\vspace{-0.2cm}
\begin{table}[ht]
\begin{tabular}{|c|c|c|l|l|l|}
\hline
$M$ \rule[-1.2ex]{0pt}{0pt} \rule{0pt}{2.5ex}  & $\G$ & $\g$ & $\lambda_1(M,\g)$ & $\scal(M,\g)$ & $\Vol(M,\g)$ \\
\hline \noalign{\medskip} \hline 
$\Ss^n$\rule[-1.2ex]{0pt}{0pt} \rule{0pt}{4ex} & $\mathsf O(n+1)$ & $\gr$ & $n$ & $n(n-1)$ & $\dfrac{2\pi^{(n+1)/2}}{\Gamma(\frac{n+1}2)}$ \\[10pt]
$\Ss^{2n+1}$\rule[-1.2ex]{0pt}{0pt} \rule{0pt}{3.5ex} & $\SU(n+1)$ & $\mathbf g(t)$ & $\min\!\left\{ 2n+\frac{1}{t^2},\; 4(n+1)\right\}$ & $2n(2n+2-t^2)$ & $\dfrac{2\pi^{n+1}}{n!} t$ \\[12pt]
$\Ss^{4n+3}$\rule[-1.2ex]{0pt}{0pt} \rule{0pt}{3.5ex} & $\Sp(n+1)$ & $\mathbf h(t_1,t_2,t_3)$ & 
\!\!\begin{tabular}{ll}
\!$\min\!\Big\{ 4n +\frac{1}{t_1^2}+\frac{1}{t_2^2}+\frac{1}{t_3^2},$\\[2pt]
 $\phantom{\min\Big\{} 8n+ \frac{4}{t_2^2} + \frac{4}{t_3^2},\;  8(n+1) \Big\}$ \end{tabular}
 &\!\!\!\begin{tabular}{ll}
$4\left(\frac{1}{t_1^2}+\frac{1}{t_2^2}+\frac{1}{t_3^2}\right) - 2\left(\frac{t_1^2}{t_2^2t_3^2} +\frac{t_2^2}{t_1^2t_3^2}+\frac{t_3^2}{t_1^2t_2^2}\right)$\\[6pt]
 $-4n\left(t_1^2 + t_2^2 + t_3^2\right) +16n(n+2)$ \end{tabular}  & $\dfrac{2\pi^{2n+2}}{(2n+1)!} t_1t_2t_3$ \\[14pt]
$\Ss^{3}$\rule[-1.2ex]{0pt}{0pt} \rule{0pt}{3.5ex} & $\SU(2)$ & $\mathbf h(t_1,t_2,t_3)$ & $\min\!\left\{\frac{1}{t_1^2}+\frac{1}{t_2^2}+\frac{1}{t_3^2},\;  \frac{4}{t_2^2} + \frac{4}{t_3^2}  \right\}$ & $4\left(\frac{1}{t_1^2}+\frac{1}{t_2^2}+\frac{1}{t_3^2}\right)-2\left(\frac{t_1^2}{t_2^2t_3^2}+\frac{t_2^2}{t_1^2t_3^2}+\frac{t_3^2}{t_1^2t_2^2}\right) $ & $2\pi^2 t_1t_2t_3$\\[10pt]
$\Ss^{15}$\rule[-2.5ex]{0pt}{0pt} \rule{0pt}{3.5ex} & $\Spin(9)$ & $\mathbf k(t)$ & $\min\!\left\{ 8+\frac{7}{t^2}, \; 32\right\}$ & $14\left(\frac{3}{t^2}+16-4t^2\right)$ & $\dfrac{2\pi^8}{7!} t^7$ \\
\hline \noalign{\medskip} \hline 
$\C P^n$\rule[-2ex]{0pt}{0pt} \rule{0pt}{4ex} & $\SU(n+1)$ & $\gFS$ & $4(n+1)$ & $4n(n+1)$ & $\dfrac{\pi^n}{n!}$ \\[8pt]
$\C P^{2n+1}$\rule[-2.6ex]{0pt}{0pt} \rule{0pt}{3.5ex} & $\Sp(n+1)$ & $\check{\mathbf h}(t)$ & $ \min\!\left\{   8n +\frac{8}{t^2},\; 8(n+1)\right\}$ & $\frac{8}{t^2} + 16n(n+2) -8nt^2$ & $\dfrac{\pi^{2n+1}}{(2n+1)!} t^2$ \\
\hline \noalign{\medskip} \hline 
$\Hr P^n$\rule[-2.6ex]{0pt}{0pt} \rule{0pt}{4ex} & $\Sp(n+1)$ & $\gFS$ & $8(n+1)$ & $16n(n+2)$ & $\dfrac{\pi^{2n}}{(2n+1)!}$ \\
\hline \noalign{\medskip} \hline 
$\Ca P^2$\rule[-2ex]{0pt}{0pt} \rule{0pt}{4ex} & $\mathsf F_4$ & $\gFS$ & $48$ & $576$ & $\dfrac{2^{49} 3^{33} \pi^8 }{11!}$ \\
\hline
\end{tabular}
\bigskip
\caption{First eigenvalue of the Laplacian, scalar curvature, and volume of homogeneous metrics on simply-connected CROSSes. In the above, we convention that $n\geq1$, $0<t_1\leq t_2\leq t_3$, the round metric $\gr$ has sectional curvatures $\sec\equiv1$ and the Fubini--Study metrics $\gFS$ have $1\leq\sec\leq4$. References are given in the previous pages.
}\label{tab:eigenvalues}
\end{table}

\newpage

\begin{table}[ht]
\begin{tabular}{|c|c|c|l|l|l|}
\hline
$M$ \rule[-1.2ex]{0pt}{0pt} \rule{0pt}{2.5ex}  & $\G$ & $\g$ & $\lambda_1(M,\g)$ & $\scal(M,\g)$ & $\Vol(M,\g)$ \\
\hline \noalign{\medskip} \hline 
$\R P^n$\rule[-1.2ex]{0pt}{0pt} \rule{0pt}{4ex} & $\mathsf O(n+1)$ & $\gr$ & $2(n+1)$ & $n(n-1)$ & $\dfrac{\pi^{(n+1)/2}}{\Gamma(\frac{n+1}2)}$ \\[14pt]
$\R P^{2n+1}$\rule[-1.2ex]{0pt}{0pt} \rule{0pt}{3.5ex} & $\SU(n+1)$ & $\mathbf g(t)$ & $\min\!\left\{ 4n+\tfrac{4}{t^2}, \; 4(n+1)\right\}$ & $2n(2n+2-t^2)$ & $\dfrac{\pi^{n+1}}{n!} t$ \\[12pt]
$\R P^{4n+3}$\rule[-1.2ex]{0pt}{0pt} \rule{0pt}{3.5ex} & $\Sp(n+1)$ & $\mathbf h(t_1,t_2,t_3)$ & 
$\min\!\Big\{ 8n+ \frac{4}{t_2^2} + \frac{4}{t_3^2},\;  8(n+1) \Big\}$  &\!\!\!\begin{tabular}{ll}
$4\left(\frac{1}{t_1^2}+\frac{1}{t_2^2}+\frac{1}{t_3^2}\right) - 2\left(\frac{t_1^2}{t_2^2t_3^2} +\frac{t_2^2}{t_1^2t_3^2}+\frac{t_3^2}{t_1^2t_2^2}\right)$\\[6pt]
 $-4n\left(t_1^2 + t_2^2 + t_3^2\right) +16n(n+2)$ \end{tabular}  & $\dfrac{\pi^{2n+2}}{(2n+1)!} t_1t_2t_3$ \\[16pt]
$\R P^{3}$\rule[-1.2ex]{0pt}{0pt} \rule{0pt}{3.5ex} & $\SU(2)$ & $\mathbf h(t_1,t_2,t_3)$ & $\frac{4}{t_2^2} + \frac{4}{t_3^2}$ & $4\left(\frac{1}{t_1^2}+\frac{1}{t_2^2}+\frac{1}{t_3^2}\right)-2\left(\frac{t_1^2}{t_2^2t_3^2}+\frac{t_2^2}{t_1^2t_3^2}+\frac{t_3^2}{t_1^2t_2^2}\right) $ & $\pi^2 t_1t_2t_3$\\[12pt]
$\R P^{15}$\rule[-2.5ex]{0pt}{0pt} \rule{0pt}{3.5ex} & $\Spin(9)$ & $\mathbf k(t)$ & $\min\!\left\{16+\tfrac{16}{t^2}, \; 32 \right\}$ & $14\left(\frac{3}{t^2}+16-4t^2\right)$ & $\dfrac{\pi^8}{7!} t^7$ \\
\hline 
\end{tabular}
\bigskip
\caption{First eigenvalue of the Laplacian, scalar curvature, and volume of homogeneous metrics on the non-simply-connected CROSSes. In the above, we convention that $n\geq1$, $0<t_1\leq t_2\leq t_3$, and the metric $\gr$ has sectional curvatures $\sec\equiv1$. References are given in the previous pages.
}\label{tab:eigenvaluesRPd}
\end{table}
\end{landscape}

\pagebreak[4]
  \global\pdfpageattr\expandafter{\the\pdfpageattr/Rotate 0}

\vspace{-0.2cm}
\begin{table}[!ht]
\begin{tabular}{|c|c|l|}
\hline
$M$ & $\g$\rule[-2.6ex]{0pt}{0pt} \rule{0pt}{4ex} & 
\!\!\!\begin{tabular}{l}
Stability as solution\\to  the Yamabe problem
\end{tabular}
\\
\hline \noalign{\medskip} \hline 
$\Ss^n$\rule[-1.2ex]{0pt}{0pt} \rule{0pt}{2.5ex} & $\gr$ & degenerate stable   \\
& &  \\
$\Ss^{2n+1}$\rule[-1.2ex]{0pt}{0pt} \rule{0pt}{3.5ex} & $\mathbf g(t)$ & $t\neq 1$ \\
& &  \\
$\Ss^{4n+3}$\rule[-1.2ex]{0pt}{0pt} \rule{0pt}{3.5ex} & $\mathbf h(t_1,t_2,t_3)$ & \!\!\!\begin{tabular}{ll}
$\big(2n(t_1^2+t_2^2+t_3^2)+8(n^2+n+1)\big)(t_1t_2t_3)^2$\\[3pt] $\quad +t_1^4+t_2^4+t_3^4>2(t_1^2t_2^2+t_1^2t_3^2+t_2^2t_3^2),$ and \\[3pt]
 $(t_1,t_2,t_3)\neq(1,1,1)$
\end{tabular}
 \\
& & \\
$\Ss^{3}$\rule[-1.2ex]{0pt}{0pt} \rule{0pt}{3.5ex} & $\mathbf h(t_1,t_2,t_3)$ & $(t_1,t_2,t_3)\neq(1,1,1)$   \\
& & \\
$\Ss^{15}$\rule[-2ex]{0pt}{0pt} \rule{0pt}{3.5ex} & $\mathbf k(t)$ & $t>\sqrt{\frac12(\sqrt{19}-4)}\cong 0.4236$, and $t\neq1$\\
\hline \noalign{\medskip} \hline 
$\R P^n$\rule[-1.2ex]{0pt}{0pt} \rule{0pt}{2.5ex} & $\gr$ & stable   \\
& &  \\
$\R P^{2n+1}$\rule[-1.2ex]{0pt}{0pt} \rule{0pt}{3.5ex} & $\mathbf g(t)$ & stable \\
& &  \\
$\R P^{4n+3}$\rule[-1.2ex]{0pt}{0pt} \rule{0pt}{3.5ex} & $\mathbf h(t_1,t_2,t_3)$ & \!\!\!\begin{tabular}{ll}
$\big(2n(t_1^2+t_2^2+t_3^2)+8(n^2+n+1)\big)(t_1t_2t_3)^2$\\[3pt] $\quad +t_1^4+t_2^4+t_3^4>2(t_1^2t_2^2+t_1^2t_3^2+t_2^2t_3^2)$
\end{tabular}
 \\
& & \\
$\R P^{3}$\rule[-1.2ex]{0pt}{0pt} \rule{0pt}{3.5ex} & $\mathbf h(t_1,t_2,t_3)$ & stable   \\
& & \\
$\R P^{15}$\rule[-2ex]{0pt}{0pt} \rule{0pt}{3.5ex} & $\mathbf k(t)$ & $t>\sqrt{\frac12(\sqrt{19}-4)}\cong 0.4236$\\
\hline \noalign{\medskip} \hline 
$\C P^n$\rule[-1.2ex]{0pt}{0pt} \rule{0pt}{2.5ex} & $\gFS$ & stable if $n\geq2$, degenerate stable if $n=1$  \\
& & \\
$\C P^{2n+1}$\rule[-2.6ex]{0pt}{0pt} \rule{0pt}{3.5ex} & $\check{\mathbf h}(t)$ & $t>\displaystyle \sqrt{\frac{\sqrt{(2n^2+n+1)^2+4n}-(2n^2+n+1)}{2n}}$ \\
\hline \noalign{\medskip} \hline 
$\Hr P^n$\rule[-1.2ex]{0pt}{0pt} \rule{0pt}{2.5ex} & $\gFS$ & stable if $n\geq2$, degenerate stable if $n=1$ \\
\hline \noalign{\medskip} \hline 
$\Ca P^2$\rule[-1.2ex]{0pt}{0pt} \rule{0pt}{2.5ex} & $\gFS$ & stable  \\
\hline
\end{tabular}
\bigskip
\caption{Classification of homogeneous metrics on CROSSes that are stable solutions to the Yamabe problem, with same conventions as in Table~\ref{tab:eigenvalues}. Metrics are labeled \emph{degenerate stable} if their Jacobi operator \eqref{eq:jacobi-operator} is positive-semidefinite with nontrivial kernel.}
\label{tab:stable}
\end{table}


\begin{thebibliography}{EGSC18}

\bibitem[AB15]{mybook}
{\sc M.~M. Alexandrino and R.~G. Bettiol}.
\newblock {\em {L}ie {G}roups and {G}eometric {A}spects of {I}sometric
  {A}ctions}.
\newblock Springer, 2015.

\bibitem[AYY13]{AnYuYu13}
{\sc J.~An, J.-K. Yu, and J.~Yu}.
\newblock {\em On the dimension datum of a subgroup and its application to
	isospectral manifolds}.
\newblock J. Differential Geom., 94 (2013), 59--85.

\bibitem[Aub98]{aubin-book}
{\sc T.~Aubin}.
\newblock {\em Some nonlinear problems in {R}iemannian geometry}.
\newblock Springer Monographs in Mathematics. Springer-Verlag, Berlin, 1998.

\bibitem[Awo19]{awonusika}
{\sc R.~O. Awonusika}.
\newblock {\em Functional determinant of {L}aplacian on {C}ayley projective
  plane {${P}^2$}({C}ay)}.
\newblock Proc. Indian Acad. Sci. Math. Sci., 129 (2019), Art. 48, 28.

\bibitem[BBB82]{Berard-BergeryBourguignon82}
{\sc L.~B\'{e}rard-Bergery and J.-P. Bourguignon}.
\newblock {\em Laplacians and {R}iemannian submersions with totally geodesic
  fibres}.
\newblock Illinois J. Math., 26 (1982), 181--200.

\bibitem[BBG85]{bbg}
{\sc P.~B\'{e}rard, G.~Besson, and S.~Gallot}.
\newblock {\em Sur une in\'{e}galit\'{e} isop\'{e}rim\'{e}trique qui
  g\'{e}n\'{e}ralise celle de {P}aul {L}\'{e}vy-{G}romov}.
\newblock Invent. Math., 80 (1985), 295--308.

\bibitem[BGM71]{BerGauMaz}
{\sc M.~Berger, P.~Gauduchon, and E.~Mazet}.
\newblock {\em Le spectre d'une vari\'et\'e riemannienne}.
\newblock Lecture Notes in Mathematics, Vol. 194. Springer-Verlag, Berlin-New York, 1971.

\bibitem[Bes78]{Besse-closed}
{\sc A.~L. Besse}.
\newblock {\em Manifolds all of whose geodesics are closed}, vol.~93 of
  Ergebnisse der Mathematik und ihrer Grenzgebiete [Results in Mathematics and
  Related Areas].
\newblock Springer-Verlag, Berlin-New York, 1978.

\bibitem[Bes08]{Besse}
{\sc A.~L. Besse}.
\newblock {\em Einstein manifolds}.
\newblock Classics in Mathematics. Springer-Verlag, Berlin, 2008.
\newblock Reprint of the 1987 edition.

\bibitem[BB90]{besson-bordoni}
{\sc G.~Besson and M.~Bordoni}.
\newblock {\em On the spectrum of {R}iemannian submersions with totally geodesic fibers}.
\newblock Atti Accad. Naz. Lincei Cl. Sci. Fis. Mat. Natur. Rend. Lincei (9)
Mat. Appl., 1 (1990), 335--340.

\bibitem[BLP]{blp-new}
{\sc R.~G. Bettiol, E.~A.~Lauret, and P.~Piccione}.
\newblock {\em Full {L}aplace spectrum of distance spheres in symmetric spaces of rank one}.
\newblock Preprint.
\newblock
  \htmladdnormallink{arXiv:2012.02349}{http://arxiv.org/abs/2012.02349}.

\bibitem[BP13a]{bp-calcvar}
{\sc R.~G. Bettiol and P.~Piccione}.
\newblock {\em Bifurcation and local rigidity of homogeneous solutions to the
  {Y}amabe problem on spheres}.
\newblock Calc. Var. Partial Differential Equations, 47 (2013), 789--807.

\bibitem[BP13b]{bp-pacific}
{\sc R.~G. Bettiol and P.~Piccione}.
\newblock {\em Multiplicity of solutions to the {Y}amabe problem on collapsing
  {R}iemannian submersions}.
\newblock Pacific J. Math., 266 (2013), 1--21.

\bibitem[BP18]{frankenstein}
{\sc R.~G. Bettiol and P.~Piccione}.
\newblock {\em Infinitely many solutions to the {Y}amabe problem on noncompact
  manifolds}.
\newblock Ann. Inst. Fourier (Grenoble), 68 (2018), 589--609.

\bibitem[CW76]{cahn-wolf}
{\sc R.~S. Cahn and J.~A. Wolf}.
\newblock {\em Zeta functions and their asymptotic expansions for compact symmetric spaces of rank one}.
\newblock Comment. Math. Helv., 51 (1976), 1--21.

\bibitem[Cha84]{chavel}
{\sc I.~Chavel}.
\newblock {\em Eigenvalues in {R}iemannian geometry}, vol.~115 of Pure and
  Applied Mathematics.
\newblock Academic Press, Inc., Orlando, FL, 1984.

\bibitem[dLPZ12]{LPZ12}
{\sc L.~L. de~Lima, P.~Piccione, and M.~Zedda}.
\newblock {\em On bifurcation of solutions of the {Y}amabe problem in product
  manifolds}.
\newblock Ann. Inst. H. Poincar\'e Anal. Non Lin\'eaire, 29 (2012), 261--277.

\bibitem[EGSC18]{gordinagafa}
{\sc N.~Eldredge, M.~Gordina, and L.~Saloff-Coste}.
\newblock {\em Left-invariant geometries on {$\rm SU(2)$} are uniformly
  doubling}.
\newblock Geom. Funct. Anal., 28 (2018), 1321--1367.

\bibitem[Fuk87]{fukaya}
{\sc K.~Fukaya}.
\newblock {\em Collapsing of {R}iemannian manifolds and eigenvalues of
  {L}aplace operator}.
\newblock Invent. Math., 87 (1987), 517--547.

\bibitem[GSS10]{GordonSchuethSutton10}
{\sc C.~S. Gordon, D.~Schueth, and C.~J. Sutton}.
\newblock {\em Spectral isolation of bi-invariant metrics on compact {L}ie
	groups}.
\newblock Ann. Inst. Fourier (Grenoble), 60 (2010), 1617--1628.

\bibitem[GS10]{GordonSutton10}
{\sc C.~S. Gordon and C.~J. Sutton}.
\newblock {\em Spectral isolation of naturally reductive metrics on simple
  {L}ie groups}.
\newblock Math. Z., 266 (2010), 979--995.

\bibitem[Hal15]{hall-book}
{\sc B.~Hall}.
\newblock {\em Lie groups, {L}ie algebras, and representations}, vol.~222 of
  Graduate Texts in Mathematics.
\newblock Springer, Cham, second edition, 2015.
\newblock An elementary introduction.

\bibitem[Kna02]{Knapp-book-beyond}
{\sc A.~W. Knapp}.
\newblock {\em Lie groups beyond an introduction}, vol.~140 of Progress in
  Mathematics.
\newblock Birkh\"{a}user Boston, Inc., Boston, MA, second edition, 2002.

\bibitem[Lau19a]{Lauret-SpecSU(2)}
{\sc E.~A. Lauret}.
\newblock {\em The smallest {L}aplace eigenvalue of homogeneous 3-spheres}.
\newblock Bull. Lond. Math. Soc., 51 (2019), 49--69.

\bibitem[Lau19b]{Lauret18}
{\sc E.~A. Lauret}.
\newblock {\em {S}pectral uniqueness of bi-invariant metrics on symplectic
  groups}.
\newblock Transform. Groups, 24 (2019), 1157--1164.

\bibitem[LP87]{lee-parker}
{\sc J.~M. Lee and T.~H. Parker}.
\newblock {\em The {Y}amabe problem}.
\newblock Bull. Amer. Math. Soc. (N.S.), 17 (1987), 37--91.

\bibitem[LSS]{LSS1}
{\sc S.~Lin, B.~Schmidt, and C.~J. Sutton}.
\newblock {\em Geometric structures and the {L}aplace spectrum}.
\newblock Preprint.
\newblock
  \htmladdnormallink{arXiv:1905.11454v1} {http://arxiv.org/abs/1905.11454v1}.

\bibitem[LSS21]{LSS2}
{\sc S.~Lin, B.~Schmidt, and C.~J. Sutton}.
\newblock {\em Geometric structures and the {L}aplace spectrum, {P}art {II}}.
\newblock Trans. Amer. Math. Soc., to appear. 
\newblock
\htmladdnormallink{DOI: 10.1090/tran/8417} {https://doi.org/10.1090/tran/8417}.

\bibitem[MU80]{MutoUrakawa80}
{\sc H.~Mut\^{o} and H.~Urakawa}.
\newblock {\em On the least positive eigenvalue of {L}aplacian for compact
  homogeneous spaces}.
\newblock Osaka J. Math., 17 (1980), 471--484.

\bibitem[OP20]{OtobaPetean1}
{\sc N.~Otoba and J.~Petean}.
\newblock {\em Bifurcation for the constant scalar curvature equation and harmonic {R}iemannian submersions}.
\newblock J. Geom. Anal., 30 (2020), 4453--4463.

\bibitem[Pet16]{petersen-3ed}
{\sc P.~Petersen}.
\newblock {\em Riemannian geometry}, vol.~171 of Graduate Texts in Mathematics.
\newblock Springer, Cham, 3rd edition, 2016.

\bibitem[Pro78]{procesi}
{\sc C.~Procesi}.
\newblock {\em Positive symmetric functions}.
\newblock Adv. in Math., 29 (1978), 219--225.

\bibitem[Pro05]{Proctor05}
{\sc E.~Proctor}.
\newblock {\em Isospectral metrics and potentials on classical compact simple
  {L}ie groups}.
\newblock Michigan Math. J., 53 (2005), 305--318.

\bibitem[Sch01]{Schueth01}
{\sc D.~Schueth}.
\newblock {\em Isospectral manifolds with different local geometries}.
\newblock J. Reine Angew. Math., 534 (2001), 41--94.
  
\bibitem[Sha01]{ravi}
{\sc K.~Shankar}.
\newblock {\em Isometry groups of homogeneous spaces with positive sectional curvature}.
\newblock Differential Geom. Appl., 14 (2001), 57--78.

\bibitem[Sut02]{Sutton02}
{\sc C.~J. Sutton}.
\newblock {\em Isospectral simply-connected homogeneous spaces and the spectral rigidity of group actions}.
\newblock Comment. Math. Helv., 77 (2002), 701--717.

\bibitem[Sut20]{Sutton16pp}
{\sc C.~J. Sutton}.
\newblock {\em On the {P}oisson relation for compact {L}ie groups}.
\newblock Ann. Global Anal. Geom., 57 (2020), 537--589. 

\bibitem[Tan79]{tanno1}
{\sc S.~Tanno}.
\newblock {\em The first eigenvalue of the {L}aplacian on spheres}.
\newblock T\^{o}hoku Math. J. (2), 31 (1979), 179--185.

\bibitem[Tan80]{tanno2}
{\sc S.~Tanno}.
\newblock {\em Some metrics on a {$(4r+3)$}-sphere and spectra}.
\newblock Tsukuba J. Math., 4 (1980), 99--105.

\bibitem[Ura79]{urakawa79}
{\sc H.~Urakawa}.
\newblock {\em On the least positive eigenvalue of the {L}aplacian for compact
  group manifolds}.
\newblock J. Math. Soc. Japan, 31 (1979), 209--226.

\bibitem[Ura86]{urakawa}
{\sc H.~Urakawa}.
\newblock {\em The first eigenvalue of the {L}aplacian for a positively curved
  homogeneous {R}iemannian manifold}.
\newblock Compositio Math., 59 (1986), 57--71.

\bibitem[Var04]{varga}
{\sc R.~S.~Varga}.
\newblock {\em Ger\v{s}gorin and his circles}.
\newblock Springer Series in Computational Mathematics, 36. Springer-Verlag, Berlin, 2004.

\bibitem[Wal73]{wallach-book}
{\sc N.~R. Wallach}.
\newblock {\em Harmonic analysis on homogeneous spaces}.
\newblock Marcel Dekker, Inc., New York, 1973.
\newblock Pure and Applied Mathematics, No. 19.

\bibitem[WY09]{WallachYacobi09}
{\sc N.~Wallach and O.~Yacobi}.
\newblock A multiplicity formula for tensor products of {${\rm SL}_2$} modules
  and an explicit {${\rm Sp}_{2n}$} to {${\rm Sp}_{2n-2}\times{\rm Sp}_2$}
  branching formula.
\newblock In {\em Symmetry in mathematics and physics}, vol.~490 of Contemp.
  Math., 151--155. AMS, Providence, RI, 2009.

\bibitem[Yu15]{Yu15}
{\sc J.~Yu}.
\newblock {\em A compactness result for dimension datum}.
\newblock Int. Math. Res. Not. IMRN,  (2015), 9438--9449.

\bibitem[Yu]{Yu18pp}
{\sc J.~Yu}.
\newblock {\em On the dimension datum of a subgroup. {II}}.
\newblock Preprint.
\newblock
\htmladdnormallink{arXiv:1803.06210}{http://arxiv.org/abs/1803.06210}.

\bibitem[Zil82]{Ziller82}
{\sc W.~Ziller}.
\newblock {\em Homogeneous {E}instein metrics on spheres and projective
  spaces}.
\newblock Math. Ann., 259 (1982), 351--358.

\end{thebibliography}
\end{document}